\documentclass[preprint,12pt]{elsarticle}

\usepackage{amssymb}
\usepackage{amsthm}

\usepackage{amsmath}
\usepackage{graphicx}
\usepackage{accents}
\usepackage{enumerate}
\usepackage{xcolor}
\usepackage[pdfborder=0,colorlinks,hyperindex,pdfpagelabels]{hyperref}

\usepackage[margin=2.5cm]{geometry}

\def\R{{\mathbb R}}

\def\D{\Delta}
\def\g{\gamma}
\def\s{\sigma}
\def\t{\theta}
\def\l{\lambda}
\def\p{\partial}
\def\O{\Omega}
\def\o{\omega}
\def\e{\varepsilon}
\def\v{\varphi}
\def\G{\Gamma}

\def\mc{\mathcal}
\def\mf{\mathfrak}
\def\ua{\uparrow}

\def\s{\sigma}
\def\t{\theta}
\def\da{\downarrow}

\newtheorem{theorem}{Theorem}[section]
\newtheorem{lemma}[theorem]{Lemma}

\newtheorem{corollary}[theorem]{Corollary} 
\newtheorem{proposition}[theorem]{Proposition} 

\theoremstyle{definition}

\theoremstyle{remark}
\newtheorem{remark}[theorem]{Remark}

\numberwithin{equation}{section}

\journal{\hspace{-3.4cm}{\color{white}\rule[-0.5cm]{3.5cm}{1cm}}}

\begin{document}

\begin{frontmatter}



\title{The Picone identity: A device to get optimal uniqueness results and global dynamics in Population Dynamics\tnoteref{t1}}
\tnotetext[t1]{Partially supported by the Ministry of Science, Innovation and Universities of Spain under Research Grant PGC2018-097104-B-100, by the Institute of Interdisciplinary Mathematics (IMI) of Complutense University, and by the Ministry of Education and Culture of Spain
under Fellowship Grant FPU15/04755.}


\author{Sergio Fern\'andez-Rinc\'on\corref{cor1}}
\ead{sergfern@ucm.es}
\address{Institute of Interdisciplinary Mathematics, Department of Mathematical Analysis and Applied Mathematics, Complutense University, Madrid 28040, Spain}
\cortext[cor1]{Corresponding author}

\author{Juli\'an L\'opez-G\'omez}
\ead{julian@mat.ucm.es}
\address{Institute of Interdisciplinary Mathematics, Department of Mathematical Analysis and Applied Mathematics, Complutense University, Madrid 28040, Spain}

\begin{abstract}
This paper infers from a generalized Picone identity the uniqueness of the stable positive solution for a class of semilinear equations of superlinear indefinite type, as well as the uniqueness and global attractivity of the coexistence state  in two generalized diffusive prototypes of the symbiotic and competing species models of Lotka--Volterra. The optimality of these uniqueness theorems reveals the tremendous strength of the Picone identity.
\end{abstract}

%

\begin{keyword}
Picone identity \sep superlinear indefinite problems \sep diffusive symbiotic models \sep diffusive competitive models \sep uniqueness of coexistence states \sep global dynamics 

\MSC[2010] Primary 35M12 \sep 35J61 \sep 35J57 \sep Secondary 35K57 \sep 35B09 \sep 35Q92 
\end{keyword}

\end{frontmatter}


\section{Introduction}
\label{sec-intro}

Mauro Picone was born in Palermo, Sicily, on the 2\textsuperscript{nd} of May of 1885 and passed away on
April 11\textsuperscript{th} 1977 at Rome. After completing his Degree at the \emph{Scuola Normale Superiore} in 1907, he remained in Pisa as an assistant of U. Dini until 1913. Then, he moved to the Technical University of Turin as an assistant of G. Fubini, where he stayed until he served as an officer in the World War I. It was during the period of the war when he realized the importance of Mathematics to solve practical problems of social relevance. Since then,  the development of constructive methods for solving Partial Differential Equations was central to Picone's vision of Applied Mathematics, though his own work in this field did not enjoy the influence that it certainly deserved on the merging field of modern numerical analysis \cite{BT}, which might be explained by Picone's membership in the Fascist Party and his involvement in the war efforts in the dark days of the World War II \cite{EKR}. Nevertheless, M. Picone contributed to the field of Ordinary Differential Equations with a celebrated identity, named after him as Picone's identity, \cite{Picone}. It is a classical variational identity which was developed, with a great success, to deliver a very short and elegant proof of the Sturm comparison theorem, and it was shown to be an extremely useful device in studying the oscillation of these equations.
\par
Theorem \ref{th2.1} of Section \ref{sec-2} delivers a (new) generalized identity of Picone type valid for arbitrary boundary conditions of mixed type, classical and non-classical, which generalizes, substantially, the previous ones of M. Picone \cite{Picone}, K. Kreith \cite{Kreith}, H. Berestycki, I. Capuzzo-Dolcetta \& L. Nirenberg \cite{BCDN}, J. L\'opez-G\'omez \cite{LG97} and S. Fern\'andez-Rinc\'on \& J. L\'opez-G\'omez \cite{FRLGNA}, as it works out under general boundary conditions of non-classical mixed type. The main goal of this paper is to infer from this generalized Picone identity the uniqueness of the stable positive solution for a class of semilinear equations of superlinear indefinite type, as well as the uniqueness and global attractivity of the coexistence state  in two generalized diffusive prototypes of the symbiotic and competing species models of Lotka--Volterra. The optimality of all these uniqueness results reveals the
tremendous strength of our generalized Picone identity.
\par
In particular, Section \ref{sec-3} analyzes the positive solutions of the superlinear indefinite problem
\begin{equation}
\label{1.1}
\left\{
\begin{array}{ll}
\mathcal{L} u=\lambda u-a(x)f(u) & \hbox{in}\;\;\Omega,\\
\mathcal{B}u=0 & \hbox{on}\;\;\partial\Omega,
\end{array}
\right.
\end{equation}
where $\Omega$ is an open bounded subset of $\mathbb{R}^N$, $N\geq 1$, of class $\mathcal{C}^2$, $a\in\mathcal{C}(\bar\Omega)$ changes of sign in $\O$, $\lambda\in\mathbb{R}$ is a parameter, $f\in\mathcal{C}(\mathbb{R})$, $f\neq 0$, satisfies $f(0)=0$, and
\[
  \mathcal{L}:=-\mathrm{div}(A\nabla\cdot)+C
\]
is an uniformly elliptic operator in divergence form with $A\in\mathcal{M}_N^\mathrm{sym}(\mathcal{C}^1(\bar\Omega))$ and $C\in\mathcal{C}(\bar\Omega)$.
The boundary operator $\mc{B}$ is of general non-classical type, as described in detail
in Section \ref{sec-3}.  The problem \eqref{1.1}  is a generalized version of the simple prototype analyzed by R. G\'omez-Re\~nasco and J. L\'opez-G\'omez
\cite{GRLG00}, \cite{GRLG01}, where $\mc{L}=-\D$, $\mc{B}=\mc{D}$ is the Dirichlet operator on $\p\O$, and
$f(u)=u^p$ for some $p\geq 2$. As a by-product of our generalized Picone identity, in the special case when
$f(u)=u^p$, we can extend the results of \cite{GRLG00} and \cite{GRLG01} characterizing whether, or not,  \eqref{1.1} admits a linearly stable positive solution, as well as establishing its uniqueness if it exists. This is a rather intriguing uniqueness result as it is folklore that some simple prototypes of \eqref{1.1} possess an arbitrarily large number of positive solutions for the appropriate parameter ranges (see R. G\'omez-Re\~nasco and J. L\'opez-G\'omez \cite{GRLG00}, J. L\'opez-G\'omez, M. Molina-Meyer and A. Tellini \cite{LGMMT} J. L\'opez-G\'omez, A. Tellini and F. Zanolin \cite{LGTZ}, J. L\'opez-G\'omez and A. Tellini \cite{LGT},  as well as the recent monograph of G. Feltrin \cite{Fel}). This striking uniqueness theorem relies on the fact that, for the special choice $f(u)=u^p$, $p\geq 2$,  any linearly neutrally stable positive solution of \eqref{1.1}  must be a quadratic subcritical turning point in the entire set of positive solutions, $(\l,u)$, of the problem \eqref{1.1}. One of the main novelties of this section, Theorem \ref{th3.7},
establishes that there are arbitrarily small perturbations of
the function $f(u)=u^p$ for which the previous uniqueness result fails to be true. Therefore, our extension
of the pioneering findings of \cite{GRLG00} and \cite{GRLG01} seems optimal. Section \ref{sec-4} analyzes the global structure of the set of positive solutions of \eqref{1.1} in the special case when $f(u)=u^p$, $p\geq 2$.
\par
In Section \ref{sec-5} we study the Lotka--Volterra symbiotic model
\begin{equation}
\label{1.2}
\left\{
\begin{array}{ll}
d_1\mathcal{L}_1 u=\lambda u-au^2+ b uv & \hbox{in}\;\;\Omega,\\
d_2\mathcal{L}_2 v=\mu v-dv^2+ c uv & \hbox{in}\;\;\Omega,\\
\mathcal{B}_1 u=\mathcal{B}_2 v=0 & \hbox{on}\;\;\partial\Omega,
\end{array}
\right.
\end{equation}
where $a$, $b$, $c$ and $d$ are continuous and positive functions,
$\mc{L}_1$ and $\mc{L}_2$ are second order elliptic operators of the form
\[
  \mathcal{L}_j:=-\mathrm{div}(A_j\nabla\cdot)+C_j,\qquad j=1, 2,
\]
much like the $\mc{L}$ above, and $\mc{B}_1$, $\mc{B}_2$ are
arbitrary boundary operators of mixed type, much like $\mc{B}$ above. As a rather direct consequence of our
generalized Picone identity,
setting
\begin{equation}
\label{1.3}
F_\pm(t):=\frac{1}{8}\left[ 27-18t - t^2\pm(9-t)^{3/2}(1-t)^{1/2}\right],\qquad
t \in[0,1],
\end{equation}
we were able to establish that if $\kappa:=\frac{bc}{ad}\lneq 1$ in $\Omega$ and
\begin{equation}
\label{1.4}
\max_{\bar\Omega} \left( \frac{a d^2}{c^3}F_-(\kappa)\right) \le \min_{\bar\Omega} \left( \frac{a d^2}{c^3}F_+(\kappa)\right),
\end{equation}
then any coexistence state of \eqref{1.2} is linearly stable, regarded as an steady-state solution of its
parabolic counterpart. When, in addition,
\begin{equation}
\label{1.5}
\max_{\bar\Omega}\left(\frac{c}{d}\right)<\min_{\bar\Omega}\left(\frac{a}{b}\right),
\end{equation}
then, the linear stability of the coexistence state entails its uniqueness by means of the fixed point index in cones.
\par
Obviously, \eqref{1.4} holds as soon as
all the functions coefficients $a$, $b$, $c$ and $d$ are assumed to be positive constants, as in the most classical Lotka--Volterra symbiotic model. Thus, since $F_-(t)$ is separated away from $F_+(t)$ when $t$ is away from $t=1$, all these function coefficients are allow to oscillate around arbitrary constants still respecting the estimate \eqref{1.4}, which at the end of the day measures how far away from these
constants can be $a$, $b$, $c$ and $d$ to guarantee the uniqueness of the coexistence state of \eqref{1.2}. Note that \eqref{1.5} entails $\max_{\bar\O}\kappa < 1$. This uniqueness result is completely new with respect to the findings of M. Delgado et al. in \cite{DLGS}.
\par
As a direct consequence of the uniqueness result of the coexistence state, the global dynamics of the parabolic counterpart of \eqref{1.2} can be easily characterized from the local attractive properties of the semitrivial positive solutions of \eqref{1.2}, those with one of their components vanishing. Thus, it mimics the dynamics of the non-spatial model where all diffusion coefficients are switched off to zero.
\par
Finally, in Section \ref{sec-6} we shortly adapt the theory of Section \ref{sec-5} to deal with the diffusive competing species model
\begin{equation}
\label{1.6}
\left\{
\begin{array}{ll}
d_1\mathcal{L}_1 u=\lambda u-au^2- b uv & \hbox{in}\;\;\Omega,\\
d_2\mathcal{L}_2 v=\mu v-dv^2- c uv & \hbox{in}\;\;\Omega,\\
\mathcal{B}_1 u=\mathcal{B}_2 v=0 & \hbox{on}\;\;\partial\Omega.
\end{array}
\right.
\end{equation}
Our new findings provide us with some sharp improvements of our previous results in \cite{FRLGNA}, which already deepened, very substantially, some previous findings of J. Furter and J. Lopez-Gomez \cite{FLG}, and X. He \& W. M. Ni \cite{HeNiA}, \cite{HeNiB} on competing species dynamics.

\section{A generalized Picone identity}
\label{sec-2}

\noindent The main goal of this section is to provide with a generalized version of a celebrated identity of M. Picone \cite{Picone}, which extends the previous ones of K. Kreith \cite{Kreith}, H. Berestycki, I. Capuzzo-Dolcetta \& L. Nirenberg \cite{BCDN}, J. L\'opez-G\'omez \cite{LG97} and S. Fern\'andez-Rinc\'on \& J. L\'opez-G\'omez \cite{FRLGNA}. It can be stated as follows. Throughout this paper, given any real Banach space $X$ and any integer $N\geq 1$, we denote by $\mathcal{M}_N^\mathrm{sym}(X)$ the set of symmetric square matrices of order $N$ with coefficients in $X$.

\begin{theorem}
\label{th2.1} Suppose that $\Omega$ is a bounded open subdomain of $\mathbb{R}^N$, $N\geq 1$, of class $\mathcal{C}^2$, $\mathrm{n}$ stands for the outward unit normal vector field along $\partial\Omega$, and let $u,v\in W^{2,\t}(\Omega)$, $\t>N$, be such that $\frac{v}{u}\in\mathcal{C}^1(\bar\Omega)$ and $\mathcal{L}u,\mathcal{L}v\in\mathcal{C}(\bar\Omega)$, where
\begin{equation}
\label{2.1}
  \mathcal{L}:=-\mathrm{div}(A\nabla\cdot)+C
\end{equation}
for some $A\in\mathcal{M}_N^\mathrm{sym}(\mathcal{C}(\bar\Omega))$ and $C\in\mathcal{C}(\bar\Omega)$. Consider $\beta\in\mathcal{C}(\partial\Omega)$ and let $\mathcal{D}$, $\mathcal{R}$ be boundary operators on $\partial\Omega$ defined by
\begin{equation*}
\left\{
\begin{array}{l}
\mathcal{D}w=w,\\[5pt]
\mathcal{R}w=\langle \nabla w, A\mathrm{n}\rangle+\beta w,
\end{array}
\right.\;\; w\in W^{2,\t}(\Omega).
\end{equation*}
Then, for every $g\in\mathcal{C}^1(\mathbb{R})$ the next identity holds
\begin{equation}
\label{2.2}
\int_\Omega g\left(\frac{v}{u}\right)[u\mathcal{L}v-v\mathcal{L}u]=\int_\Omega u^2 g'\left(\frac{v}{u}\right)\langle \nabla\frac{v}{u},A\nabla\frac{v}{u}\rangle-\int_{\partial\Omega}g\left(\frac{v}{u}\right)[\mathcal{D}u\mathcal{R}v-\mathcal{D}v\mathcal{R}u].
\end{equation}
\end{theorem}
\begin{proof}
Expanding the integrand on the left hand side and using the symmetry of $A$ yields
\begin{align*}
g\left(\frac{v}{u}\right)[u\mathcal{L}v-v\mathcal{L}u]
&=g\left(\frac{v}{u}\right)[v\,\mathrm{div}(A\nabla u)-u\,\mathrm{div}(A\nabla v)]\\
&=g\left(\frac{v}{u}\right)\mathrm{div}(v A\nabla u - u A\nabla v)\\
&=\mathrm{div}\left[g\left(\frac{v}{u}\right)(vA\nabla u-u A \nabla v)\right]-\langle \nabla g\left(\frac{v}{u}\right),v A\nabla u-u A\nabla v\rangle\\
&=\mathrm{div}\left[g\left(\frac{v}{u}\right)(v A \nabla u-u A \nabla v)\right]-g'\left(\frac{v}{u}\right)\langle \nabla\frac{v}{u},A(v\nabla u-u\nabla v)\rangle\\
&=\mathrm{div}\left[g\left(\frac{v}{u}\right)(v A \nabla u-uA\nabla v)\right]+u^2 g'\left(\frac{v}{u}\right)\langle \nabla \frac{v}{u},A\nabla\frac{v}{u}\rangle.
\end{align*}
Thus, integrating in $\Omega$, we find that
\[
\int_\Omega g\left(\frac{v}{u}\right)[u\mathcal{L}v-v\mathcal{L}u]  =\int_\Omega \mathrm{div}\left[g\left(\frac{v}{u}\right)(v A \nabla u-uA\nabla v)\right] + \int_\Omega u^2 g'\left(\frac{v}{u}\right)\langle \nabla\frac{v}{u},A\nabla\frac{v}{u}\rangle.
\]
As integrating by parts shows that
\begin{align*}
\int_\Omega\mathrm{div}\Big[g\left(\frac{v}{u}\right)(v A\nabla u & -u A \nabla v)\Big]
= \int_{\partial\Omega}g\left(\frac{v}{u}\right)\langle v A\nabla u-u A \nabla v,\mathrm{n}\rangle\\
&= \int_{\partial\Omega}g\left(\frac{v}{u}\right)\left[v\left(\langle A\nabla u,\mathrm{n}\rangle+\beta u\right)-u\left(\langle A\nabla v,\mathrm{n}\rangle+\beta v\right)\right]\\
&=-\int_{\partial\Omega}g\left(\frac{v}{u}\right)[\mathcal{D}u\mathcal{R}v-\mathcal{D}v\mathcal{R}u],
\end{align*}
the identity \eqref{2.2} holds.
\end{proof}

Since the symmetric  matrix $A(x)$ is not required to be positive definite, the second order differential operator $\mathcal{L}$ defined in \eqref{2.1} might not be of elliptic type. This is  a real novelty with respect to all previous existing results.

\section{Applications to a general class of superlinear indefinite problems}
\label{sec-3}

\noindent In this section we use Theorem \ref{th2.1} to study the existence of positive solutions of the superlinear indefinite problem
\begin{equation}
\label{3.1}
\left\{
\begin{array}{ll}
\mathcal{L} u=\lambda u-a(x)f(u) & \hbox{in}\;\;\Omega,\\
\mathcal{B}u=0 & \hbox{on}\;\;\partial\Omega,
\end{array}
\right.
\end{equation}
where $\Omega$ is an open bounded subset of $\mathbb{R}^N$, $N\geq 1$, of class $\mathcal{C}^2$, $a\in\mathcal{C}(\bar\Omega)$ is allowed to change sign, $\lambda\in\mathbb{R}$ is a parameter, $f\in\mathcal{C}(\mathbb{R})\setminus\{0\}$ with $f(0)=0$, and $\mathcal{L}$ is an uniformly elliptic operator in divergence form, like \eqref{2.1}, with $A\in\mathcal{M}_N^\mathrm{sym}(\mathcal{C}^1(\bar\Omega))$ and $C\in\mathcal{C}(\bar\Omega)$. As far as concerns the boundary of $\Omega$, $\partial\Omega$, we are assuming that $\partial\Omega=\Gamma_{\mathcal{D}}\cup\Gamma_{\mathcal{R}}$, where $\Gamma_{\mathcal{D}}$ and $\Gamma_{\mathcal{R}}$ are two disjoint closed and open subsets of $\partial \Omega$ associated with the mixed boundary operator defined by
\begin{equation*}
\mathcal{B}:=
\left\{
\begin{array}{ll}
\mathcal{D}=Id & \hbox{on}\;\;\Gamma_{\mathcal{D}},\\
\mathcal{R}=\langle \nabla \cdot,A\mathrm{n}\rangle+\beta & \hbox{on}\;\;\Gamma_{\mathcal{R}},
\end{array}
\right.
\end{equation*}
for some $\beta\in\mathcal{C}(\Gamma_{\mathcal{R}})$, where $\mathrm{n}$ stands for the outward unit normal vector field along $\partial\Omega$. As $\beta$ might change of sign, this boundary operator is of general mixed non-classical type.
\par
By \cite[Th. 7.7]{LG13}, under these general conditions,  the linear eigenvalue problem
\begin{equation}
\label{3.2}
   \left\{ \begin{array}{ll} \mathcal{L} \varphi = \sigma \varphi &\quad \hbox{in}\;\;\Omega,\\
   \mathcal{B}\varphi =0 & \quad \hbox{on}\;\;\partial\Omega,\end{array}\right.
\end{equation}
possesses a unique \emph{principal eigenvalue},
\[
  \s_0:=\sigma[\mathcal{L},\mathcal{B},\Omega],
\]
which is algebraically simple and strictly dominant. By a principal eigenvalue it is meant an eigenvalue associated with it there is a positive eigenfunction
$\varphi_0$. In such case, $\varphi_0\gg 0$ in the sense that
\begin{equation}
\label{3.3}
  \varphi_0(x)>0 \;\;\hbox{for all}\;\; x\in\Omega \cup \Gamma_{\mathcal{R}} \;\;\hbox{and}\;\;
  \frac{\partial\varphi_0}{\partial \mathrm{n}}(x)<0\;\;\hbox{for all}\;\; x \in \Gamma_{\mathcal{D}}.
\end{equation}
According to \cite[Ch. 5]{LG13}, $\varphi_0\in W^{2,\t}_{\mathcal{B}}(\Omega)$ for all $\t>N$, where
\begin{equation}
\label{3.4}
   W^{2,\t}_{\mathcal{B}}(\Omega) :=\{u\in  W^{2,\t}(\Omega)\;:\;\; \mathcal{B}u =0\;\;\hbox{on}\;\;\partial \Omega\}.
\end{equation}
Thus, $\varphi_0\in \mathcal{C}^1_{\mathcal{B}}(\bar \Omega)\cap \mathcal{C}^{1,\nu}(\Omega)$ for all $\nu<1$ and it is almost everywhere twice differentiable in $\Omega$,  much like the weak positive solutions of \eqref{3.1}.
\par
The following result establishes the existence of a curve of positive solutions, $(\lambda,u)$, of
\eqref{3.1} emanating from $u=0$ as  $\lambda$ crosses $\s_0$. It  is a straightforward application of the main theorem of M. G. Crandall \&  P. H. Rabinowitz \cite{CR71} based on the fact that $\s_0$ is algebraically simple.

\begin{theorem}
\label{th3.1}
Assume that $f$ is of class $\mathcal{C}^r$, $r\ge 2$, in a neighborhood of zero and $f(0)=f'(0)=0$. Let $\varphi_0\in W^{2,\t}_{\mathcal{B}}(\Omega)$ be the principal eigenfunction associated with $\s_0$ normalized so that
\begin{equation*}
  \int_\Omega \varphi_0^2(x)\,dx =1.
\end{equation*}
Then, there exist $\varepsilon >0$ and two maps of class $\mathcal{C}^{r-1}$,
\begin{equation*}
\mu:(-\varepsilon,\varepsilon)\to\mathbb{R},\quad y:(-\varepsilon,\varepsilon)\to W^{2,\t}_{\mathcal{B}}(\Omega),
\end{equation*}
such that $\mu(0)=0$, $y(0)=0$, $\int_{\Omega}y(s)\varphi_0=0$ for all $s\in(-\varepsilon,\varepsilon)$, and
\begin{equation}
\label{3.5}
(\lambda(s),u(s)):=(\s_0+\mu(s),s(\varphi_0+y(s)))
\end{equation}
solves \eqref{3.1} for every $s\in (-\varepsilon,\varepsilon)$. Moreover, there exists a neighborhood of $(\s_0,0)$ in $\mathbb{R}\times W^{2,\t}_{\mathcal{B}}(\Omega)$, $\mathcal{U}$, such that, for any solution $(\lambda,u)\in \mathcal{U}$ of \eqref{3.1}, either $u=0$, or there exists $s\in (-\varepsilon,\varepsilon)$ such that $(\lambda,u)=(\lambda(s),u(s))$.
\end{theorem}
\begin{proof}
Let $\omega>0$ be such that $\o >-\s_0$. Then,
\[
  \sigma[\mathcal{L}+\omega,\mathcal{B},\Omega]=\s_0+\omega>0.
\]
Hence, the solutions of \eqref{3.1} are given by the zeroes of the operator
\[
  \mathfrak{F}(\lambda,u):=u-(\mathcal{L}+\o)^{-1}[(\l+\o)u-a(x)f(u)],\qquad
  (\l,u)\in \R \times L^\t(\O),
\]
which is a compact perturbation of the identity map of class $\mathcal{C}^r$; in particular, it is Fredholm of index zero. We have that $\mathfrak{F}(\l,0)=0$ for all $\l\in\R$. Moreover, the Fr\'{e}ch\`{e}t differential $\mathfrak{L}(\l):= D_u\mathfrak{F}(\l,0)$ is given by
\[
  D_u\mathfrak{F}(\l,0) u= u-(\mathcal{L}+\o)^{-1}[(\l+\o)u].
\]
Thus, it is apparent that $\mathfrak{L}(\l)$ is an
isomorphism if $\l$ is not an eigenvalue of \eqref{3.2}. Furthermore,
\[
  \mathrm{Ker\,}D_u\mathfrak{F}(\s_0,0) =\mathrm{span\,}[\v_0]
\]
and the next transversality condition holds
\[
  \mathfrak{L}'(\s_0)\v_0 = -(\mathcal{L}+\o)^{-1}\v_0 \notin \mathrm{Im\,}\mathfrak{L}(\sigma_0).
\]
On the contrary, assume that, for some $u\in W^{2,\t}_{\mathcal{B}}(\Omega)$,
\[
  u-(\mathcal{L}+\o)^{-1}[(\s_0+\o)u] = -(\mathcal{L}+\o)^{-1}\v_0.
\]
Then,
\[
  \left\{ \begin{array}{ll} (\mathcal{L}-\s_0)u=-\v_0 & \quad \hbox{in}\;\;\O,\\ \mathcal{B}u=0&\quad
  \hbox{on}\;\;\partial\O,\end{array}\right.
\]
which contradicts \cite[Th. 7.8]{LG13}. Therefore, the desired result follows by applying
\cite[Th. 2.2.1]{LG01}, which is the main theorem of \cite{CR71}, with
\[
  Y:= \left\{y\in W^{2,\t}_{\mathcal{B}}(\O)\;:\; \int_\O \v_0(x)y(x)\,dx=0\right\}.
\]
The proof is complete.
\end{proof}

As a consequence of the definition of $u(s)$, we have that $u'(0)=\v_0\gg 0$. Hence, $\e$ can be shortened, if necessary, so that
\[
  u'(s):=\frac{du}{ds}(s)\gg 0\quad \hbox{for all} \;\; s\in (-\e,\e).
\]
Moreover, $u(s)\gg 0$ if $s\in (0,\e)$, while $u(s)\ll 0$ if $s\in (-\e,0)$, and the next result holds.

\begin{proposition}
\label{Prop3.2}
Under the same assumptions of Theorem \ref{th3.1}, the following assertions are true:
\begin{enumerate}[{\rm (i)}]
\item\label{Prop3.2.item1}  For every $p\ge 1$,
\begin{equation}
\label{3.6}
\lim_{s\to 0^{\pm}}\frac{\l(s)-\s_0}{|s|^{p-1}}=\lim_{s\to 0^{\pm}}\frac{f(s)}{s|s|^{p-1}}\int_{\Omega} a(x) \varphi_0^{p+1}(x)\,dx
\end{equation}
if the  limit on the right hand side exists.

\item\label{Prop3.2.item2} If $\l'(s)u(s)>0$ for some $s\in(-\e,\e)$, then $u(s)$ is linearly stable as a steady-state solution of the parabolic problem
\begin{equation}
\label{3.7}
\left\{
\begin{array}{ll}
\frac{\p u}{\p t} + \mathcal{L} u=\lambda u-a(x)f(u) & (x,t)\in \Omega\times (0,+\infty),\\
\mathcal{B}u=0 & (x,t)\in \p\O \times (0,+\infty),\\ u(x,0)=u_0(x)\geq 0 & x\in\O.
\end{array}
\right.
\end{equation}
In other words,
\begin{equation}
\label{3.8}
  \s[\mathcal{L}+a(x)f'(u(s))-\l(s),\mathcal{B},\O]>0.
\end{equation}

\end{enumerate}
\end{proposition}
\begin{proof}
Substituting \eqref{3.5}  in \eqref{3.1} we are driven to
\begin{equation*}
s\mathcal{L}(\varphi_0 + y(s))=s (\s_0  + \mu(s)) (\varphi_0+y(s))-a f(s(\varphi_0+y(s))).
\end{equation*}
Thus,
\begin{equation*}
s(\mathcal{L}-\s_0)y(s)=s \mu(s) (\varphi_0+y(s))-a f(s(\varphi_0+y(s))).
\end{equation*}
Hence, for any given $p\geq 1$, multiplying both sides of this identity  by $\frac{\varphi_0}{s|s|^{p-1}}$, $s\neq 0$, it is apparent that
\begin{equation*}
\frac{1}{|s|^{p-1}}\varphi_0(\mathcal{L}-\s_0) y(s)=\frac{\mu(s)}{|s|^{p-1}} \varphi_0(\varphi_0+y(s))-a \varphi_0 \frac{f(s(\varphi_0+y(s)))}{s|s|^{p-1}}.
\end{equation*}
Therefore, since
\[
  \int_\O \v_0 (\mathcal{L}-\s_0) y(s) = \int_\O y(s)(\mathcal{L}-\s_0)\v_0=0,
\]
we find that
\begin{equation}
\label{3.9}
0= \frac{\mu(s)}{|s|^{p-1}}\int_{\Omega}\varphi_0(\varphi_0+y(s))-\int_{\Omega} a \varphi_0  (\varphi_0+y(s))^p\frac{f(s(\varphi_0+y(s)))}{s|s|^{p-1}(\varphi_0+y(s))^p}.
\end{equation}
The identity \eqref{3.6} follows from Lebesgue's dominated convergence theorem by letting $s\to 0$ in
\eqref{3.9} provided $p\geq 1$ satisfy
\begin{equation*}
\lim_{s\to 0^{\pm}}\frac{f(s)}{s|s|^{p-1}}\in \mathbb{R}.
\end{equation*}
Finally, differentiating with respect to $s$ the identity
\[
  \mathfrak{F}(\l(s),u(s))=0,\qquad s\in (-\e,\e),
\]
inverting $(\mathcal{L}+\o)^{-1}$ and rearranging terms, it becomes apparent that
\begin{equation*}
(\mathcal{L}-\lambda(s)+af'(u(s)))u'(s)=\lambda'(s)u(s),\quad s\in (-\e,\e).
\end{equation*}
Since shortening $\e$, we can assume that $u'(s)\gg 0$ for all $s\in(-\e,\e)$, it follows from
\cite[Th.7.10]{LG13} that $\lambda'(s)u(s)>0$ implies \eqref{3.8}, ending the proof.
\end{proof}

It should be noted that \eqref{3.6} provides us with the sign of $\mu(s)=\l(s)-\s_0$ and hence, the bifurcation
direction of the curve of positive solutions, $(\l(s),u(s))$, $s>0$,  in terms of the behavior of
$f(u)$ at $u=0$ and the sign of the integral
\[
\int_{\Omega} a(x) \varphi_0^{p+1}(x)\,dx.
\]
However, as we are applying Theorem \ref{th3.1}, $f$ is required to be of class $\mathcal{C}^2$ regularity. In particular, the next result holds.

\begin{corollary}
\label{co3.3}
Under the same assumptions of Theorem \ref{th3.1}, suppose that in addition $f(u):=u|u|^{p-1}$ for some $p\ge 2$. Then,
\begin{equation*}
\lim_{s\to 0^{\pm}} \frac{\l(s)-\s_0}{|s|^{r-1}}=0\;\;\hbox{for all}\;\;r\in[1,p)\;\; \hbox{and}\;\;
\lim_{s\to 0^{\pm}} \frac{\l(s)-\s_0}{|s|^{p-1}}=\int_\Omega a(x)\varphi_0^{p+1}(x)\,dx.
\end{equation*}
Thus, the bifurcation to positive solutions is supercritical if $\int_\Omega a(x)\varphi_0^{p+1}(x)\,dx>0$, while it is subcritical, if $\int_\Omega a(x)\varphi_0^{p+1}(x)\,dx<0$.
\end{corollary}

In the general case when $f$ is merely continuous, one can still use the global bifurcation theorems of \cite[Ch. 6]{LG13} to infer that the set of solutions of \eqref{3.1} possesses a (connected) component, $\mathfrak{C}^+$, of positive solutions which is unbounded in $\R\times \mc{C}(\bar\O)$ and
satisfies $(\s_0,0)\in \mathfrak{\bar C}^+$. But, in this general case, the sharp information
provided by Theorem \ref{th3.1} in a neighborhood of $(\s_0,0)$ is lost.

\subsection{Nonexistence of small positive solutions for $\lambda\ge\s_0$.}
\label{sec-3.1}

\noindent Astonishingly, the next result provides us with a sufficient (optimal) condition so that \eqref{3.1} cannot admit  positive small solutions for $\l \geq \s_0$ even in the general case when $f$ is continuous.

\begin{theorem}
\label{th3.4}
Assume that, for some  $p\geq 1$,
\begin{equation}
\label{3.10}
\lim_{s\da 0} \frac{f(s)}{s^{p}}\int_\Omega a(x)\varphi_0^{p+1}(x)\,dx< 0.
\end{equation}
Then, there exists $\varepsilon>0$ such that $\lambda< \s_0$ if \eqref{3.1} admits a solution, $(\lambda,u)$, with $u \gneq 0$ and $\|u\|_{\infty}<\varepsilon$. In other words, \eqref{3.1} cannot admit small positive solutions if $\l \geq \s_0$.
\end{theorem}
\begin{proof}
Let $(\lambda,u)$ be a positive solution of \eqref{3.1}. Then,  since $p\geq 1$, it is easily seen
that $u\gg 0$ in $\O$, in the sense of \eqref{3.3}.
Thus, since $\varphi_0\gg 0$, the quotient $\tfrac{\varphi_0}{u}$ preserves its regularity even on the Dirichlet components of $\partial\Omega$. Thus, applying Theorem \ref{th2.1} with $g(t)=t^p$, $t\in\R$,  to the functions $u$ and $\varphi_0$ and taking into account that $A(x)$ is positive definite  yields the estimate
\begin{equation*}
\int_{\Omega} \left(\frac{\varphi_0}{u}\right)^{p}\left(u\mathcal{L}\varphi_0-\varphi_0\mathcal{L}u\right)=
p\int_{\Omega}\frac{\varphi_0^{p-1}}{u^{p-3}}\langle \nabla\frac{\varphi_0}{u},A\nabla\frac{\varphi_0}{u}\rangle \ge 0,
\end{equation*}
since $\mathcal{B}u=\mathcal{B}\varphi_0=0$ on $\partial\Omega$ and, so, either $\mathcal{D}u=\mathcal{D}\varphi_0=0$, or $\mathcal{R}u=\mathcal{R}\varphi_0=0$, on each component of $\partial\Omega$. On the other hand, using the fact that $u$ solves \eqref{3.1} it follows from  the definition of $\varphi_0$ that
\begin{equation*}
u\mathcal{L}\varphi_0-\varphi_0\mathcal{L}u=(\s_0-\lambda)u\varphi_0+a(x)f(u)\varphi_0.
\end{equation*}
Hence, multiplying this identity  by $\frac{\varphi_0^p}{u^p}$ and integrating in $\Omega$ we obtain that
\begin{equation}
\label{3.12}
(\s_0-\lambda)\int_\Omega \frac{\varphi_0^{p+1}}{u^{p-1}}+\int_{\Omega} a \varphi_0^{p+1}\frac{f(u)}{u^p}=\int_\Omega\left(\frac{\varphi_0}{u}\right)^{p}(u\mathcal{L}\varphi_0-\varphi_0\mathcal{L}u)\ge 0.
\end{equation}
Therefore, by the Lebesgue's dominated convergence theorem, it follows from \eqref{3.10} that
\begin{equation*}
(\s_0-\lambda)\int_\Omega \frac{\varphi_0^{p+1}}{u^{p-1}}\ge -\int_{\Omega} a \varphi_0^{p+1}\frac{f(u)}{u^p}\xrightarrow{\|u\|_{\infty}\to 0,\ u\ge 0}-\lim_{s\to 0^+} \frac{f(s)}{s^{p}}\int_{\Omega} a \varphi_0^{p+1}> 0.
\end{equation*}
Consequently, since $\int_{\Omega}\frac{\varphi_0^{p+1}}{u^{p-1}}>0$, it is apparent that $\s_0>\l$. This ends the proof.
\end{proof}

In this result, the size of $\varepsilon>0$ only depends on how
\begin{equation}
\label{3.13}
\int_{\Omega} a \varphi_0^{p+1}\frac{f(u)}{u^p}\;\;\hbox{approximates}\;\;\lim_{s\da 0} \frac{f(s)}{s^{p}}\int_{\Omega} a \varphi_0^{p+1}\;\;\hbox{as}\;\;\|u\|_{\infty}\to 0.
\end{equation}

\subsection{Nonexistence of positive solutions for $\l\geq \s_0$}
\label{sec-3.2}

\noindent The next result shows that when the approximation \eqref{3.13} occurs suddenly, i.e., when $f(u)=u|u|^{p-1}$, $u\in\R$, for some $p > 1$, then $\varepsilon=+\infty$, i.e., \eqref{3.1} cannot
admit a positive solution if $\lambda\geq \s_0$. It is a substantial extension of a result of Section 6 of H. Berestycki, I. Capuzzo-Dolcetta \& L. Nirenberg \cite{BCDN} and \cite[Th. 4.2]{LG97}. Later, in Theorem \ref{th3.7}, we  will establish the optimality of this result.

\begin{theorem}
\label{th3.5}
Assume that  $p>1$ exists such that $f(s):=s^p$ for every $s\ge 0$ and
\begin{equation}
\label{3.14}
\int_\Omega a(x)\varphi_0^{p+1}(x)\,dx\le 0.
\end{equation}
Then, $\lambda< \s_0$ if \eqref{3.1} admits a positive solution, $(\lambda,u)$.
\end{theorem}
\begin{proof}
The proof can be easily adapted from the proof of Theorem \ref{th3.4}. First, assume that $\frac{\varphi_0}{u}$ is not constant. Then, by \eqref{3.12} and Theorem \ref{th2.1},
\begin{equation*}
(\s_0-\lambda)\int_{\Omega}\frac{\varphi_0^{p+1}}{u^{p-1}}+\int_{\Omega} a\varphi_0^{p+1}=p\int_{\Omega}\frac{\varphi_0^{p-1}}{u^{p-3}}\langle \nabla\frac{\varphi_0}{u},A\nabla\frac{\varphi_0}{u}\rangle >0.
\end{equation*}
Hence,
\begin{equation*}
(\s_0-\lambda)\int_{\Omega}\frac{\varphi_0^{p+1}}{u^{p-1}}>-\int_{\Omega} a\varphi_0^{p+1}\ge 0,
\end{equation*}
and so $\lambda<\s_0$. On the other hand, if $\frac{\varphi_0}{u}$ is a (positive) constant, then $u$ satisfies
\begin{equation*}
\lambda u-a(x)u^p=\mathcal{L}u=\s_0 u
\end{equation*}
and hence,
\begin{equation*}
a(x)u^{p-1}(x)=\lambda-\s_0\quad\hbox{for all}\;\; x\in\Omega.
\end{equation*}
Thus, $a$ cannot change sign, which contradicts \eqref{3.14}. Therefore, $\lambda<\s_0$. This ends the proof.
\end{proof}

\subsection{Neutrally stable solutions are quadratic subcritical turning points}
\label{sec-3.3}

\noindent A positive solution of \eqref{3.1}, $(\l_0,u_0)$, is said to be \emph{neutrally stable} if
\begin{equation}
\label{3.15}
\sigma[\mathcal{L}-\lambda_0+a(x)f'(u_0),\mathcal{B},\O]=0.
\end{equation}
The next result, also based on Theorem \ref{th2.1},  provides us with the local structure of the set of  solutions of \eqref{3.1} around any neutrally stable positive solution. It is a substantial
generalization of Proposition 3.2 of R. G\'{o}mez-Re\~{n}asco \& J. L\'{o}pez-G\'{o}mez \cite{GRLG00}.
Based on this result, we will establish in Section \ref{sec-3.5} the uniqueness of the linearly stable positive solution of \eqref{3.1} if it exists. This uniqueness result generalize, very substantially, the corresponding uniqueness theorems of \cite{GRLG00}, \cite{GRLG01} and \cite[Ch. 9]{LG15}.

\begin{theorem}
\label{th3.6}
Assume that one of the following three condition holds:
\begin{enumerate}[{\rm i)}]
\item\label{th3.6.item1}  $\Gamma_\mathcal{D}=\emptyset$ and $f(u)=u\log u$ for every $u\ge 0$,
\item\label{th3.6.item2} $\Gamma_\mathcal{D}=\emptyset$ and $f(u)=u^p$  for every $u\ge 0$ with $p\in (0,1)\cup(1,2)$,
\item\label{th3.6.item3} $f(u)=u^p$ for every $u\ge 0$ with $p \geq 2$.
\end{enumerate}
Let $(\lambda_0,u_0)$ be a neutrally stable positive solution of \eqref{3.1} such that $u_0\ge 1$ in case {\rm\ref{th3.6.item1})} and $u_0\ge \tau> 0$ in case {\rm \ref{th3.6.item2})}. Let $\psi_0\in W^{2,\t}(\O)$, $\t>N$, denote the principal eigenfunction associated with \eqref{3.15} normalized so that $\int_\O \psi_0^2=1$. Then, there exist $\varepsilon>0$ and two functions of class $\mc{C}^2$,
\begin{equation*}
\lambda :(-\varepsilon,\varepsilon)\to \mathbb{R},
\quad\hbox{and}\quad u:(-\varepsilon,\varepsilon)\to W^{2,\t}_{\mathcal{B}}(\O)
\end{equation*}
such that
\begin{equation*}
(\lambda(0),u(0))=(\lambda_0,u_0),\quad(\lambda'(0),u'(0))=(0,\psi_0),\quad \lambda''(0)<0,
\end{equation*}
for which the curve $(\lambda(s),u(s))$ provides us with the set of solutions of \eqref{3.1}
in a neighborhood of $(\lambda_0,u_0)$. Moreover, shortening $\varepsilon$, if necessary, $u(s)$ is linearly stable if $s\in(-\varepsilon,0)$ and linearly unstable if $s\in(0,\varepsilon)$.
\end{theorem}
\begin{proof}
The existence of a real analytic curve of solutions is an immediate consequence of Proposition 20.7 of H. Amann  \cite{Am76}. A more recent approach, where the underlying analysis has been considerably tidied up, can be found in \cite[Pr. 9.7]{LG15}. In these references the existence of the curve follows from the implicit function theorem after a Lyapunov--Schmidt decomposition. The fact that
\begin{equation*}
(\lambda(0),u(0))=(\lambda_0,u_0),\quad u'(0) =\psi_0,
\end{equation*}
follows easily from these previous constructions. Differentiating with respect to $s$ in
\begin{equation*}
\mathcal{L}u(s)=\lambda(s)u(s)-a(x)f(u(s))
\end{equation*}
yields
\begin{equation}
\label{3.16}
\mathcal{L}u'(s)=\lambda'(s)u(s)+\lambda(s)u'(s)-a(x)f'(u(s))u'(s).
\end{equation}
Thus, particularizing at $s=0$, multiplying the resulting identity by $\psi_0$ and integrating by parts in $\Omega$ yields
\begin{equation*}
\lambda'(0)=\frac{\int_{\Omega}\psi_0 [\mathcal{L}-\lambda_0+a(x)f'(u_0)]\psi_0}{\int_{\Omega} u_0\psi_0}=0.
\end{equation*}
Similarly, by differentiating \eqref{3.16} with respect to $s$, it follows that
\begin{align*}
\mathcal{L}u''(s)=\lambda''(s)u(s) & +2\lambda'(s)u'(s)  +\lambda(s)u''(s)\\
&-a(x)f''(u(s))(u'(s))^2-a(x)f'(u(s))u''(s).
\end{align*}
Thus, since $\l'(0)=0$, particularizing at $s=0$ shows that
\begin{equation*}
\mathcal{L}u''(0)=\lambda''(0)u_0+\lambda_0 u''(0)-a(x)f''(u_0)\psi_0^2-a(x)f'(u_0)u''(0).
\end{equation*}
Hence,  multiplying by $\psi_0$ and integrating by parts in $\Omega$ yields
\begin{align*}
\lambda''(0)&=\frac{\int_{\Omega}u''(0)[\mathcal{L}-\lambda_0+af'(u_0)]\psi_0+
\int_{\Omega}a\psi_0^3 f''(u_0)}{\int_{\Omega}u_0\psi_0}=\frac{\int_{\Omega}a\psi_0^3 f''(u_0)}{\int_{\Omega}u_0\varphi_0}.
\end{align*}
To prove that $\l''(0)<0$ one can argue as follows. By the definition of $u_0$ and $\psi_0$, we have that \begin{equation}
\label{3.17}
\begin{split}
\left(\frac{\psi_0}{u_0}\right)^2( u_0\mathcal{L}\psi_0 & -\psi_0\mathcal{L}u_0) =
\frac{\psi_0^2}{u_0}\mathcal{L}\psi_0-\frac{\psi_0^3}{u_0^2}\mathcal{L}u_0\\ & =
\frac{\psi_0^2}{u_0}\left(\l_0\psi_0-af'(u_0)\psi_0\right) -  \frac{\psi_0^3}{u_0^2}\left( \l_0 u_0-
af(u_0)\right)\\ & = -a\psi_0^3\frac{f'(u_0)u_0-f(u_0)}{u_0^2} = -a\psi_0^3 \left( \frac{f(u)}{u}\right)'\Big|_{u=u_0}.
\end{split}
\end{equation}
Thus, integrating in $\O$, we find that
\[
\int_\Omega\left(\frac{\psi_0}{u_0}\right)^2(u_0\mathcal{L}\psi_0-\psi_0\mathcal{L}u_0) =-\int_\Omega a\psi_0^3\left(\frac{f(u)}{u}\right)'\Big|_{u=u_0}.
\]
As it turns out that the functions $f(u)=u^p$, $p\in (0,+\infty)\setminus \{1\}$, and $f(u)=u\log u$, are the unique ones satisfying
\begin{equation*}
\left(\frac{f(u)}{u}\right)'=\frac{1}{p}\,f''(u)\quad\hbox{and}\quad
\left(\frac{f(u)}{u}\right)'=f''(u),
\end{equation*}
respectively, it becomes apparent that, for these functions,
\begin{equation*}
\mathrm{sgn}\,\lambda''(0)=\mathrm{sgn}\,\int_{\Omega}a\psi_0^3 f''(u_0) = - \mathrm{sgn}\,\int_\Omega\left(\frac{\psi_0}{u_0}\right)^2
(u_0\mathcal{L}\psi_0-\psi_0\mathcal{L}u_0).
\end{equation*}
Consequently, the fact that $\l''(0)<0$ follows easily from Theorem \ref{th2.1}, which provides us with the identity
\begin{equation*}
\int_\Omega\left(\frac{\psi_0}{u_0}\right)^2(u_0\mathcal{L}\psi_0-\psi_0\mathcal{L}u_0)=2 \int_\Omega \psi_0u_0 \langle \nabla\frac{\psi_0}{u_0}, \nabla\frac{\psi_0}{u_0}\rangle> 0,
\end{equation*}
because  $u_0$ cannot be a multiple of $\psi_0$. Indeed, on the contrary case, $\psi_0=\kappa u_0$ for some $\kappa >0$ and hence, it follows from \eqref{3.17} that
\begin{equation*}
a \left(f(u_0)-u_0 f'(u_0)\right)=0\;\;\hbox{in}\;\;\Omega.
\end{equation*}
Thus, since we are assuming that  $a\in\mathcal{C}(\bar\Omega)$ satisfies $a\neq 0$, there exists  $x_0\in\Omega$ such that
\begin{equation*}
f(u_0(x_0))=u_0(x_0) f'(u_0(x_0)).
\end{equation*}
However, for the special choices $f(u)=u\log u$, and  $f(u)=u^p$, $p\in(0,+\infty)\setminus\{1\}$,
this identity entails $u_0(x_0)=0$. Therefore, we are in case \ref{th3.6.item3}) and necessarily $x_0\in\partial\Omega$, which contradicts $x_0\in \Omega$.
\par
For the stability of the positive solution $(\lambda(s),u(s))$ for sufficiently small $s\sim 0$, we have to ascertain the sign of the principal eigenvalue
\begin{equation*}
\Sigma(s):=\sigma[\mathcal{L}-\lambda(s)+a(x)f'(u(s)),\mathcal{B},\O].
\end{equation*}
In particular, since $\Sigma(0)=0$, it suffices to show that $\Sigma'(0)<0$. As $\Sigma(s)$ is a simple eigenvalue of class $\mc{C}^1$ in $s$, it follows from the abstract theory of T. Kato \cite{Kato} that $\psi_0$ admits a $\mc{C}^1$ perturbation, $\psi(s)$, $s\sim 0$, such that $\psi(0)=\psi_0$ and $\int_\O \psi^2(s)=1$ for sufficiently small $s$ (see also \cite[Le. 2.2.1]{LG01}). Thus, differentiating with respect to $s$ the identity
\begin{equation*}
\mathcal{L}\psi(s)-\lambda(s)\psi(s)+af'(u(s))\psi(s)=\Sigma(s)\psi(s),
\end{equation*}
we are driven to the identity
\begin{equation*}
[\mathcal{L}-\lambda(s)+af'(u(s))]\psi'(s)-\lambda'(s)\psi(s)+af''(u(s))u'(s)\psi(s)
=\Sigma(s)\psi'(s)+\Sigma'(s)\psi(s).
\end{equation*}
So, particularizing at $s=0$, we have that
\begin{equation*}
[\mathcal{L}-\l_0+af'(u_0)]\psi'(0)+af''(u_0)\psi_0^2=\Sigma'(0)\psi_0.
\end{equation*}
Therefore, multiplying by $\psi_0$ this identity and integrating by parts in $\O$ the next identity holds \begin{equation*}
\Sigma'(0)=\int_\Omega a f''(u_0)\psi_0^3=\lambda''(0)\int_\Omega u_0 \psi_0<0,
\end{equation*}
which ends the proof.
\end{proof}

Figure \ref{Fig1} represents a genuine quadratic subcritical turning point. Theorem \ref{th3.6}
establishes that this is the bifurcation diagram of \eqref{3.1} in a neighborhood of any
linearly neutrally stable positive solution, $(\l_0,u_0)$. The half low branch, plotted with a continuous line, is filled in by linearly stable positive solutions, while the upper one, plotted with a discontinuous line, consists of linearly unstable positive solutions with one-dimensional unstable manifold.

\begin{figure}[h!]
\centering
\includegraphics[scale=0.23]{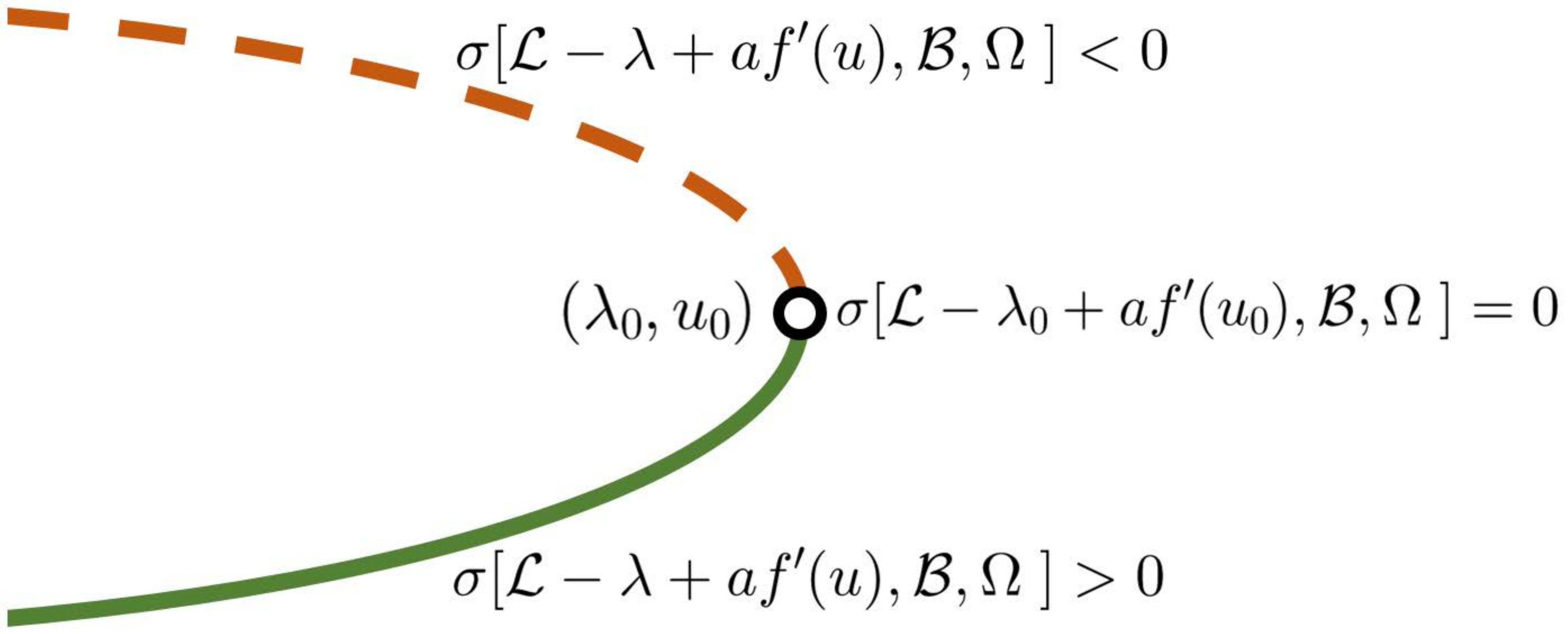}
\caption{Local bifurcation diagram at a linearly neutrally stable positive solution, $(\l_0,u_0)$.}
\label{Fig1}
\end{figure}

\subsection{Optimality of Theorems \ref{th3.5} and \ref{th3.6}}
\label{sec-3.4}

\noindent The next result shows the optimality of Theorem \ref{th3.5}, in the sense that it fails to be true when $f(u)$ does not have the required form. Remember that $\varphi_0\gg 0$ stands for the principal eigenfunction associated with $\s_0:=\sigma[\mathcal{L},\mathcal{B},\Omega]$ normalized so that $\int_\O\v_0^2=1$.

\begin{theorem}
\label{th3.7}
Assume that there exist $0<p<q$ and $F$, $G\in \mathcal{C}^r([0,+\infty);\mathbb{R})$, $r\ge 2$, such that $F(0)=G(0)=0$, $F'(0)=G'(0)=0$ and
\begin{equation*}
\lim_{s\to 0}\frac{F(s)}{s|s|^{p-1}} \int_\Omega a\varphi_0^{p+1}<0<\lim_{s\to 0}\frac{G(s)}{s|s|^{q-1}} \int_\Omega a\varphi_0^{q+1}.
\end{equation*}
Then, there exists $\nu_0>0$ such that, for every $\nu\in (0,\nu_0)$,  the problem
\begin{equation}
\label{3.18}
\left\{
\begin{array}{ll}
\mathcal{L}u=\lambda u-a(x)(\nu F(u)+G(u)) & \hbox{in}\;\;\Omega,\\
\mathcal{B}u=0 & \hbox{on}\;\;\partial\Omega,
\end{array}
\right.
\end{equation}
admits positive solutions for values of the parameter $\l$ at both sides of $\s_0$.
\end{theorem}

The hypothesis of Theorem \ref{th3.7} can be fulfilled even with polynomials. For example, the choices
\[
  \Omega=\left(-\frac{\pi}{2},\frac{\pi}{2}\right), \qquad \mathcal{L}=-\frac{d^2}{dx^2}, \qquad \G_{\mc{R}}=\emptyset,
\]
and
\[
   a(x)=\cos x -0.9,\qquad F(u)=u^2,\qquad G(u)=u^3,
\]
satisfy  all its assumptions with $p=2$ and $q=3$. Indeed, in this case
\[
   \s_0=1,\qquad \v_0(x)=\sqrt{\tfrac{2}{\pi}} \cos x,
\]
and
\begin{align*}
\lim_{s\to 0}\frac{F(s)}{s|s|}\int_\Omega (\cos x-0.9) \cos^3 x \,dx & =-0.0219028 < 0,\\[5pt]
\lim_{s\to 0}\frac{G(s)}{s|s|^2}\int_\Omega (\cos x -0.9)\cos^4 x\,dx & =0.00637915 > 0.
\end{align*}
Therefore, all the requirements of Theorem \ref{th3.7} hold.
\vspace{0.2cm}

\noindent\emph{Proof of Theorem \ref{th3.7}}.
Suppose $\nu\geq 0$. Then, owing to Theorem \ref{th3.1}, there exist $s_0=s_0(\nu) >0$ and two maps of class $\mathcal{C}^{r-1}$
\begin{equation*}
\lambda(s):\;[0,s_0)\to \mathbb{R},\qquad u(s):\;[0,s_0)\to W^{2,\t}_\mathcal{B}(\Omega),
\end{equation*}
such that
\begin{equation*}
(\lambda(0),u(0))=(\s_0,0),\quad u'(0)=\varphi_0,
\end{equation*}
and $(\l(s),u(s))$ solves \eqref{3.1} for all $s \in(-s_0,s_0)$.
Moreover, by Proposition \ref{Prop3.2}, for every $\nu>0$,
\begin{align*}
\lim_{s\da 0}  \frac{\lambda(s)-\s_0}{s^{p-1}} & =\lim_{s\da 0}
\frac{\nu   F(s)+G(s)}{s^p}\int_{\Omega}a \varphi_0^{p+1}\\ &
=\nu \lim_{s\da 0}\frac{F(s)}{s^p} \int_\Omega a\varphi_0^{p+1}+\lim_{s\da 0}\left(
s^{q-p} \frac{G(s)}{s^q}\right) \int_{\Omega}a \varphi_0^{p+1}\\ & =
\nu \lim_{s\da 0}\frac{F(s)}{s^p} \int_\Omega a\varphi_0^{p+1}<0,
\end{align*}
whereas, for $\nu=0$,
\begin{equation*}
\lim_{s\da 0} \frac{\lambda(s)-\s_0}{s^{q-1}}=\lim_{s\da 0} \frac{G(s)}{s^q}\int_{\Omega}a \varphi_0^{q+1}>0.
\end{equation*}
Thus, when $\nu=0$, the bifurcation of the curve of positive solutions $(\lambda(s),u(s))$ from the trivial branch is supercritical. In such case, by Proposition \ref{Prop3.2}, we also have that $u(s)$, as an steady-state solution of \eqref{3.7}, is linearly stable for sufficiently small $s\in (0,s_0)$. Subsequently, we shorten $s_0$, if necessary, so that $u(s)$ is linearly stable for all $s\in(0,s_0)$. Then, $\lambda'(s)>0$ for each $s\in (0,s_0)$. Thus, there exists $\e>0$ such that the curve of positive solutions $(\l(s),u(s))$, $s\sim 0$, can be parameterized by $\l$, $(\l,u(\l))$,
with $\lambda\in(\s_0,\s_0+\e)$, in a neighborhood of the bifurcation point $(\s_0,0)$. In particular,
\begin{equation}
\label{3.19}
\sigma[\mathcal{L}-\lambda+aG'(u(\lambda)),\mathcal{B},\O]>0,\qquad \l\in (\s_0,\s_0+\e).
\end{equation}
Pick $\o>-\s_0$ arbitrary and two values $\l_1, \l_2\in (\s_0,\s_0+\e)$, with $\l_1<\l_2$. Then, setting
\begin{equation*}
\mathfrak{F}(\nu,\lambda,u):=u-(\mathcal{L}+\o)^{-1}[(\lambda +\o)u-a (\nu F(u)+G(u))],
\end{equation*}
we have that, for every $\lambda\in [\l_1,\l_2]$,
\begin{equation}
\label{3.20}
D_u \mathfrak{F}(0,\lambda,u(\lambda))u=u-(\mathcal{L}+\o)^{-1}\left[(\lambda+\o) u+a G'(u(\lambda))u\right].
\end{equation}
Since it is a compact perturbation of the identity map, $D_u \mathfrak{F}(0,\lambda,u(\lambda))$ is Fredholm of index zero in $L^\t(\O)$ for all $\t>N$. Moreover, owing to \eqref{3.19}, it is a topological isomorphism. Therefore, by the implicit function theorem and the compactness of $[\l_1,\l_2]$, $u(\l)$ can be regarded as a function of class $\mathcal{C}^r$-regularity of $\l$ and $\nu$, $u(\l,\nu)$,
in $[\l_1,\l_2]\times [0,\nu_0]$ for sufficiently small $\nu_0>0$. Furthermore, by \eqref{3.19}, it becomes apparent that $(\l,u(\l,\nu))$ is linearly asymptotically stable for all $\l\in [\l_1,\l_2]$
and $\nu \in [0,\nu_0]$.
\par
On the other hand, as soon as $\nu>0$, we have that
\begin{equation*}
\lim_{s\da 0} \frac{\lambda(s)-\s_0}{s^{p-1}}=\lim_{s\da 0} \frac{\nu F(s)+G(s)}{s^p}\int_{\Omega}a \varphi_0^{p+1}=\nu \lim_{s\da 0} \frac{F(s)}{s^p}\int_{\Omega}a \varphi_0^{p+1}<0.
\end{equation*}
Thus, the bifurcation of $(\lambda(s),u(s))$ from $(\l,0)$ is subcritical. Consequently,
for every $\nu \in (0,\nu_0)$, \eqref{3.1} admits positive solutions at both sides of $\s_0$,
which ends the proof of the theorem.  \hfill $\Box$
\vspace{0.2cm}
\par
With a little bit more of effort, much like in the proof of the main theorem of M. G. Crandall \& P. H. Rabinowitz \cite{CR71} (see also the proof of Theorem 2.2.1 in \cite{LG01}), one can also define the auxiliar operator
\[
  \mathfrak{G}(s,\nu,\l,y):= \left\{ \begin{array}{ll} s^{-1}\mathfrak{F}(\nu,\l,s(\v_0+y)), &\quad
  \hbox{if}\;\; s\neq 0, \\[5pt] D_u\mathfrak{F}(\nu,\l,0)(\v_0+y),&\quad \hbox{if}\;\; s=0, \end{array}\right.
\]
and apply the implicity function to it at $(0,0,\s_0,0)$ to infer that, actually, the local bifurcation diagram of positive solutions of \eqref{3.18} for sufficiently small $\nu>0$ looks like shows Figure \ref{Fig2}. This is a direct consequence from the uniqueness, uniform in $\nu$, of the smooth curve of positive solutions of
\eqref{3.18} bifurcating from $(\s_0,0)$.  Therefore, the example after the statement of Theorem \ref{th3.7} also shows the optimality of Theorem \ref{th3.6} in the sense that if condition \ref{th3.6.item3}) of Theorem \ref{th3.6} fails, then
the problem can admit supercritical turning points at linearly neutrally stable positive solutions, like
the one shown on the right plot of Figure \ref{Fig2}.

\begin{figure}[h!]
\centering
\includegraphics[scale=0.24]{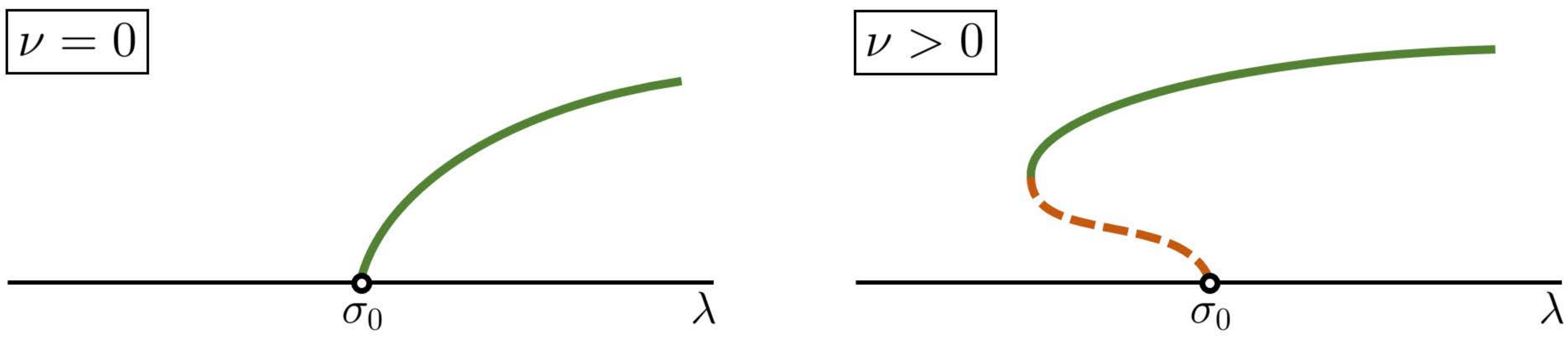}
\caption{Bifurcation diagrams of \eqref{3.18} when $\nu=0$ and $\nu>0$.}
\label{Fig2}
\end{figure}

\subsection{Uniqueness of the stable solution of \eqref{3.1} when $f(u)=u^p$, $p\geq 2$}
\label{sec-3.5}

\noindent The next result relies on Theorem \ref{th3.6}. It is a substantial extension of
the previous results of R. G\'{o}mez-Re\~{n}asco \& J. L\'{o}pez-G\'{o}mez \cite{GRLG00},
\cite{GRLG01}, for as here we are dealing with general boundary operators of mixed type.

\begin{theorem}
\label{th3.8}
Suppose that $f(u)=u^p$ for all $u\geq 0$ with $p\geq 2$. Then:
\begin{enumerate}[{\rm i)}]
\item\label{th3.8.item1} Any positive solution, $(\l_0,u_0)$ of \eqref{3.1} with $\l_0\leq \s_0$ must be linearly unstable, as an steady state of \eqref{3.7}, i.e.,
\begin{equation}
\label{3.21}
  \s[\mathcal{L}-\l_0+af'(u_0),\mathcal{B},\O]=\s[\mathcal{L}-\l_0+pau_0^{p-1},\mathcal{B},\O]<0.
\end{equation}
\item\label{th3.8.item2} The problem \eqref{3.1} admits some linearly stable positive solution, $(\l_0,u_0)$,  if, and only if,
\begin{equation}
\label{3.22}
  \mathcal{D}:=\int_\O a(x)\v_0^{p+1}(x)\,dx >0.
\end{equation}
Moreover, in such case, $\l_0>\s_0$.
\item\label{th3.8.item3} Suppose \eqref{3.22} and $\l>\s_0$. Then, the unique positive linearly stable or linearly neutrally stable solution of \eqref{3.1} is the minimal one.
\item\label{th3.8.item4} Suppose that \eqref{3.1} admits a positive solution $(\l_0,u_0)$ for some $\lambda_0>\s_0$. Then, it admits a minimal solution $(\l_0,u_{\mathrm{min}})$.
\end{enumerate}
\end{theorem}
\begin{proof}
First we will prove Part \ref{th3.8.item1}). Let  $(\l_0,u_0)$ be a positive solution of \eqref{3.1} with $\l_0\leq \s_0$. Then, by the uniqueness of the principal eigenvalue, it follows from \eqref{3.1} that
\begin{equation}
\label{3.23}
  \l_0= \s[\mathcal{L}+au_0^{p-1},\mathcal{B},\O].
\end{equation}
Thus, when $a\lneq 0$, it follows from the monotonicity of the principal eigenvalue with respect to
the potential that
\[
  \s[\mathcal{L}+p au_0^{p-1}-\l_0,\mathcal{B},\O]<\s[\mathcal{L}+ au_0^{p-1}-\l_0,\mathcal{B},\O]=0.
\]
This monotonicity, in the general setting covered in this paper, was established by S. Cano-Casanova
\& J. L\'{o}pez-G\'{o}mez \cite[Pr. 3.3]{CCLG}. Moreover, in this case, it also follows from
\eqref{3.23} that
\[
   \l_0= \s[\mathcal{L}+au_0^{p-1},\mathcal{B},\O]<  \s[\mathcal{L},\mathcal{B},\O]=\s_0.
\]
Thus, \eqref{3.1} cannot admit a positive solution if $\l\geq \s_0$.
\par
The proof of Part \ref{th3.8.item1}) in the general case when $a(x)$ changes of sign is far more subtle than
in the special case when $a\lneq 0$. Our proof here is an adaptation of the proof of Theorem 9.9
in \cite{LG15}. It proceeds by contradiction. Suppose that \eqref{3.1} possesses a positive solution, $(\l_0,u_0)$, such that
\begin{equation}
\label{3.24}
  \l_0\leq \s_0,\qquad \s[\mathcal{L}-\l_0+pau_0^{p-1},\mathcal{B},\O]\geq 0.
\end{equation}
According to Theorem \ref{th3.6}, this entails the existence of some positive solution,
$(\l_1,u_1)$, of \eqref{3.1} such that
\begin{equation}
\label{3.25}
  \l_1\leq \l_0\leq \s_0,\qquad \s[\mathcal{L}-\l_1+p au_1^{p-1},\mathcal{B},\O]>0.
\end{equation}
By the implicit function theorem applied to the operator $\mathfrak{F}(\l,u)$ defined in the
proof of Theorem \ref{th3.1}, it becomes apparent that $(\l_1,u_1)$ lies on a smooth curve of
positive solutions, $(\l,u(\l))$, $\l\sim \l_1$, such that
\begin{equation}
\label{3.26}
       \s[\mathcal{L}-\l+p au^{p-1}(\l),\mathcal{B},\O]>0\quad \hbox{for}\;\;\l\sim\l_1.
\end{equation}
By global continuation of the curve $(\l,u(\l))$ for $\l<\l_1$, one of the following options occurs:
\begin{enumerate}[{\rm (a)}]
\item\label{th3.8.sub.item1} $u(\l)\gg 0$ and $\s[\mathcal{L}-\l+p au^{p-1}(\l),\mathcal{B},\O]>0$ for all $\l<\l_1$.
\item\label{th3.8.sub.item2} There exists $\l_2<\l_1$ such that $u(\l)\gg 0$ and $\s[\mathcal{L}-\l+p au^{p-1}(\l),\mathcal{B},\O]>0$ for all $\l<\l_2$, though $u(\l_2)=0$.

\item\label{th3.8.sub.item3} There exists $\l_3<\l_1$ such that $u(\l)\gg 0$ and $\s[\mathcal{L}-\l+p au^{p-1}(\l),\mathcal{B},\O]>0$ for all $\l<\l_3$, though $u(\l_3)\gg 0$ and
\[
  \s[\mathcal{L}-\l_3+p au_1^{p-1}(\l_3),\mathcal{B},\O]=0.
\]
\end{enumerate}
The option (\ref{th3.8.sub.item2}) cannot occur, because $\l_2<\l_0$ and $(\l_0,0)$ is the unique bifurcation point from
$(\l,0)$ to positive solutions of \eqref{3.1}. By Theorem \ref{th3.6}, the option (\ref{th3.8.sub.item3}) cannot occur neither. Therefore, (\ref{th3.8.sub.item1}) occurs. By differentiating $\mathfrak{F}(\l,u(\l))=0$ with respect to $\l$ it becomes apparent that
\[
  \left( \mathcal{L}-\l +pau^{p-1}(\l)\right)u'(\l)=u(\l)
\]
for all $\l \leq \l_1$ and hence,
\[
  u'(\l)= \left( \mathcal{L}-\l +pau^{p-1}(\l)\right)^{-1}u(\l) \gg 0.
\]
Therefore, the map $\l \mapsto u(\l)$ is point-wise increasing. In particular, there exists a constant
$\tau>0$ such that
\begin{equation}
\label{3.27}
  \|u(\l)\|_{\mathcal{C}(\bar\O)}\leq \tau  \quad \hbox{for all}\;\; \l \leq \l_1.
\end{equation}
Now, consider the change of variables given by
\begin{equation*}
u(\lambda)=|\lambda|^\frac{1}{p-1}v(\lambda),\quad\hbox{for all}\;\;\lambda\le \l_2:=\min\{0,\lambda_1\}.
\end{equation*}
Then, owing to \eqref{3.27}, we have that
\begin{equation}
\label{3.28}
(v(\lambda))^{p-1}\le \frac{\tau^{p-1}}{|\lambda|}\quad\hbox{for all}\;\;\lambda\le \l_2.
\end{equation}
Moreover, for every $\lambda<\l_2$, $v(\lambda)$ is a positive solution of
\begin{equation}
\label{3.29}
\frac{1}{|\lambda|}\mathcal{L} v=-v-a(x)v|v|^{p-1}\quad \hbox{in}\;\; \Omega.
\end{equation}
Let denote by $\varphi_0\gg 0$ the principal eigenfunction associated with $\sigma_0=\sigma[\mathcal{L},\mathcal{B},\Omega]$ normalized so that $\int_\O \v_0^2=1$.
Then, multiplying \eqref{3.29} by $\varphi_0$  and integrating in $\O$ yields
\begin{equation*}
\frac{\sigma_0}{|\lambda|}\int_\Omega \varphi_0 v(\lambda)=-\int_\Omega v(\lambda)\varphi_0-\int_\Omega a\varphi_0 (v(\lambda))^p.
\end{equation*}
Thus, thanks to \eqref{3.28}, we find that
\begin{equation*}
\left( \frac{\sigma_0}{|\lambda|}+1 \right)\int_\Omega \varphi_0 v(\lambda)\le \frac{1}{|\lambda|}\|a\|_\infty \tau^{p-1} \int_\Omega\varphi_0 v(\lambda).
\end{equation*}
Hence, for every $\l<\l_2$,
\begin{equation*}
\frac{\sigma_0}{|\lambda|}+1 \le \frac{1}{|\lambda|}\|a(x)\|_\infty \tau^{p-1},
\end{equation*}
which is impossible. This contradiction ends the proof of Part i).
\par
Part \ref{th3.8.item2}) is an immediate consequence of Proposition \ref{Prop3.2}, Theorem \ref{th3.5}  and Part \ref{th3.8.item1}). Indeed, if $\mathcal{D}>0$ then Proposition \ref{Prop3.2} provides us with a supercritical bifurcation, from the trivial branch, of a curve of linearly stable positive solutions, whereas, if $\mathcal{D}\le 0$, then Theorem \ref{th3.5} restricts to $\lambda\in(-\infty,\sigma_0)$ the values of the parameter for which the problem \eqref{3.1} admits positive solutions. By Part \ref{th3.8.item1}), these solutions are linearly unstable.
Therefore, \eqref{3.1} cannot admit a linearly stable positive solution if $\mathcal{D}\le 0$.
\par
To prove the uniqueness in Part \ref{th3.8.item3}), let $\lambda_0>\sigma_0$ be for which \eqref{3.1} has two linearly stable positive solutions of \eqref{3.1}, $u_0$ and $v_0$, $u_0\neq v_0$. As a consequence of the implicit function theorem, much like in the proof of Part \ref{th3.8.item1}), we can get two different curves of solutions, $u_0(\lambda)$ and $v_0(\lambda)$, for every $\lambda$ in a neighborhood of $\lambda_0$. By a rather standard global continuation argument, each of these curves should satisfy some of the alternatives, (\ref{th3.8.sub.item1}), (\ref{th3.8.sub.item2}) or (\ref{th3.8.sub.item3}), as in the proof of Part \ref{th3.8.item1}). Moreover, as all these solutions are non-degenerate, $u_0(\l)\neq v_0(\l)$, as soon as some of these solutions is linearly stable. As a consequence of Part \ref{th3.8.item1}), option (\ref{th3.8.sub.item1}) cannot occur. Similarly, by Theorem \ref{th3.6}, these curves cannot satisfy the option (\ref{th3.8.sub.item3}) neither.  Therefore,
$u_0(\lambda)$ and $v_0(\lambda)$ should bifurcate supercritically from the trivial branch at $(\lambda,u)=(\sigma_0,0)$, which contradicts the local uniqueness at $(\s_0,0)$ obtained as an application of Theorem \ref{th3.1}.
\par
The fact that the minimal positive solution of \eqref{3.1} is linearly stable, or linearly neutrally stable, for any $\l >\s_0$ where it admits a positive solution can be easily inferred by adapting the argument given in the proof of \cite[Th. 9.12]{LG15}, which was adapted from \cite{LGMMT} and H. Amann \cite{Am76}.
\par
The proof of Part \ref{th3.8.item4}) follows similar patterns as the proof of \cite[Th. 9.13]{LG15}. Fix $\lambda=\lambda_0>\s_0$. Under this assumption, $\underline{u}_\e:=\varepsilon\varphi_0$,
where $\v_0\gg 0$ stands for the normalized principal eigenfunction associated to $\s_0$,
is a subsolution of \eqref{3.1} for sufficiently small $\varepsilon>0$. Let $w_\varepsilon$ denote the unique solution of
\begin{equation*}
\left\{
\begin{array}{ll}
\frac{\partial w}{\partial t}+\mathcal{L} w=\lambda_0 w-a(x)w^p & \hbox{in}\;\;\Omega\times (0,+\infty),\\
\mathcal{B}w=0 & \mbox{on}\;\;\partial\Omega\times (0+\infty),\\
w(\cdot,0)=\underline{u}_\e & \mbox{in}\;\;\Omega.
\end{array}
\right.
\end{equation*}
Since this equation preserves the ordering, we find that
\[
  w_\varepsilon(t)\le u_0\quad \hbox{in}\;\; \Omega\;\;\hbox{for all}\;\; t>0.
\]
Moreover, thanks to the abstract theory of D. Sattinger \cite{Sat}, $w_\varepsilon(t)$ is increasing and globally defined in time. Actually, owing to the main theorem of M. Langlais \& D. Phillips \cite{LP},
the limit
\begin{equation*}
  w_\varepsilon^*:=\lim_{t\to+\infty} w_\varepsilon(t)
\end{equation*}
is well defined and provides us with a positive solution of \eqref{3.1} for  $\lambda=\lambda_0$. Moreover, by construction,
\begin{equation*}
    w_{\varepsilon_1}^*\le w_{\varepsilon_2}^*\le u_0\quad \mbox{if}\;\;0<\varepsilon_1<\varepsilon_2.
\end{equation*}
Thus, since $\varepsilon$ can be chosen sufficiently small so that $\underline{u}_\varepsilon$ lies below any given positive  solution of \eqref{3.1}, the limit
\begin{equation*}
 w^*_{\mathrm{min}}:=\min_{\varepsilon>0} w^*_\e
\end{equation*}
is the minimal positive solution of \eqref{3.1}, because we are assuming that $\lambda_0 > \sigma_0$ and $\sigma_0$ is the unique bifurcation point to positive solutions from $u=0$. Note that $w^*_{\mathrm{min}}$ stays below any positive solution of \eqref{3.1}, by construction.
\par
It remains to prove that the minimal positive solution, $w^*_{\mathrm{min}}$, is either linearly stable or neutrally stable. To show it we will argue by contradiction assuming that
\begin{equation*}
\sigma_{\mathrm{min}}:=\sigma[\mathcal{L}-\lambda_0+a(x)p (w^*_{\mathrm{min}})^{p-1},\mathcal{B};\Omega]<0.
\end{equation*}
Let $\varphi_{\mathrm{min}}\gg 0$ be a positive eigenfunction associated with
$\sigma_{\mathrm{min}}$. Then, for sufficiently small $\varepsilon>0$,
\begin{equation*}
\overline{u}_\varepsilon:=w_{\mathrm{min}}^*-\varepsilon \varphi_{\mathrm{min}}
\end{equation*}
provides us with a supersolution of \eqref{3.1} such that
\[
  \overline{u}_\varepsilon \ll w_{\mathrm{min}}^*.
\]
Indeed, as $\e\da 0$, we find that
\begin{align*}
\mathcal{L}\overline{u}_\varepsilon &= \l_0 w_{\mathrm{min}}^* - a(x)(w^*_{\mathrm{min}})^{p}-\e \mathcal{L} \varphi_{\mathrm{min}}\\ & =
\l_0 w_{\mathrm{min}}^* - a(x)(w^*_{\mathrm{min}})^{p}-\e \left( \sigma_{\mathrm{min}} \varphi_{\mathrm{min}}+\lambda_0 \varphi_{\mathrm{min}} -a(x)p (w^*_{\mathrm{min}})^{p-1}\varphi_{\mathrm{min}}\right) \\
&= \l_0 \overline{u}_\varepsilon - a(x)\left( (w^*_{\mathrm{min}})^{p} +\varepsilon p (w^*_{\mathrm{min}})^{p-1}\varphi_{\mathrm{min}}\right)-\varepsilon\sigma_{\mathrm{min}} \varphi_{\mathrm{min}}\\
&= \l_0 \overline{u}_\varepsilon - a(x)\overline{u}_\varepsilon^p - a(x)\left( (w^*_{\mathrm{min}})^{p} +\varepsilon p (w^*_{\mathrm{min}})^{p-1}\varphi_{\mathrm{min}}-(w_{\mathrm{min}}^*-\varepsilon \varphi_{\mathrm{min}})^p\right)-\varepsilon\sigma_{\mathrm{min}} \varphi_{\mathrm{min}}\\
&= \l_0 \overline{u}_\varepsilon - a(x)\overline{u}_\varepsilon^p - a(x) o(\varepsilon)-\varepsilon\sigma_{\mathrm{min}} \varphi_{\mathrm{min}}.
\end{align*}
Therefore, since $\sigma_{\mathrm{min}}<0$, for sufficiently small $\varepsilon>0$,
\begin{equation*}
  \mathcal{L}\overline{u}_\varepsilon > \l_0 \overline{u}_\varepsilon - a(x)\overline{u}_\varepsilon^p \quad\mbox{in}\;\;\Omega.
\end{equation*}
Since \eqref{3.1} admits arbitrarily small positive subsolutions, it becomes apparent that
\eqref{3.1} possesses a positive solution below $\overline{u}_\varepsilon$ and, hence, below $w^*_{\mathrm{min}}$, which is impossible. This ends the proof.
\end{proof}

\section{Global structure of the set of l.s. positive solutions  when $f(u)=u^p$, $p\geq 2$}
\label{sec-4}

\noindent The next result provides us with the global structure of the set of linearly stable and linearly neutrally stable positive solutions of \eqref{3.1} in the special case when $f(u)=u^{p}$ for some $p\geq 2$.

\begin{theorem}
\label{th4.1}  Suppose that $a(x)$ changes sign in $\O$, $f(u)=u^p$, $u\geq 0$, with $p\geq 2$, and that  \eqref{3.22} holds. Then, the supremum of the set of $\mu>\s_0$ for which  \eqref{3.1} possesses a positive
solution for each $\l \in (\s_0,\mu)$, $\l_*$, satisfies $\l_*\in (\s_0,+\infty)$.
Moreover, the set of linearly stable positive solutions of \eqref{3.1} consists of a $\mc{C}^1$ point-wise increasing curve
\[
  \mathfrak{C}_+:=\{(\l,u(\l))\;:\;\; \l \in (\s_0,\l_*)\}.
\]
Furthermore, some of the next excluding options occurs:
\begin{enumerate}[{\rm i)}]
\item\label{th4.1.item1} Either $\{u(\l)\}_{\l \in (\s_0,\l_*)}$ is bounded in $\mathcal{C}(\bar\O)$, and then
\[
  u_* := \lim_{\l \ua \l_*} u(\l)
\]
is a linearly neutrally stable positive solution of \eqref{3.1} at $\l=\l_*$, or

\item\label{th4.1.item2}  $\lim_{\l \ua \l_*} \|u(\l)\|_{\mathcal{C}(\bar \O)}=+\infty$.
\end{enumerate}
In both cases, \eqref{3.1} cannot admit any further positive solution for $\l >\l_*$.
\end{theorem}
\begin{proof}
The existence of $\l_*$, as well as the fact that $\lambda_*\in(\sigma_0,+\infty]$,
is a direct consequence of Theorem \ref{th3.1} and Proposition \ref{Prop3.2}. The fact that $\lambda_*$ is finite follows with the next argument. Let $(\lambda,u)$ be a positive solution of \eqref{3.1}. Then,
\begin{equation*}
  \l= \sigma[\mathcal{L}+a(x)u^{p-1},\mathcal{B},\Omega].
\end{equation*}
Moreover, since $a(x)$ changes sign, there exists a ball, $B\subset\O$, such that $a(x)<0$ for all $x\in B$. Thus, thanks to \cite[Cor. 3.6]{CCLG},
\begin{equation*}
\l=\sigma[\mathcal{L}+a(x)u^{p-1},\mathcal{B},\Omega]< \sigma[\mathcal{L}+a(x)u^{p-1},\mathcal{D},B]<
\sigma[\mathcal{L},\mathcal{D},B].
\end{equation*}
Therefore, $\lambda_*\leq \sigma[\mathcal{L},\mathcal{D},B]$. In particular, $\l_*\in (\s_0,\infty)$.
\par
Thanks to Proposition \ref{Prop3.2}, the solutions bifurcating from $u=0$ at $\s_0$ are linearly stable.
Therefore, thanks to implicit function theorem applied to the integral equation associated to
\eqref{3.1}, they consist of a $\mc{C}^1$ curve parameterized by $\l$ which is point-wise strictly increasing. By a global continuation argument involving the implicit function theorem, this curve can be globally parameterized by $\l$ in the form $(\lambda,u(\lambda))$, with $\lambda\in(\sigma_0,\lambda_\mathrm{max})$, for some maximal $\lambda_{\max}\in (\sigma_0,+\infty)$.  Necessarily, $\lambda_{\max}\le \lambda_*$. Moreover, thanks to their linearized stability, $u'(\lambda)\gg 0$ for all $\lambda\in (\sigma_0,\lambda_{\max})$. By construction, the curve \begin{equation*}
\mathfrak{C}_+:=\{(\lambda,u(\lambda))\;:\;\lambda\in (\sigma_0,\lambda_{\max})\},
\end{equation*}
provides us with the maximal set of linearly stable positive solutions of \eqref{3.1} that bifurcates from $u=0$. By the monotonicity of the solution on this curve, either
\begin{equation*}
\lim_{\lambda\uparrow\lambda_{\max}}\|u(\lambda)\|_{\mathcal{C}(\bar\Omega)}=+\infty,
\end{equation*}
much like illustrated in the right picture of Figure \ref{Fig3},
or $\{u(\lambda)\}_{\lambda\in(\sigma_0,\lambda_{\max})}$ stays bounded. In the latest case, by a rather standard compactness argument, it is easily seen that
\begin{equation*}
u_{\max} := \lim_{\lambda \uparrow \lambda_{\max}} u(\lambda)
\end{equation*}
is a solution of \eqref{3.1} for $\lambda=\lambda_{\mathrm{max}}$. By the continuity of the principal eigenvalue with respect to the potential, $u_{\max}$ is either linearly stable or neutrally stable. In the former case, the implicit function theorem would allow us to continue the curve
$\mf{C}_+$ beyond  $\lambda_{\max}$, which contradicts the maximality of $\l_\mathrm{max}$. Thus,  $u_{\max}$ is  neutrally stable and, due to Theorem \ref{th3.6}, the set of solutions surrounding it consists of a subcritical quadratic turning point, as illustrated by
the left picture of Figure \ref{Fig3}.
\par
Lastly, the proof that  $\lambda_{\max}=\lambda_*$ proceeds by contradiction. Suppose $\l_{\max}<\lambda_*$. Then, \eqref{3.1} admits a positive solution for some $\l_1>\lambda_{\max}$. By Theorem \ref{th3.8} \ref{th3.8.item4}), \eqref{3.1} also admits a minimal positive solution, which is either linearly stable, and hence part of an increasing curve of solutions, or neutrally stable, and hence a subcritical quadratic turning point. By a backwards global
continuation argument in $\l$ starting at $\l_1$, one can  construct an analytic curve of linearly stable positive solutions up to reach $(\lambda,u)=(\sigma_0,0)$, which contradicts
the definition of $\l_\mathrm{max}$ and ends the proof.
\end{proof}

Figure \ref{Fig3} shows two admissible global bifurcation diagrams for each of the cases \ref{th4.1.item1}) and \ref{th4.1.item2})
discussed by Theorem \ref{th4.1}. Some general conditions ensuring that the option
\ref{th4.1.item1}) of Theorem \ref{th4.1} occurs can be formulated from the a priori bounds of H. Amann \& J. L\'opez-G\'omez \cite{ALG}. The problem of ascertaining weather or not each
of these options can occur has not been solved yet.

\begin{figure}[h!]
\centering
\includegraphics[scale=0.3]{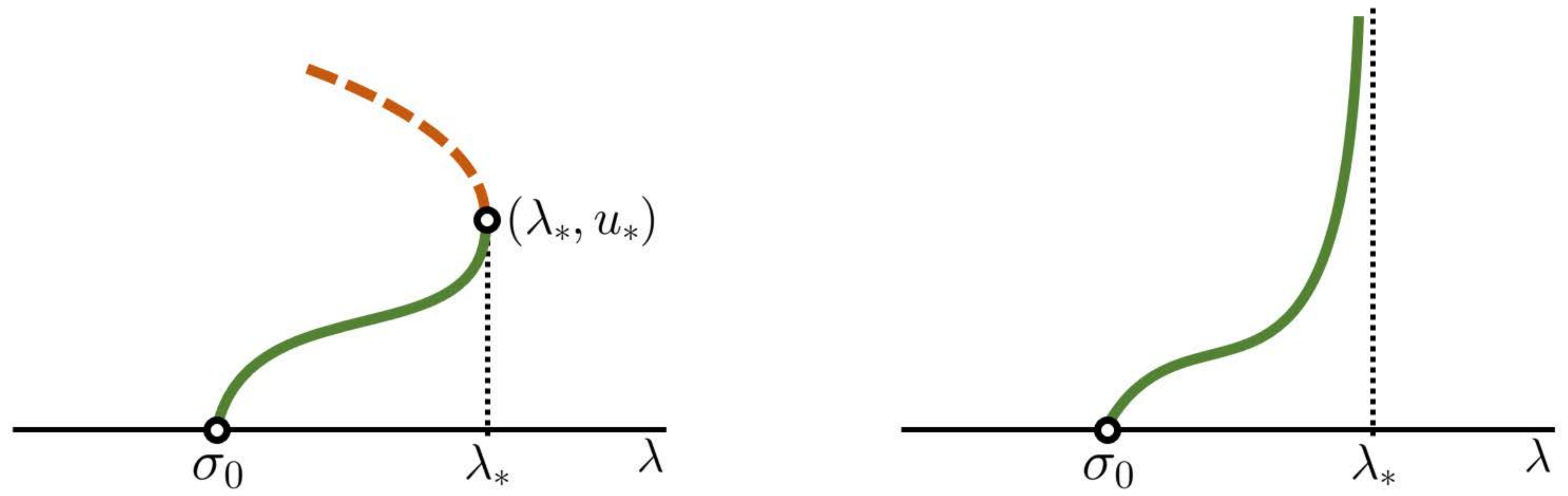}
\caption{Two admissible global bifurcation diagrams of linearly stable positive solutions.}
\label{Fig3}
\end{figure}

\section{Uniqueness of coexistence states in a class of symbiotic systems}
\label{sec-5}

\noindent This section derives from Theorem \ref{th2.1} a general sufficient condition so that
every coexistence state of a general class of diffusive Lotka--Volterra symbiotic systems is linearly stable, which entails their uniqueness.  Our results sharpen, very substantially, some of the previous findings of M. Delgado et al \cite{DLGS}. Throughout this section, we assume that, for every $i=1,2$, $\mathcal{L}_i$ is an uniformly elliptic second order selfadjoint operator of the form
\begin{equation*}
\mathcal{L}_i:=-\mathrm{div}(A_i\nabla \cdot)+C_i,
\end{equation*}
where $A_i\in\mathcal{M}_N^\mathrm{sym}(\mathcal{C}^1(\bar\Omega))$ is definite positive, and $C_i\in\mathcal{C}(\bar\Omega)$. For each $i=1, 2$, let
\begin{equation*}
\partial\Omega=\Gamma_{\mathcal{D}}^{i}\cup\Gamma_{\mathcal{R}}^{i}
\end{equation*}
be the decomposition of $\partial\Omega$ in disjoint closed and open subsets associated to the boundary operator
\begin{equation*}
\mathcal{B}_i:=
\left\{
\begin{array}{ll}
\mathcal{D}_i=Id & \hbox{on}\;\;\Gamma_{\mathcal{D}}^{i},\\[5pt]
\mathcal{R}_i=\langle \nabla \cdot,A_i{\bf n}\rangle+\beta_i & \hbox{on}\;\;\Gamma_{\mathcal{R}}^{i},
\end{array}
\right.
\end{equation*}
where $\beta_i\in\mathcal{C}(\Gamma_{\mathcal{R}}^{i})$, $i\in\{1,2\}$. Then, the symbiotic Lotka--Volterra reaction-diffusion prototype model can be written down as follows
\begin{equation}
\label{5.1}
\left\{
\begin{array}{ll}
d_1\mathcal{L}_1 u=\lambda u-au^2+ b uv & \hbox{in}\;\;\Omega,\\
d_2\mathcal{L}_2 v=\mu v-dv^2+ c uv & \hbox{in}\;\;\Omega,\\
\mathcal{B}_1 u=\mathcal{B}_2 v=0 & \hbox{on}\;\;\partial\Omega,
\end{array}
\right.
\end{equation}
where $d_1$, $d_2>0$ measure the strength of the diffusivities of the species $u$ and $v$, $\lambda$, $\mu\in\mathcal{C}(\bar\Omega)$ stand for the growth, or decay, rates of the species, $a$, $d\in\mathcal{C}(\bar\Omega;(0,+\infty))$ are the intra-specific rates of $u$ and $v$, respectively,
and $b$, $c\in\mathcal{C}(\bar\O;(0,+\infty))$ represent the symbiotic effects between both populations.
It is said that \emph{low symbiotic effects} occur whenever
\begin{equation*}
  bc\lneq ad\quad \hbox{in}\;\; \Omega.
\end{equation*}
In such case, setting
\begin{equation}
\label{5.2}
F_\pm(t):=\frac{1}{8}\left[ 27-18t - t^2\pm(9-t)^{3/2}(1-t)^{1/2}\right],\qquad
t \in[0,1],
\end{equation}
the next result holds.

\begin{theorem}
\label{th5.1}
Suppose that $\kappa:=\frac{bc}{ad}\lneq 1$ in $\Omega$, and that
\begin{equation}
\label{5.3}
\max_{\bar\Omega} \left( \frac{a d^2}{c^3}F_-(\kappa)\right) \le \min_{\bar\Omega} \left( \frac{a d^2}{c^3}F_+(\kappa)\right).
\end{equation}
Then,  every coexistence state of \eqref{5.1} is linearly stable. Moreover, \eqref{5.3} holds  provided some of the quotients $\frac{ad^2}{c^3}$, or $\frac{bd}{c^2}$, or $\frac{b^2}{ac}$, or $\frac{b^3}{a^2d}$, is constant on $\Omega$. Note that
\[
  \frac{bd}{c^2} = \frac{ad^2}{c^3}\kappa,\qquad \frac{b^2}{ac}=\frac{bd}{c^2}\kappa=\frac{ad^2}{c^3}\kappa^2,\qquad
  \frac{b^3}{a^2d}=\frac{b^2}{ac}\kappa=\frac{bd}{c^2}\kappa^2=\frac{ad^2}{c^3}\kappa^3.
\]
\end{theorem}
\begin{proof}
Let $(u,v)$ be  a coexistence state of \eqref{5.1}. Then, the linearized stability of $(u,v)$ as
an steady-state solution of the parabolic counterpart of \eqref{5.1} is given by the signs of the eigenvalues of the linear eigenvalue problem
\begin{equation}
\label{5.4}
\left\{
\begin{array}{ll}
[d_1\mathcal{L}_1 -\lambda +2au- b v]\varphi- b u\psi=\sigma\varphi & \hbox{in}\;\;\Omega,\\[0ex]
[d_2\mathcal{L}_2 -\mu +2dv -  c u]\psi- c v\varphi=\sigma \psi & \hbox{in}\;\;\Omega,\\[0ex]
\mathcal{B}_1 \varphi=\mathcal{B}_2 \psi=0 & \hbox{on}\;\;\partial\Omega.
\end{array}
\right.
\end{equation}
Since $b(x)u(x)>0$ and $c(x)v(x)>0$ for all $x\in\O$, it follows from  \cite[Th. 1.3]{ALLG} that \eqref{5.4} possesses a unique principal eigenvalue, $\sigma_0$, associated with it there is an eigenfunction $(\v,\psi)$ with $\v\gg 0$ and $\psi\gg 0$.  This result extends the findings of
\cite{LGMM} to cover our general setting even in the context of periodic-parabolic problems.
\par
Particularizing \eqref{5.4} at $\s=\s_0$, multiplying the first equation of \eqref{5.4} by $u$ and using the $u$-equation of \eqref{5.1} yields
\begin{equation}
\begin{split}
\label{5.5}
\sigma_0 u \varphi
&=u d_1\mathcal{L}_1\varphi-\varphi(\lambda-a u+ b v)u+ u^2(a\varphi-b\psi)\\
&=d_1(u\mathcal{L}_1\varphi-\varphi\mathcal{L}_1 u)+u^2(a\varphi-b\psi).
\end{split}
\end{equation}
Similarly, multiplying the second equation of \eqref{5.4} by $v$ and using the $v$-equation
of \eqref{5.1}, it is easily seen that
\begin{equation}
\begin{split}
\label{5.6}
\sigma_0 v \psi
&=v d_2\mathcal{L}_2\psi-\psi(\mu-d v+ c u)v+ v^2(d\psi-c\varphi)\\
&=d_2(v\mathcal{L}_2\psi-\psi\mathcal{L}_2 v)+v^2(d\psi-c\varphi).
\end{split}
\end{equation}
Multiplying \eqref{5.5} and \eqref{5.6} by $\frac{\varphi^2}{u^2}$ and $\frac{\psi^2}{v^2}$, respectively, and integrating in $\Omega$ it becomes apparent that
\begin{equation*}
\sigma_0 \int_\Omega \frac{\varphi^3}{u} =d_1\int_\Omega \left(\frac{\varphi}{u}\right)^2(u\mathcal{L}_1\varphi-\varphi\mathcal{L}_1 u)+\int_\Omega \varphi^2(a\varphi-b\psi),
\end{equation*}
\begin{equation*}
\sigma_0 \int_\Omega \frac{\psi^3}{v} =d_2\int_\Omega \left(\frac{\psi}{v}\right)^2(v\mathcal{L}_2 \psi-\psi\mathcal{L}_2 v)+\int_\Omega \psi^2(d \psi-c\varphi).
\end{equation*}
On the other hand, applying \eqref{2.2} with $g(t)=t^2$ and using the uniform ellipticity of $\mathcal{L}_1$ and $\mathcal{L}_2$ provides us with the estimates
\begin{equation}
\label{5.7}
\begin{split}
\int_\Omega \left(\frac{\varphi}{u}\right)^2(u\mathcal{L}_1\varphi-\varphi\mathcal{L}_1 u)
&=\int_\Omega 2u\varphi\langle \nabla \frac{\varphi}{u},A_1\nabla\frac{\varphi}{u}\rangle - \int_{\partial\Omega}\left(\frac{\varphi}{u}\right)^2 [\mathcal{D} u\mathcal{R}\varphi-\mathcal{D}\varphi\mathcal{R} u]\\ &=\int_\Omega 2u\varphi\langle \nabla \frac{\varphi}{u},A_1\nabla\frac{\varphi}{u}\rangle\ge 0,
\end{split}
\end{equation}
and
\begin{equation}
\label{5.8}
\begin{split}
\int_\Omega \left(\frac{\psi}{v}\right)^2(v\mathcal{L}_2\psi-\psi\mathcal{L}_2 v)
&=\int_\Omega 2v\psi\langle \nabla \frac{\psi}{v},A_2\nabla\frac{\psi}{v}\rangle - \int_{\partial\Omega}[\mathcal{D} v\mathcal{R}\psi-\mathcal{D}\psi\mathcal{R} v]\\
&=\int_\Omega 2v\psi\langle \nabla \frac{\psi}{v},A_2\nabla\frac{\psi}{v}\rangle\ge 0,
\end{split}
\end{equation}
where we have used that
\begin{align*}
  \mathcal{D}u& =\mathcal{D}\varphi=0\quad \hbox{on}\;\; \Gamma_{\mathcal{D}}^1,\qquad
  \mathcal{R}u=\mathcal{R}\varphi=0\quad \hbox{on} \;\; \Gamma_{\mathcal{R}}^1,\\
  \mathcal{D}v & =\mathcal{D}\psi=0\quad \hbox{on} \;\; \Gamma_{\mathcal{D}}^2,\qquad \mathcal{R}v=\mathcal{R}\psi=0\quad \hbox{on} \;\; \Gamma_{\mathcal{R}}^2,
\end{align*}
as well as the fact that $u,v,\phi,\psi\gg 0$ in $\Omega$. Hence,
\begin{equation*}
\sigma_0 \int_\Omega \frac{\varphi^3}{u} \ge \int_\Omega \varphi^2(a\varphi-b\psi)\quad\hbox{and}\quad
\sigma_0 \int_\Omega \frac{\psi^3}{v} \ge -\int_\Omega \psi^2(c\varphi-d\psi).
\end{equation*}
Therefore, for every positive constant $\xi>0$, we find that
\begin{equation}
\label{5.9}
\sigma_0 \left(\int_\Omega \frac{\varphi^3}{u}+\xi \int_\Omega \frac{\psi^3}{v}\right)
\ge \int_\Omega \left[\varphi^2(a\varphi-b\psi)-\xi\psi^2(c\varphi-d\psi)\right].
\end{equation}
Next, we will ascertain the values of $\xi>0$ for which
\begin{equation*}
\varphi^2(a\varphi-b\psi)-\xi\psi^2(c\varphi-d\psi)\geq 0\quad \hbox{in}\;\;\Omega.
\end{equation*}
Dividing by $\psi^3$ and setting $y:=\varphi/\psi$, it suffices to show that, for every $y\ge 0$, $\xi$ satisfies
\begin{equation*}
y^2(a y-b)-\xi(c y-d)\geq 0\quad \hbox{in}\;\;\Omega.
\end{equation*}
Further, setting $z=\frac{c}{d}y\ge 0$ and dividing by $d$ yields
\begin{equation}
\label{5.10}
\frac{a d^2}{c^3}z^2\left(z-\frac{b c}{a d}\right)\geq \xi(z-1)\quad \hbox{in}\;\;\Omega.
\end{equation}
Note that, since $\kappa=\frac{bc}{ad} \leq 1$ in $\bar \Omega$, \eqref{5.10} holds if $z\in\{0,1\}$ for every $\xi > 0$. Hence, the inequality \eqref{5.10} can be split into
\begin{equation*}
\frac{a d^2}{c^3}z^2\frac{z-\kappa}{z-1}\geq \xi\quad\hbox{in}\;\;\Omega\;\;\hbox{for all}\;\; z>1,
\end{equation*}
and
\begin{equation*}
\frac{a d^2}{c^3}z^2\frac{z-\kappa}{z-1}\leq \xi\quad\hbox{in}\;\;\Omega\;\;\hbox{for all}\;\; z\in(0,1).
\end{equation*}
Therefore, in order to get \eqref{5.10} for all $z\ge 0$ and $x\in\O$ it suffices to make sure that the constant $\xi$ satisfies
\begin{equation}
\label{5.11}
\frac{a d^2}{c^3}\sup_{0< z<1} \left(z^2\frac{z-\kappa}{z-1}\right)\leq \xi\leq \frac{a d^2}{c^3} \inf_{z>1} \left(z^2\frac{z-\kappa}{z-1}\right) \quad \hbox{in}\;\;\Omega.
\end{equation}
Subsequently we consider the function
\begin{equation*}
    F(z;k):=z^2\frac{z-k}{z-1},\qquad z >0,\;\; z\neq 1,\quad k \in (0,1].
\end{equation*}
Note that here $k$ is a constant, not a function, like $\kappa:=\kappa(x)$.
By differentiating with respect to $z$ yields
\begin{equation*}
     \frac{d F}{d z}(z;k)=\frac{(3z^2-2k z)(z-1)-(z^3-k z^2)}{(z-1)^2}=
     \frac{2 z^3-(3+k)z^2+2 k z}{(z-1)^2}.
\end{equation*}
Thus, the critical points of $F(\cdot,k)$ are given by the roots of
\begin{equation*}
  (2z^2-(3+k)z+2k)z=0,
\end{equation*}
which are $z=0$ plus
\begin{equation*}
   z_\pm(k):=\frac{3+k\pm \sqrt{(3+k)^2-16k}}{4}=\frac{3+k\pm \sqrt{(9-k)(1-k)}}{4}.
\end{equation*}

\begin{figure}[h!]
\centering
\includegraphics[scale=0.47]{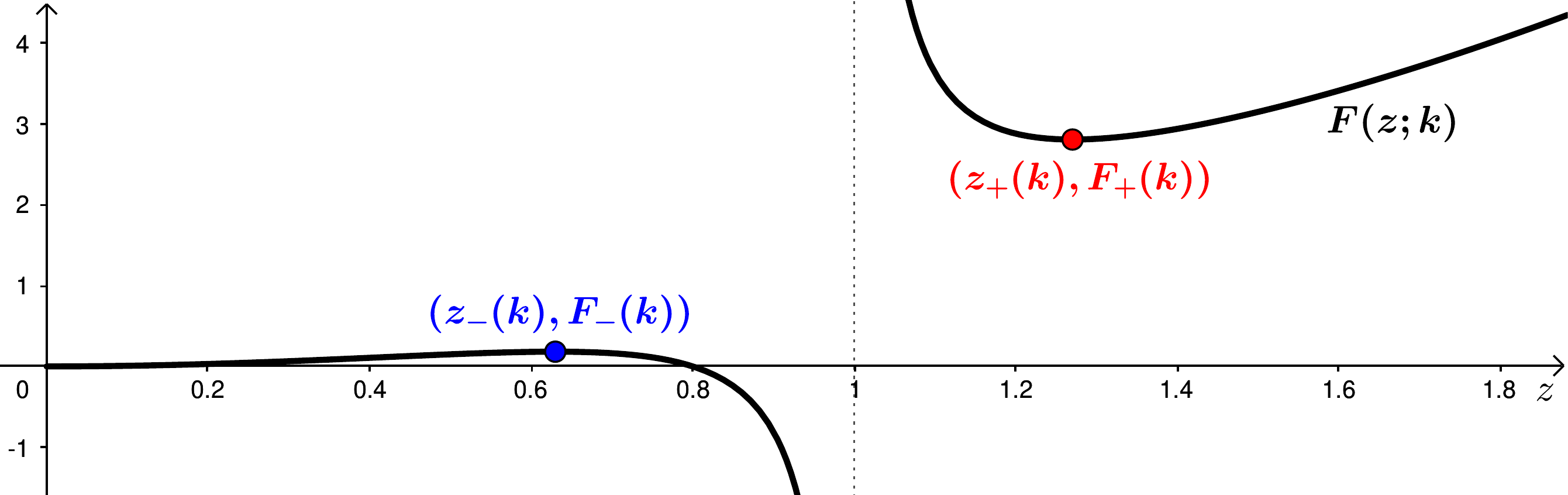}
\caption{A plot of the function $F$ for $k=0.8$}
\label{Fig4}
\end{figure}

It is straightforward to check that $F(\cdot,k)$ has local minimum at $z_{+}(k)\in(1,+\infty)$, which is global in $z\in(1+\infty)$,
while it has a local maximum at $z_{-}(k)\in (0,1)$, which is global in $z\in[0,1)$. Moreover,
\begin{align*}
F_{\pm}(k)&:=F(z_{\pm}(k);k)= z_\pm^2(k)\frac{z_\pm(k)-k}{z_\pm(k)-1} \\[1ex]
&=\frac{\left((3+k)^2-8k\pm(3+k)\sqrt{(9-k)(1-k)}\right)
\left(3-3k\pm\sqrt{(9-k)(1-k)}\right)}{8\left(-1+k\pm\sqrt{(9-k)(1-k)}\right)} \\[1ex]
&=\frac{\left((3+k)^2-8k\pm(3+k)\sqrt{(9-k)(1-k)}\right)\left(3\sqrt{1-k}\pm\sqrt{9-k}  \right)}{8\left(-\sqrt{1-k}\pm\sqrt{9-k}\right)}
\end{align*}
and, rationalizing and simplifying, we find that
\begin{align*}
F_{\pm}(k)&=\frac{\left((3+k)^2-8k\pm(3+k)\sqrt{(9-k)(1-k)}\right)\left(3-k\pm\sqrt{(9-k)(1-k)}  \right)}{16}  \\
&=\frac{1}{8} \left[ -k^2-18k+27\pm(9-k)\sqrt{(9-k)(1-k)}\right].
\end{align*}
Hence, the condition \eqref{5.11} can be rewritten, equivalently,  as
\begin{equation}
\label{5.12}
\frac{a d^2}{c^3}F_-(\kappa)\leq \xi\leq \frac{a d^2}{c^3} F_+(\kappa) \quad \hbox{in}\;\;\Omega.
\end{equation}
Therefore, if there exists a constant $\xi>0$ satisfying \eqref{5.12}, then \eqref{5.9} yields $\sigma_0\ge 0$. Naturally, the condition \eqref{5.3} guarantees the existence of $\xi>0$ such that \eqref{5.12} holds. Let us check that actually $\sigma_0>0$. Indeed, arguing by contradiction, assume that $\sigma_0=0$. We claim that, in such case,
\[
  \int_\Omega 2u\varphi\langle \nabla \frac{\varphi}{u},A_1\nabla\frac{\varphi}{u}\rangle =
  \int_\Omega 2v\psi\langle \nabla \frac{\psi}{v},A_2\nabla\frac{\psi}{v}\rangle =0.
\]
Indeed, if either
\[
  \int_\Omega 2u\varphi\langle \nabla \frac{\varphi}{u},A_1\nabla\frac{\varphi}{u}\rangle >0, \quad
  \hbox{or}\;\;  \int_\Omega 2v\psi\langle \nabla \frac{\psi}{v},A_2\nabla\frac{\psi}{v}\rangle >0,
\]
then,
\begin{equation*}
\sigma_0 \int_\Omega \frac{\varphi^3}{u} > \int_\Omega \varphi^2(a\varphi-b\psi),\quad\hbox{or}\quad
\sigma_0 \int_\Omega \frac{\psi^3}{v} >  -\int_\Omega \psi^2(c\varphi-d\psi),
\end{equation*}
respectively and, hence, by the choice of $\xi$, we find that
\begin{equation*}
0=\sigma_0 \left(\int_\Omega \frac{\varphi^3}{u}+\xi \int_\Omega \frac{\psi^3}{v}\right)
> \int_\Omega \left[\varphi^2(a\varphi-b\psi)-\xi\psi^2(c\varphi-d\psi)\right]\geq 0,
\end{equation*}
which is a contradiction. Thus, since $A_1$ and $A_2$ are definite positive,
it becomes apparent that $\varphi$ and $\psi$ are proportional to $u$ and $v$, respectively.
Consequently, going back to \eqref{5.5} and \eqref{5.6}, we find that
\begin{equation*}
0=\varphi^2(a\varphi-b\psi)\quad\hbox{and}\quad 0=\psi^2(d\psi-c\varphi)\quad\hbox{in}\;\;\Omega,
\end{equation*}
or, equivalently,
\begin{equation*}
0= a\varphi-b\psi=d\psi-c\varphi\quad\hbox{in}\;\;\Omega,
\end{equation*}
which implies $ad=bc$ in $\O$ and contradicts our assumption that $\kappa =\frac{bc}{ad}\lneq 1$ in $\Omega$. This contradiction shows that $\s_0>0$ and ends the proof of the first assertion of the theorem.
\par
The validity of the second assertion of the theorem can be easily shown from the fact that
\begin{equation*}
\frac{bd}{c^2}=\kappa\frac{ad^2}{c^3},\qquad  \frac{b^2}{ac}=\kappa^{2}\frac{ad^2}{c^3},\qquad
\frac{b^3}{a^2d}=\kappa^{3}\frac{ad^2}{c^3},
\end{equation*}
taking into account that
\begin{equation*}
F_-(k)\le k^3\le k^2 \le k \le 1 \le F_+(k)\quad \hbox{for all}\;\; k\in[0,1].
\end{equation*}

\begin{figure}[h!]
\centering
\includegraphics[scale=0.5]{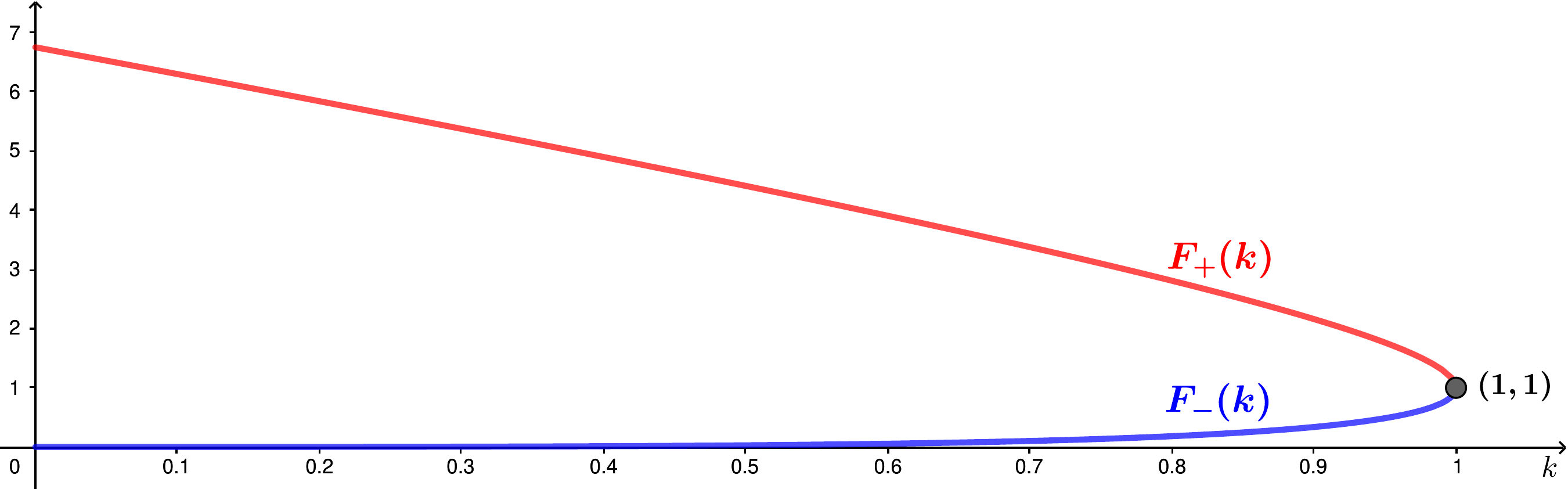}
\caption{Plots of the functions $F_-$ and $F_+$ with respect to $k$.}
\label{Fig5}
\end{figure}

Indeed, for every $k\in[0,1]$,
\begin{align*}
F_+(k)-1&=\frac{1}{8} \left[ -k^2-18k+19+(9-k)\sqrt{(9-k)(1-k)}\right]\\[1ex]
&=\frac{1}{8} \left[ (19+k)(1-k)+(9-k)\sqrt{(9-k)(1-k)}\right]\ge 0
\end{align*}
and
\begin{align*}
k^3-F_-(k)&=\frac{1}{8} \left[8k^3 +k^2+18k-27+(9-k)\sqrt{(9-k)(1-k)}\right] \\[1ex]
&=\frac{1}{8}\left[-(8k^2+9k+27)(1-k)+(9-k)\sqrt{(9-k)(1-k)}\right]\\[1ex]
&=\frac{1}{8}\;\frac{16k^3(1-k)(4k^2+5k+23)}{(8k^2+9k+27)(1-k)+(9-k)\sqrt{(9-k)(1-k)}}\ge 0.
\end{align*}
This ends the proof.
\end{proof}

Note that, in particular, Theorem \ref{th5.1} always applies when $a$, $b$, $c$ and $d$ are positive constants such that $\kappa =bc/ad<1$. Very simple examples show that Theorem \ref{th5.1} might not be true if $\kappa\geq 1$. Indeed, suppose that $\mc{L}_1=\mc{L}_2=\mc{L}$, $\mc{B}_1=\mc{B}_2=\mc{B}$, $d_1=d_2=1$, $\l=\mu$, $a=d=1$ and $b=c$ with $\kappa=b^2\geq 1$. Then, $b\geq 1$ and the problem \eqref{5.1} becomes
\begin{equation}
\label{5.13}
\left\{
\begin{array}{ll}
\mathcal{L} u=\lambda u-u^2+ b uv & \hbox{in}\;\;\Omega,\\
\mathcal{L} v=\l v-v^2+ b uv & \hbox{in}\;\;\Omega,\\
\mathcal{B} u=\mathcal{B} v=0 & \hbox{on}\;\;\partial\Omega.
\end{array}
\right.
\end{equation}
Thus, for any positive solution of
\begin{equation}
\label{5.14}
\left\{
\begin{array}{ll}
\mathcal{L} w=\l w+(b-1)w^2 & \hbox{in}\;\;\Omega,\\
\mathcal{B} w=0 & \hbox{on}\;\;\partial\Omega,
\end{array}
\right.
\end{equation}
the pair $(w,w)$ provides us with a coexistence state of \eqref{5.13}. As for $b=1$ and
$\l=\s[\mc{L},\mc{B},\O]$, \eqref{5.14} possesses infinitely many positive solutions,
for this choice of the values of the parameters, \eqref{5.13} admits multiplicity of coexistence states. Similarly,  if $b>1$, it is folklore that \eqref{5.14} can admit multiple positive solutions for
some range of $\l$'s bellow $\s[\mc{L},\mc{B},\O]$ if $\O$ has the appropriate geometry (see, e.g.,
E. N. Dancer \cite{Dancer}). Therefore, the condition  $\kappa<1$ is optimal for the validity of
Theorem \ref{th5.1} under constant coefficients.
\par
Theorem \ref{th5.1}  provides us with the linear stability of the coexistence states of
\eqref{5.1} (and, hence, their uniqueness as it will become later) not only in the case of constant coefficients, but also in a wide range of situations where two among the coefficients $a$, $b$, $c$ and $d$ are arbitrary while the remaining ones are chosen so that all assumptions of Theorem \ref{th5.1} are fulfilled. For example, choose $c$ and $d$ arbitrary and pick $\eta>0$, take $a:=\frac{\eta c^3}{d^2}$, and finally choose any function $b$ such that
\[
  b\lneq \frac{ad}{c}=\frac{\eta c^2}{d}.
\]
Note that $b$ can be arbitrary by choosing $\eta$ sufficiently large. Another advantage of Theorem \ref{th5.1} is that it provides us with a method that guarantees the linearized stability of
any coexistence state though an easily computable condition.
\par
It is worth-emphasizing that the growth rates of the species, $\lambda$ and $\mu$, do not play any role in Theorem \ref{th5.1}. However, they are ultimately responsible of the dynamics of the associated  non-spatial model ($d_1=d_2=0$). Thus, for any given domain $\Omega$ and functions $a$, $b$, $c$ and $d$ satisfying the hypothesis of Theorem \ref{th5.1}, $\lambda$ and $\mu$ can be chosen so that, for every $x\in\O$, the non-spatial model
\begin{equation}
\label{5.15}
\left\{
\begin{array}{l}
u'=\lambda(x) u-a(x)u^2+ b(x) uv,\\[1ex]
v'=\mu(x) v-d(x) v^2+ c(x) uv,\end{array}
\right.
\end{equation}
can exhibit any desired low-symbiosis dynamics, as soon as it respects the continuity of $\lambda$ and $\mu$. This feature reveals the huge versatility of Theorem \ref{th5.1}, for as it can be applied
independently of the underlying non-spatial dynamics of \eqref{5.15}.
\par
The next by-products of Theorem \ref{th5.1} show the role played by variable diffusion rates in the symbiotic system \eqref{5.1}.

\begin{corollary}
\label{co5.2}
Assume that $d_1,d_2\in\mathcal{C}(\bar\Omega;(0,+\infty))$, $\kappa=\frac{bc}{ad}\lneq 1$ in $\Omega$, and
\begin{equation*}
\max_{\bar{\Omega}}\left( \frac{ad^2}{c^3}\frac{d_2}{d_1}F_{-}(\kappa)\right) \le \min_{\bar{\Omega}}\left(\frac{ad^2}{c^3}\frac{d_2}{d_1}F_{+}(\kappa)\right),
\end{equation*}
with $F_\pm$ defined as in \eqref{5.2}. Then, every coexistence state of
\begin{equation}
\label{5.16}
\left\{
\begin{array}{ll}
d_1(x)\mathcal{L}_1 u=\lambda u-au^2+ b uv & \hbox{in}\;\;\Omega,\\
d_2(x)\mathcal{L}_2 v=\mu v-dv^2+ c uv & \hbox{in}\;\;\Omega,\\
\mathcal{B}_1 u=\mathcal{B}_2 v=0 & \hbox{on}\;\;\partial\Omega,
\end{array}
\right.
\end{equation}
is linearly stable. In particular, there exist functions $d_1(x),d_2(x)\in\mathcal{C}(\bar\Omega;(0,+\infty))$, as close to zero as desired, such that every coexistence state of \eqref{5.16} is linearly stable.
\end{corollary}
\begin{proof}
The first assertion is an immediate consequence of Theorem \ref{th5.1}, by dividing by $d_1$ the $u$-equation and by $d_2$ the $v$-equation. For the second assertion, since $F_-(\kappa)\le 1\le F_+(\kappa)$, it suffices to take $d_1:=\varepsilon a d^2$ and $d_2:=\varepsilon c^3$ with  $\varepsilon>0$ arbitrary.
\end{proof}

The next result guarantees the existence of, at most, a coexistence state based  on the existence and linearized stability of the semitrivial solutions,  $(\theta_{\{d_1,\lambda,a\}},0)$ and $(0,\theta_{\{d_2,\mu,d\}})$. 

\begin{theorem}
\label{th5.3}
Assume $d_1,d_2\in\mathcal{C}(\bar\Omega;(0,+\infty))$,
\begin{equation}
\label{5.17}
\max_{\bar\Omega}\left(\frac{c}{d}\right)<\min_{\bar\Omega}\left(\frac{a}{b}\right),
\end{equation}
and
\begin{equation*}
\max_{\bar{\Omega}}\left( \frac{ad^2}{c^3}\frac{d_2}{d_1}F_{-}(\kappa)\right)\le \min_{\bar{\Omega}}\left(\frac{ad^2}{c^3}\frac{d_2}{d_1}F_{+}(\kappa)\right).
\end{equation*}
Then:
\begin{enumerate}[{\rm (a)}]
\item\label{th5.3.item1} If some of the semitrivial solutions exists and it is linearly unstable, then the system \eqref{5.16} admits a unique coexistence state. Moreover, the coexistence state  is linearly stable and, actually, it is a global attractor for the component-wise positive solutions of the parabolic counterpart of \eqref{5.16}.
\item\label{th5.3.item2} In any other case, \eqref{5.16} cannot admit admit coexistence states.
\end{enumerate}
\end{theorem}

The proof of Theorem \ref{th5.3} follows by applying the fixed point index to a certain integral operator associated to the problem
\begin{equation}
\label{5.18}
\left\{
\begin{array}{ll}
d_1(x){\mathcal{L}}_1 u={\lambda} u - a u^2 + \gamma b u v & \mbox{in}\;\;\Omega,\\
d_2(x){\mathcal{L}}_2 v={\mu} v - d v^2 + \gamma c u v & \mbox{in}\;\;\Omega,\\
\mathcal{B}_1 u=\mathcal{B}_2 v=0 & \mbox{on}\;\,\partial\Omega,
\end{array}
\right.
\end{equation}
where $\g\in [0,1]$ is regarded as an homotopy parameter to uncouple \eqref{5.16}.
\par
Subsequently, we will set
\begin{equation*}
W^{2,\infty}_{\mathcal{B}_i}:=\bigcap_{\theta>N} W^{2,\theta}_{\mathcal{B}_i}(\Omega)\quad\mbox{and}\quad
P_{W^{2,\infty}_{\mathcal{B}_i}}=\{w\in W^{2,\infty}_{\mathcal{B}_i}\,:\,w\ge 0\;\;\mbox{in}\;\;\Omega\},\quad i=1,2.
\end{equation*}
Then, $w\in \mbox{int}\, P_{W^{2,\infty}_{\mathcal{B}_i}}$ if $w\in W^{2,\infty}_{\mathcal{B}_i}$ and  satisfies \eqref{3.3} with $(\G_\mc{R},\G_\mc{D})=(\G_\mc{R}^i,\G_\mc{D}^i)$, $i=1, 2$.

The next result provides us with uniform a priori bounds in $\g\in [0,1]$ for the positive solutions of \eqref{5.18}.

\begin{lemma}
\label{le5.4}
Under condition \eqref{5.17}, there exits a bounded set $\mathcal{U}\times\mathcal{V}\subset W^{2,\infty}_{\mathcal{B}_1}\times W^{2,\infty}_{\mathcal{B}_2}$, independent of $\gamma\in [0,1]$, such that $(u,v)\in\mathrm{int}\;(\mathcal{U}\times \mathcal{V})$ if $(u,v)$ is a solution of \eqref{5.18} with $(u,v)\in P_{W^{2,\infty}_{\mathcal{B}_1}}\times \,P_{W^{2,\infty}_{\mathcal{B}_2}}$ for some
$\g \in[0,1]$.
\end{lemma}
\begin{proof}
Consider $\Psi$ defined by
\begin{equation*}
\Psi(x):=\exp\left(- M\, \mathrm{dist}_{\partial\Omega}(x) \right)\quad \hbox{with}\;\; M>\max_{\partial\Omega}\left\{0,\frac{-\beta_1}{\langle{\rm n},A_1 \mathrm{n}\rangle}, \frac{-\beta_2}{\langle{\rm n},A_2 {\rm n}\rangle} \right\},
\end{equation*}
on a sufficiently narrow neighborhood of $\partial \Omega$. According to \cite[Sec. 2 \& 3]{FRLGANS} this function is of class $\mathcal{C}^2$, and can be extended to the entire  $\bar\Omega$ with smoothness and positiveness by mean of cut-off functions. Furthermore, $\Psi$ satisfies $\mathcal{B}_i\Psi>0$ in $\Gamma_{\mathcal{D}}^i$ and
\begin{equation*}
\mathcal{B}_i\Psi=\langle {\rm n},A_i \nabla \Psi\rangle + \beta_i \Psi=
-\Psi M \langle {\rm n},A_i \nabla \mathrm{dist}_{\partial\Omega}(\cdot)\rangle + \beta_i \Psi=\Psi \left( M\langle {\rm n},A_i {\rm n}\rangle + \beta_i \right)> 0\quad \hbox{on}\;\;\Gamma_{\mathcal{R}}^i.
\end{equation*}
Now, we will show that there exist two positive constants, $\bar{u}, \bar{v}>0$, independent of $\g \in [0,1]$, such that $(\bar{u}\Psi,\bar{v}\Psi)$ is a positive supersolution of \eqref{5.18}. Once established this, the result holds as a direct consequence of the maximum principle for cooperative systems (see, e.g., \cite[Th. 2.1]{LGMM}, \cite[Th. 6.3]{LGS}, \cite[Th. 5.4]{FRLGNA}, or \cite[Th.1.3]{ALLG}). Indeed, for $(\bar{u}\Psi,\bar{v}\Psi)$ to be a supersolution of \eqref{5.18}, $\bar{u}$ and $\bar{v}$ should satisfy
\begin{equation}
\label{5.19}
\begin{split}
d_1{\mathcal{L}}_1 (\bar{u}\Psi) & \ge\lambda \bar{u}\Psi - a (\bar{u}\Psi)^2 + \gamma b \bar{u}\bar{v}\Psi^2 \quad \hbox{in}\;\;\Omega, \\
d_2{\mathcal{L}}_2 (\bar{v}\Psi) & \ge\mu \bar{v}\Psi - d (\bar{v}\Psi)^2 + \gamma c \bar{u}\bar{v}\Psi^2 \quad \mbox{in}\;\;\Omega,
\end{split}
\end{equation}
for every $\gamma\in[0,1]$. Note that, by the choice of $\Psi$,
\[
  \mathcal{B}_1(\bar{u}\Psi)=\bar u \mc{B}_1\Psi >0\quad \hbox{and} \quad
  \mathcal{B}_2(\bar{v}\Psi)=\bar{v} \mc{B}_2\Psi >0\quad \hbox{on}\;\; \partial \Omega.
\]
Since $\bar{u}$ and $\bar{v}$ are assumed to be positive constants, dividing by $a\bar{u}\Psi^2$
the first equation of \eqref{5.19} and by $d\bar v \Psi^2$ the second one, it becomes apparent that
\eqref{5.19} holds for every $\g\in [0,1]$ as soon as
\begin{equation*}
\bar{u} \ge \frac{b}{a}\bar{v} + \frac{\lambda - d_1 \frac{{\mathcal{L}}_1 \Psi}{\Psi}}{a\Psi} \quad \hbox{and}\quad \bar{v} \ge \frac{c}{d} \bar{u} + \frac{\mu  - d_2 \frac{{\mathcal{L}}_2 \Psi}{\Psi}}{d\Psi} \quad \mbox{in}\;\;\Omega.
\end{equation*}
Naturally, these inequalities are satisfied provided
\begin{equation}
\label{5.20}
\bar{u} \ge \max_{\bar{\Omega}}\left(\frac{b}{a}\right)\bar{v} + \max_{\bar\Omega}\left(\frac{\lambda - d_1 \frac{{\mathcal{L}}_1 \Psi}{\Psi}}{a\Psi}\right) \quad \hbox{and}\quad \bar{v} \ge \max_{\bar\Omega}\left(\frac{c}{d}\right) \bar{u} + \max_{\bar\Omega}\left(\frac{\mu  - d_2 \frac{{\mathcal{L}}_2 \Psi}{\Psi}}{d\Psi}\right).
\end{equation}
Thanks to \eqref{5.17}, we have that
\begin{equation*}
\frac{1}{\max_{\bar{\Omega}}\Big(\frac{a}{b}\Big)}=\min_{\bar{\Omega}}\Big(\frac{a}{b}\Big)
>\max_{\bar{\Omega}}\Big(\frac{c}{d}\Big)>0.
\end{equation*}
Thus, \eqref{5.20} determines a nonempty cone, $Q$, containing points within the first quadrant, and any of these points is a valid choice for $(\bar{u},\bar{v})$. This ends the proof.
\end{proof}

\begin{remark}
\label{re5.5}\rm If $(u,v)$ is a coexistence state of \eqref{5.18}, then
\[
0\le \theta_{\{d_1,\lambda,a\}}\le u\quad \hbox{and}\quad0\le \theta_{\{d_2,\mu,d\}}\le v\quad\hbox{in}\;\;\Omega,
\]
since $(\theta_{\{d_1,\lambda,a\}},\theta_{\{d_2,\mu,d\}})$ is a component-wise non-negative subsolution of \eqref{5.18}. Moreover, as a consequence of the proof of Lemma \ref{le5.4}, \eqref{5.18} admits supersolutions as large as desired. Hence, by the theory of super and subsolutions for cooperative systems
(see, e.g., \cite{MM1}, \cite{MM2} and \cite{AmPM}), \eqref{5.18} admits at least a coexistence state if both semitrivial states exist. Finally, note that \eqref{5.17} guarantees that the positive solutions of the parabolic counterpart of \eqref{5.18} are globally defined in time, as they cannot blow-up in a finite time (see \cite{DLGS}).
\end{remark}

Now, take $m\ge 0$ large enough so that
\begin{equation}
\label{5.21}
\sigma[d_1\mathcal{L}_1+m,\mathcal{B}_1,\Omega]>1,\quad
\sigma[d_2\mathcal{L}_2+m,\mathcal{B}_2,\Omega]>1,
\end{equation}
and
\begin{equation*}
\lambda-a u+\gamma b v+m>0\quad \hbox{and}\quad\mu-d v +\gamma c u+m>0\quad\hbox{in}\;\;\bar\Omega\;\;\hbox{for all}\;\;\gamma\in[0,1],\;\;u\in\mathcal{U},\;\;v\in\mathcal{V},
\end{equation*}
where $\mc{U}$ and $\mc{V}$ are those given by Lemma \ref{le5.4}. Note that, in particular,
\begin{equation}
\label{5.22}
  \l + m >0 \quad\hbox{and}\quad \mu + m>0 \quad\hbox{in}\;\; \bar \O,
\end{equation}
because $(0,0)\in \mc{U}\times \mc{V}$, by construction.
\par
Subsequently, we consider the family of operators
\[
   \mathcal{H}:[0,1]\times\mathcal{U}\times\mathcal{V}\to W^{2,\infty}_{\mathcal{B}_1}\times W^{2,\infty}_{\mathcal{B}_2}
\]
defined through
\begin{equation*}
\mathcal{H}(\gamma,u,v):=\left(
\begin{array}{ll}
(d_1{\mathcal{L}}_1+m)^{-1}[({\lambda}-au+\gamma bv+m)u]\\
(d_2{\mathcal{L}}_2+m)^{-1}[({\mu}-dv+\gamma cu+m)v]
\end{array}
\right).
\end{equation*}
By our assumptions on $m$, for every $\gamma\in[0,1]$, $\mathcal{H}(\gamma,\cdot,\cdot)$ is a compact order preserving operator. Moreover, the fixed points of $\mathcal{H}(\gamma,\cdot,\cdot)$ are the solutions of \eqref{5.18} in $P_{W^{2,\infty}_{\mathcal{B}_1}}\times P_{W^{2,\infty}_{\mathcal{B}_2}}$. Next, we will calculate all the fixed point indices of these solutions as fixed points of $\mc{H}$.

\begin{lemma}
\label{le5.6}
The total fixed point index of $\mc{H}(1,\cdot,\cdot)$ in \, $\mathrm{int}\,(\mathcal{U}\times\mathcal{V})$
\, is given by
\begin{equation*}
i_{P_{W^{2,\infty}_{\mathcal{B}_1}}\times P_{W^{2,\infty}_{\mathcal{B}_2}}}(\mathcal{H}(1,\cdot,\cdot),\mathrm{int}\,(\mathcal{U}\times\mathcal{V}))=1.
\end{equation*}
\end{lemma}
\begin{proof} By the homotopy invariance property of the fixed point index (see \cite[Th. 11.1(iii)]{Am76}),
\begin{equation*}
i_{P_{W^{2,\infty}_{\mathcal{B}_1}}\times P_{W^{2,\infty}_{\mathcal{B}_2}}}(\mathcal{H}(1,\cdot,\cdot),\hbox{int}\,(\mathcal{U}\times\mathcal{V}))=
i_{P_{W^{2,\infty}_{\mathcal{B}_1}}\times P_{W^{2,\infty}_{\mathcal{B}_2}}}(\mathcal{H}(0,\cdot,\cdot),\hbox{int}\,(\mathcal{U}\times\mathcal{V})).
\end{equation*}
Thus, owing to the product formula, we find that
\begin{equation*}
i_{P_{W^{2,\infty}_{\mathcal{B}_1}}\times P_{W^{2,\infty}_{\mathcal{B}_2}}}(\mathcal{H}(1,\cdot,\cdot),\hbox{int}\,(\mathcal{U}\times\mathcal{V}))
=i_{P_{W^{2,\infty}_{\mathcal{B}_1}}}(\mathcal{H}_1(\cdot),\hbox{int}\,\mathcal{U})\cdot i_{P_{W^{2,\infty}_{\mathcal{B}_1}}}(\mathcal{H}_2(\cdot),\hbox{int}\,\mathcal{V}),
\end{equation*}
where $\mathcal{H}_1:\mathcal{U}\to W^{2,\infty}_{\mathcal{B}_1}$ and $\mathcal{H}_2:\mathcal{V}\to W^{2,\infty}_{\mathcal{B}_2}$ are the operators defined by
\begin{equation*}
\mathcal{H}_1(u):=(d_1{\mathcal{L}}_1+m)^{-1}[({\lambda}-au+m)u]
\quad \hbox{and}\quad
\mathcal{H}_2(v):=(d_2{\mathcal{L}}_2+m)^{-1}[({\mu} -d v+m)v].
\end{equation*}
Now, consider the homotopies
\begin{align*}
\mathcal{G}_1(\gamma,u) & :=(d_1{\mathcal{L}}_1+m)^{-1}\left[ \left(m+\sigma[d_1\mathcal{L}_1,\mathcal{B}_1,\Omega]-1+\gamma ({\lambda}-\sigma[d_1\mathcal{L}_1,\mathcal{B}_1,\Omega]+1-au)\right)u\right], \\
\mathcal{G}_2(\gamma,v) & :=(d_2{\mathcal{L}}_2+m)^{-1}\left[\left(m+\sigma[d_2\mathcal{L}_2,\mathcal{B}_2,\Omega]-1+\gamma({\mu}-
\sigma[d_2\mathcal{L}_2,\mathcal{B}_2,\Omega]+1-dv)\right)v\right].
\end{align*}
By the homotopy invariance property,
\begin{align*}
i_{P_{W^{2,\infty}_{\mathcal{B}_1}}}(\mathcal{H}_1(\cdot),\hbox{int}\,\mathcal{U}) & =i_{P_{W^{2,\infty}_{\mathcal{B}_1}}}(\mathcal{G}_1(1,\cdot),\hbox{int}\,\mathcal{U})=
i_{P_{W^{2,\infty}_{\mathcal{B}_1}}}(\mathcal{G}_1(0,\cdot),\hbox{int}\,\mathcal{U}),\\
i_{P_{W^{2,\infty}_{\mathcal{B}_2}}}(\mathcal{H}_2(\cdot),\hbox{int}\,\mathcal{V}) & =i_{P_{W^{2,\infty}_{\mathcal{B}_2}}}(\mathcal{G}_2(1,\cdot),\hbox{int}\,\mathcal{V})=
i_{P_{W^{2,\infty}_{\mathcal{B}_2}}}(\mathcal{G}_2(0,\cdot),\hbox{int}\,\mathcal{V}).
\end{align*}
On the other hand, the spectral radio of $\mathcal{G}_j(0,\cdot)$, $j=1,2$, is
\begin{equation*}
\varrho(\mathcal{G}_j(0,\cdot))=\frac{m+\sigma[d_j\mathcal{L}_j,\mathcal{B}_j,\Omega]-1}
{m+\sigma[d_j\mathcal{L}_j,\mathcal{B}_j,\Omega]}<1.
\end{equation*}
Thus, thanks to \cite[Th. 13.1]{Am76}, we can infer that
\begin{equation*}
i_{P_{W^{2,\infty}_{\mathcal{B}_1}}}(\mathcal{G}_1(0,\cdot),\hbox{int}\,\mathcal{U})=
i_{P_{W^{2,\infty}_{\mathcal{B}_2}}}(\mathcal{G}_2(0,\cdot),\hbox{int}\,\mathcal{V})=1.
\end{equation*}
This ends the proof.
\end{proof}

The following result provides us with the fixed point index of $(0,0)$ when it is non-degenerate.

\begin{lemma}
\label{le5.7}
Suppose
\begin{equation*}
   \sigma[d_1\mathcal{L}_1-\lambda,\mathcal{B}_1,\Omega]\neq 0,\qquad
   \sigma[d_2\mathcal{L}_2-\mu,\mathcal{B}_2,\Omega]\neq 0.
\end{equation*}
Then,
\begin{equation*}
i_{P_{W^{2,\infty}_{\mathcal{B}_1}}\times P_{W^{2,\infty}_{\mathcal{B}_2}}}(\mathcal{H}(1,\cdot,\cdot),(0,0))=
\left\{
\begin{array}{ll}
1 & \hbox{if}\;\; \sigma[d_1{\mathcal{L}}_1-{\lambda},\mathcal{B}_1,\Omega]>0\;\; \hbox{and}\;\; \sigma[d_2{\mathcal{L}}_2-{\mu},\mathcal{B}_2,\Omega]> 0,\\[0.5ex]
0 & \hbox{if}\;\; \sigma[d_1{\mathcal{L}}_1-{\lambda},\mathcal{B}_1,\Omega]<0,\;\; \hbox{or}\;\; \sigma[d_2{\mathcal{L}}_2-{\mu},\mathcal{B}_2,\Omega]< 0.
\end{array}
\right.
\end{equation*}
\end{lemma}
\begin{proof}
Differentiating $\mathcal{H}(1,\cdot,\cdot)$ with respect to $(u,v)$ and particularizing at $(0,0)$ yields
\begin{equation*}
D_{(u,v)}\mathcal{H}(1,0,0)(u,v)=
\left(
\begin{array}{l}
(d_1{\mathcal{L}}_1+m)^{-1}[({\lambda}+m)u]\\ (d_2{\mathcal{L}}_2+m)^{-1}[({\mu}+m)v]
\end{array}
\right).
\end{equation*}
Suppose that
\begin{equation}
\label{5.23}
\sigma[d_1{\mathcal{L}}_1-{\lambda},\mathcal{B}_1,\Omega]>0\quad \hbox{and}\quad \sigma[d_2{\mathcal{L}}_2-{\mu},\mathcal{B}_2,\Omega]> 0.
\end{equation}
Let $r_0\in\mathbb{R}$ be an eigenvalue of $D_{(u,v)}\mathcal{H}(1,0,0)$ to a component-wise non-negative eigenvector, $(\varphi,\psi)\neq (0,0)$.
Without loss of generality, we can assume that $\varphi>0$. Then,
\begin{equation*}
\sigma\Big[d_1{\mathcal{L}}_1+m-\frac{m+\lambda}{r_0},\mathcal{B}_1,\Omega\Big]=0.
\end{equation*}
Moreover, thanks to \eqref{5.22}, by the strict monotonicity of the principal eigenvalue with respect to the potential (see, e.g.,  \cite[Th. 4.1]{FRLGANS}), it becomes apparent that
\[
  r\mapsto \sigma\left[d_1{\mathcal{L}}_1+m-\frac{m+\lambda}{r},\mathcal{B}_1,\Omega\right]
\]
is strictly increasing and, in addition,  is continuous in $(0,+\infty)$. Taking into account \eqref{5.21},
\begin{equation}
\label{5.24}
\lim_{r\to \xi} \sigma\Big[d_1{\mathcal{L}}_1+m-\frac{m+{\lambda}}{r},\mathcal{B}_1,\Omega\Big]=
\left\{
\begin{array}{ll}
\sigma[d_1{\mathcal{L}}_1+m,\mathcal{B}_1,\Omega]>1 & \hbox{if}\;\; \xi= +\infty,\\
\sigma[d_1{\mathcal{L}}_1-{\lambda},\mathcal{B}_1,\Omega] & \hbox{if}\;\; \xi= 1,\\
-\infty & \hbox{if}\;\; \xi= 0.
\end{array}
\right.
\end{equation}
By \eqref{5.24}, $r_0<1$. Thus,  $D_{(u,v)}\mathcal{H}(1,0,0)$ cannot admit a positive eigenvector to an eigenvalue greater or equal than one. Therefore, owing to \cite[Th. 13.1]{Am76}, we find that
\begin{equation*}
i_{P_{W^{2,\infty}_{\mathcal{B}_1}}\times P_{W^{2,\infty}_{\mathcal{B}_2}}}(\mathcal{H}(1,\cdot,\cdot),(0,0))=1.
\end{equation*}
Now, assume that some of the principal eigenvalues in \eqref{5.23} is negative, instead of positive.
For example,
\[
  \sigma[d_1{\mathcal{L}}_1-{\lambda},\mathcal{B}_1,\Omega]<0.
\]
Then, by \eqref{5.24}, there is a unique $r_0>1$ such that
\[
  \sigma\left[d_1{\mathcal{L}}_1+m-\frac{m+{\lambda}}{r_0},\mathcal{B}_1,\Omega\right]=0.
\]
Let $\varphi>0$ be any principal eigenfunction associated to this eigenvalue. Then, $(\varphi, 0)$ provides us with a positive eigenvector of $D_{(u,v)}\mathcal{H}(1,0,0)$ to an eigenvalue greater than one. Therefore, thanks to \cite[Th. 13.1]{Am76},
\begin{equation*}
i_{P_{W^{2,\infty}_{\mathcal{B}_1}}\times P_{W^{2,\infty}_{\mathcal{B}_2}}}(\mathcal{H}(1,\cdot,\cdot),(0,0))=0.
\end{equation*}
From these ingredients the proof can be easily completed in all possible cases.
\end{proof}

To calculate the indices of the semitrivial solutions we will make an intensive use of \cite[Le. 4.1]{LGDIE}, which goes back to \cite[Le. 2 \& Le. 4]{Dancer83}. In our setting, it can be stated as follows. An analogous version holds true for $(0,\theta_{\{d_2,\mu,d\}})$. Subsequently, we will denote by
\[
  P_2:W^{2,\infty}_{\mathcal{B}_1}\times W^{2,\infty}_{\mathcal{B}_2}\to \{0\}\times W^{2,\infty}_{\mathcal{B}_2}
\]
the projection on the second component, i.e., $P_2(u,v):=(0,v)$.

\begin{lemma}
\label{le5.8}
Assume that $\sigma[d_1\mathcal{L}_1-\lambda,\mathcal{B}_1,\Omega]<0$. So, $\theta_{\{d_1,\lambda,a\}}\neq 0$. Then, the following assertions hold:
\begin{enumerate}[{\rm a)}]
\item\label{le5.8.item1} If the operator $I-D_{(u,v)}\mathcal{H}(1,\theta_{\{d_1,\lambda,a\}},0)$ is injective  in $W^{2,\infty}_{\mathcal{B}_1}\times W^{2,\infty}_{\mathcal{B}_2}$ and the spectral radius of $P_2 D_{(u,v)}\mathcal{H}(1,\theta_{\{d_1,\lambda,a\}},0)\Big|_{\{0\}\times  W^{2,\infty}_{\mathcal{B}_2}}$ is greater than one, then
\begin{equation*}
i_{P_{W^{2,\infty}_{\mathcal{B}_1}}\times P_{W^{2,\infty}_{\mathcal{B}_2}}}(\mathcal{H}(1,\cdot,\cdot),(\theta_{\{d_1,\lambda,a\}},0))=0.
\end{equation*}
\item\label{le5.8.item2} If the operator $I-D_{(u,v)}\mathcal{H}(1,\theta_{\{d_1,\lambda,a\}},0)$ is injective in $W^{2,\infty}_{\mathcal{B}_1}\times W^{2,\infty}_{\mathcal{B}_2}$ and the spectral radius of $P_2 D_{(u,v)}\mathcal{H}(1,\theta_{\{d_1,\lambda,a\}},0)\Big|_{\{0\}\times  W^{2,\infty}_{\mathcal{B}_2}}$ is less than one, then
\begin{equation*}
i_{P_{W^{2,\infty}_{\mathcal{B}_1}}\times P_{W^{2,\infty}_{\mathcal{B}_2}}}(\mathcal{H}(1,\cdot,\cdot),(\theta_{\{d_1,\lambda,a\}},0))=(-1)^\chi,
\end{equation*}
where $\chi$ stands for the sum of the algebraic multiplicities of the eigenvalues of $D_{(u,v)}\mathcal{H}(1,\theta_{\{d_1,\lambda,a\}},0)$ greater than one.

\item\label{le5.8.item3} If $I-D_{(u,v)}\mathcal{H}(1,\theta_{\{d_1,\lambda,a\}},0)$ is injective in $W^{2,\infty}_{\mathcal{B}_1}\times P_{W^{2,\infty}_{\mathcal{B}_2}}$,  instead of in $W^{2,\infty}_{\mathcal{B}_1}\times W^{2,\infty}_{\mathcal{B}_2}$, and there exists $w\in W^{2,\infty}_{\mathcal{B}_1}\times P_{W^{2,\infty}_{\mathcal{B}_2}}$ such that the equation
\[
   (I-D_{(u,v)}\mathcal{H}(1,\theta_{\{d_1,\lambda,a\}},0))y=w
\]
has no solution $y\in  W^{2,\infty}_{\mathcal{B}_1}\times P_{W^{2,\infty}_{\mathcal{B}_2}}$, then
\begin{equation*}
i_{P_{W^{2,\infty}_{\mathcal{B}_1}}\times P_{W^{2,\infty}_{\mathcal{B}_2}}}(\mathcal{H}(1,\cdot,\cdot),(\theta_{\{d_1,\lambda,a\}},0))=0.
\end{equation*}
\end{enumerate}
\end{lemma}
As a direct consequence of Lemma \ref{le5.8}, the next result establishes that the fixed point index of each semitrivial solution is determined by its linear stability as  a steady-state solution of the parabolic counterpart of \eqref{5.16}, as soon as it is a non-degenerate steady state.

\begin{lemma}
\label{le5.9} The following identities are satisfied:
\begin{equation}
\label{5.25}
i_{P_{W^{2,\infty}_{\mathcal{B}_1}}\times P_{W^{2,\infty}_{\mathcal{B}_2}}}(\mathcal{H}(1,\cdot,\cdot),(\theta_{\{d_1,\lambda,a\}},0))=
\left\{
\begin{array}{ll}
0 & \hbox{if}\;\;\sigma[d_2\mathcal{L}_2-\mu-c\theta_{\{d_1,\lambda,a\}},\mathcal{B}_2,\Omega]<0,\\
1 & \hbox{if}\;\;\sigma[d_2\mathcal{L}_2-\mu-c\theta_{\{d_1,\lambda,a\}},\mathcal{B}_2,\Omega]>0,
\end{array}
\right.
\end{equation}
and
\begin{equation}
\label{5.26}
i_{P_{W^{2,\infty}_{\mathcal{B}_1}}\times P_{W^{2,\infty}_{\mathcal{B}_2}}}(\mathcal{H}(1,\cdot,\cdot),(0,\theta_{\{d_2,\mu,d\}}))=
\left\{
\begin{array}{ll}
0 & \hbox{if}\;\;\sigma[d_1\mathcal{L}_1-\lambda-b\theta_{\{d_2,\mu,d\}},\mathcal{B}_1,\Omega]<0,\\
1 & \hbox{if}\;\;\sigma[d_1\mathcal{L}_1-\lambda-b\theta_{\{d_2,\mu,d\}},\mathcal{B}_1,\Omega]>0.
\end{array}
\right.
\end{equation}
\end{lemma}
\begin{proof} We will only prove \eqref{5.25}, as \eqref{5.26} follows by symmetry.
Assume that $\sigma[d_1{\mathcal{L}}_1-{\lambda},\mathcal{B}_1,\Omega]<0$. Then,  $\theta_{\{d_1,\lambda,a\}}\gg 0$ and differentiating $\mathcal{H}(1,\cdot,\cdot)$ yields
\begin{equation*}
D_{(u,v)}\mathcal{H}(1,\theta_{\{d_1,\lambda,a\}},0)(u,v)=
\left(
\begin{array}{l}
(d_1{\mathcal{L}}_1+m)^{-1}\left[({\lambda}-2a \theta_{\{d_1,\lambda,a\}} +m)u+b \theta_{\{d_1,\lambda,a\}} v\right]\\
(d_2{\mathcal{L}}_2+m)^{-1}[({\mu}+c \theta_{\{d_1,\lambda,a\}} +m)v]
\end{array} \right).
\end{equation*}
Thanks to \eqref{5.21} and \eqref{5.22}, $D_{(u,v)}\mathcal{H}(1,\theta_{\{d_1,\lambda,a\}},0)$
maps $W^{2,\infty}_{\mathcal{B}_1}\times P_{W^{2,\infty}_{\mathcal{B}_2}}$ into itself.
Assume that
\begin{equation}
\label{5.27}
\sigma[d_2{\mathcal{L}}_2-{\mu}-c\theta_{\{d_1,{\lambda},a\}},\mathcal{B}_2,\Omega]>0.
\end{equation}
Then, $I-D_{(u,v)}\mathcal{H}(1,\theta_{\{d_1,\lambda,a\}},0)$ is injective on $W^{2,\infty}_{\mathcal{B}_1}\times W^{2,\infty}_{\mathcal{B}_2}$. Indeed, if there exists $(u,v)\in W^{2,\infty}_{\mathcal{B}_1}\times W^{2,\infty}_{\mathcal{B}_2}$ such that
\[
  D_{(u,v)}\mathcal{H}(1,\theta_{\{d_1,\lambda,a\}},0)(u,v)=(u,v),
\]
then
\begin{equation}
\label{5.28}
\left(d_1{\mathcal{L}}_1-{\lambda}+2a \theta_{\{d_1,\lambda,a\}}\right) u=b \theta_{\{d_1,\lambda,a\}} v
\end{equation}
and
\begin{equation}
\label{5.29}
\left(d_2{\mathcal{L}}_2-{\mu}-c \theta_{\{d_1,\lambda,a\}}\right) v=0.
\end{equation}
If $v\neq 0$, then $0$ is an eigenvalue of $d_2{\mathcal{L}}_2-{\mu}-c \theta_{\{d_1,\lambda,a\}}$. Thus,
\[
 \sigma[d_2{\mathcal{L}}_2-{\mu}-c \theta_{\{d_1,\lambda,a\}},\mathcal{B}_2,\Omega]\leq 0,
\]
which contradicts  \eqref{5.27}. Hence, $v=0$ and \eqref{5.28} becomes
\[
\left(d_1{\mathcal{L}}_1-{\lambda}+2a \theta_{\{d_1,\lambda,a\}}\right) u=0.
\]
If $u\neq 0$, then $0$ is an eigenvalue of $d_1{\mathcal{L}}_1-{\lambda}+2a \theta_{\{d_1,\lambda,a\}}$. Thus,
\begin{equation}
\label{5.30}
\sigma[d_1{\mathcal{L}}_1-{\lambda}+2a \theta_{\{d_1,\lambda,a\}},\mathcal{B}_1,\Omega]\leq 0.
\end{equation}
Consequently, by the monotonicity of the principal eigenvalue with respect to the potential,
\begin{equation}
\label{5.31}
\sigma[d_1{\mathcal{L}}_1-{\lambda}+a \theta_{\{d_1,\lambda,a\}},\mathcal{B}_1,\Omega]<\sigma[d_1{\mathcal{L}}_1-{\lambda}+2a \theta_{\{d_1,\lambda,a\}},\mathcal{B}_1,\Omega]\leq 0.
\end{equation}
This contradicts the fact that
$$
   (d_1{\mathcal{L}}_1-{\lambda}+a \theta_{\{d_1,\lambda,a\}})\theta_{\{d_1,\lambda,a\}}=0
$$
because this entails that
\begin{equation}
\label{5.32}
\sigma[d_1{\mathcal{L}}_1-{\lambda}+a \theta_{\{d_1,\lambda,a\}},\mathcal{B}_1,\Omega]=0.
\end{equation}
Therefore, $(u,v)=(0,0)$ and, hence, $I-D_{(u,v)}\mathcal{H}(1,\theta_{\{d_1,\lambda,a\}},0)$ is injective.
For applying Lemma \ref{le5.8}, it remains to estimate the spectral radio of the operator
\[
P_2 D_{(u,v)}\mathcal{H}(1,\theta_{\{d_1,\lambda,a\}},0)v:=(d_2{\mathcal{L}}_2+m)^{-1}[({\mu}+c \theta_{\{d_1,\lambda,a\}} +m)v], \quad v\in W^{2,\infty}_{\mc{B}_2}.
\]
A direct calculation shows that the spectral radius, $r_0$, of this operator satisfies
\[
\sigma\Big[d_2{\mathcal{L}}_2+m-\frac{{\mu}+c \theta_{\{d_1,\lambda,a\}} +m}{r_0},\mathcal{B}_2,\Omega\Big]=0.
\]
Arguing as in the proof of Lemma \ref{le5.7}, it is apparent that the map
$$
  r\mapsto \sigma\left[d_2{\mathcal{L}}_2+m-\frac{{\mu}+c \theta_{\{d_1,\lambda,a\}} +m}{r},\mathcal{B}_2,\Omega\right]
$$
is continuous, strictly increasing and, owing to \eqref{5.27}, satisfies
\begin{equation*}
\lim_{r\to \xi} \sigma\Big[d_2{\mathcal{L}}_2+m-\frac{{\mu}+c \theta_{\{d_1,\lambda,a\}} +m}{r},\mathcal{B}_2,\Omega\Big]=
\left\{
\begin{array}{ll}
\sigma[d_2{\mathcal{L}}_2+m,\mathcal{B}_2,\Omega]>1 & \hbox{if}\;\; \xi= +\infty,\\
\sigma[d_2{\mathcal{L}}_2-{\mu}-c \theta_{\{d_1,\lambda,a\}},\mathcal{B}_2,\Omega]>0 & \hbox{if}\;\; \xi= 1,\\ -\infty & \hbox{if}\;\; \xi= 0.
\end{array}
\right.
\end{equation*}
Thus, $r_0<1$. Therefore, due to Lemma \ref{le5.8} \ref{le5.8.item2}),
\begin{equation*}
i_{P_{W^{2,\infty}_{\mathcal{B}_1}}\times P_{W^{2,\infty}_{\mathcal{B}_2}}}(\mathcal{H}(1,\cdot,\cdot),(\theta_{\{d_1,\lambda,a\}},0))=(-1)^\chi,
\end{equation*}
where $\chi$ is the sum of the multiplicities of the eigenvalues of $D_{(u,v)}\mathcal{H}(1,\theta_{\{d_1,\lambda,a\}},0)$ greater than one. Now, assume that $\tau>1$ is an eigenvalue of $D_{(u,v)}\mathcal{H}(1,\theta_{\{d_1,\lambda,a\}},0)$ with associated eigenvector $(u,v)\neq (0,0)$. If $v\ne 0$, then, it is easily seen that
\begin{equation}
\label{5.33}
\sigma\Big[d_2{\mathcal{L}}_2+m-\frac{{\mu}+c \theta_{\{d_1,\lambda,a\}} +m}{\tau},\mathcal{B}_2,\Omega\Big]\leq 0,
\end{equation}
by the dominance of the principal eigenvalue. However, by the strict monotonicity of the principal eigenvalue with respect to the potential, it follows from \eqref{5.27} that
\[
\sigma\Big[d_2{\mathcal{L}}_2+m-\frac{{\mu}+c \theta_{\{d_1,\lambda,a\}} +m}{\tau},\mathcal{B}_2,\Omega\Big]>\sigma\Big[d_2{\mathcal{L}}_2-{\mu}-c \theta_{\{d_1,\lambda,a\}},\mathcal{B}_2,\Omega\Big] >0,
\]
which contradicts \eqref{5.33}. Consequently, $v=0$. Hence, $u\neq 0$. Thus, it follows from
the $u$-equation that
\[
\sigma\Big[d_1{\mathcal{L}}_1+m-\frac{{\lambda}-2a \theta_{\{d_1,\lambda,a\}} +m}{\tau},\mathcal{B}_1,\Omega\Big]\leq 0.
\]
Moreover, by the strict monotonicity of the principal eigenvalue and \eqref{5.32},
\begin{align*}
\sigma\Big[d_1{\mathcal{L}}_1+m-\frac{{\lambda}-2a \theta_{\{d_1,\lambda,a\}} +m}{\tau},\mathcal{B}_1,\Omega\Big]&>\sigma\Big[d_1{\mathcal{L}}_1+m-\frac{{\lambda}-a \theta_{\{d_1,\lambda,a\}} +m}{\tau},\mathcal{B}_1,\Omega\Big] \\ & >\sigma\left[d_1{\mathcal{L}}_1-{\lambda}+a \theta_{\{d_1,\lambda,a\}},\mathcal{B}_1,\Omega\right] =0,
\end{align*}
which again leads to a contradiction. Therefore, $\chi=0$ and
\begin{equation*}
i_{P_{W^{2,\infty}_{\mathcal{B}_1}}\times P_{W^{2,\infty}_{\mathcal{B}_2}}}(\mathcal{H}(1,\cdot,\cdot),(\theta_{\{d_1,\lambda,a\}},0))=1.
\end{equation*}
This ends the proof of the second identity of \eqref{5.25}.
\par
Next, suppose that
\begin{equation}
\label{5.34}
\sigma[d_2{\mathcal{L}}_2-{\mu}-c\theta_{\{d_1,{\lambda},a\}},\mathcal{B}_2,\Omega]<0,
\end{equation}
instead of \eqref{5.27}. We claim that $I-D_{(u,v)}\mathcal{H}(1,\theta_{\{d_1,\lambda,a\}},0)$ is an injective operator on $W^{2,\infty}_{\mathcal{B}_1}\times P_{W^{2,\infty}_{\mathcal{B}_2}}$. Indeed, if there is a $(u,v)\in W^{2,\infty}_{\mathcal{B}_1}\times P_{W^{2,\infty}_{\mathcal{B}_2}}$ such that
$$
  D_{(u,v)}\mathcal{H}(1,\theta_{\{d_1,\lambda,a\}},0)(u,v)=(u,v),
$$
then the identities \eqref{5.28} and \eqref{5.29} hold. Since $v\in P_{W^{2,\infty}_{\mathcal{B}_2}}$, if $v\neq 0$, then
\[
\sigma[d_2{\mathcal{L}}_2-{\mu}-c\theta_{\{d_1,{\lambda},a\}},\mathcal{B}_2,\Omega]=0,
\]
which contradicts \eqref{5.34}. Thus, $v=0$. Similarly, arguing as in the previous case,
by \eqref{5.30} and \eqref{5.31}, one can easily infer that $u=0$. Hence, $I-D_{(u,v)}\mathcal{H}(1,\theta_{\{d_1,\lambda,a\}},0)$ is injective on $W^{2,\infty}_{\mathcal{B}_1}\times P_{W^{2,\infty}_{\mathcal{B}_2}}$. According to Lemma \ref{le5.8} \ref{le5.8.item3}), to complete the proof of \eqref{5.25}, it suffices to show that  $I-D_{(u,v)}\mathcal{H}(1,\theta_{\{d_1,\lambda,a\}},0)$ is not
surjective on $W^{2,\infty}_{\mathcal{B}_1}\times P_{W^{2,\infty}_{\mathcal{B}_2}}$. To prove this, we
proceed by contradiction. Assume that, for every $(w_1,w_2)\in W^{2,\infty}_{\mathcal{B}_1}\times P_{W^{2,\infty}_{\mathcal{B}_2}}$ there exists $(u,v)\in W^{2,\infty}_{\mathcal{B}_1}\times P_{W^{2,\infty}_{\mathcal{B}_2}}$ such that
\[
[I-D_{(u,v)}\mathcal{H}(1,\theta_{\{d_1,\lambda,a\}},0)](u,v)=(w_1,w_2),
\]
i.e.,
\begin{equation*}
u-(d_1{\mathcal{L}}_1+m)^{-1}\left[({\lambda}-2a \theta_{\{d_1,\lambda,a\}} +m)u+b \theta_{\{d_1,\lambda,a\}} v\right]=w_1
\end{equation*}
and
\begin{equation}
\label{5.35}
v-(d_2{\mathcal{L}}_2+m)^{-1}[({\mu}+c \theta_{\{d_1,\lambda,a\}} +m)v]=w_2.
\end{equation}
In particular, since we are assuming that $\sigma[d_2{\mathcal{L}}_2+m,\mathcal{B}_2,\Omega]>0$, for every $w\in P_{W^{2,\infty}_{\mc{B}_2}}\setminus\{0\}$, the function $w_2$ defined by
\[
w_2:=(d_2{\mathcal{L}}_2+m)^{-1}\omega
\]
is strongly positive. For this choice, \eqref{5.35} becomes
\[
(d_2{\mathcal{L}}_2-{\mu}-c \theta_{\{d_1,\lambda,a\}})v=\omega> 0.
\]
This implies that $v>0$ and hence, thanks to Theorem 7.10 of \cite{LG13},
\[
  \sigma[d_2{\mathcal{L}}_2-{\mu}-c \theta_{\{d_1,\lambda,a\}},\mathcal{B}_2,\Omega]>0,
\]
which contradicts \eqref{5.34}. This ends the proof.
\end{proof}

It remains to obtain the fixed point index of the coexistence  states.

\begin{lemma}
\label{le5.10}
Under the hypothesis of Theorem \ref{th5.3},
\begin{equation*}
i_{P_{W^{2,\infty}_{\mathcal{B}_1}}\times P_{W^{2,\infty}_{\mathcal{B}_2}}}(\mathcal{H}(1,\cdot,\cdot),(u,v))=1
\end{equation*}
for every coexistence state, $(u,v)$, of \eqref{5.16}.
\end{lemma}
\begin{proof}
This is an immediate consequence of Corollary \ref{co5.2}, the dominance of the principal eigenvalue of the linearization at $(u,v)$, and the definition of the fixed point index (see \cite[Th. 11.4]{Am76}).
\end{proof}

Now, we have all the ingredients to complete the proof of Theorem \ref{th5.3}.

\subsection*{Proof of Theorem \ref{th5.3}:} Subsequently, we will denote by $\mf{C}$ the set of coexistence
states of \eqref{5.16}. By construction, we already know that $\mf{C}$ lies in the interior of $\mc{U}\times \mc{V}$. Then, by the additivity property of the fixed point index (see \cite[Th. 11.1]{Am76}), we have that
\begin{align*}
i_{P_{W^{2,\infty}_{\mathcal{B}_1}}\times P_{W^{2,\infty}_{\mathcal{B}_2}}} & (\mathcal{H}(1,\cdot,\cdot),\hbox{int}\,(\mathcal{U}\times\mathcal{V}))
=i_{P_{W^{2,\infty}_{\mathcal{B}_1}}\times P_{W^{2,\infty}_{\mathcal{B}_2}}}(\mathcal{H}(1,\cdot,\cdot),(0,0))\\
&+i_{P_{W^{2,\infty}_{\mathcal{B}_1}}\times P_{W^{2,\infty}_{\mathcal{B}_2}}}(\mathcal{H}(1,\cdot,\cdot),(\theta_{\{d_1,\lambda,a\}},0))
+i_{P_{W^{2,\infty}_{\mathcal{B}_1}}\times P_{W^{2,\infty}_{\mathcal{B}_2}}}(\mathcal{H}(1,\cdot,\cdot),(0,\theta_{\{d_2,\mu,d\}}))\\
&+\sum_{(u,v)\in\mf{C}} i_{P_{W^{2,\infty}_{\mathcal{B}_1}}\times P_{W^{2,\infty}_{\mathcal{B}_2}}}(\mathcal{H}(1,\cdot,\cdot),(u,v))
\end{align*}
provided both semitrivial states exist.
Thus, as a straightforward consequence of Lemmas \ref{le5.6}, \ref{le5.7}, \ref{le5.9}, and \ref{le5.10}, the system \eqref{5.16} admits a unique coexistence state in each of the following situations
\begin{enumerate}[(A)]
\item\label{item5.1} If $\sigma[d_1{\mathcal{L}}_1-{\lambda},\mathcal{B}_1,\Omega]<0$, $\sigma[d_2{\mathcal{L}}_2-{\mu},\mathcal{B}_2,\Omega]> 0$ and $\sigma[d_2\mathcal{L}_2-\mu-c\theta_{\{d_1,\lambda,a\}},\mathcal{B}_2,\Omega]<0$.
\item\label{item5.2} If $\sigma[d_1{\mathcal{L}}_1-{\lambda},\mathcal{B}_1,\Omega]<0$ and $\sigma[d_2{\mathcal{L}}_2-{\mu},\mathcal{B}_2,\Omega]< 0$. This is the case where both semitrivial states exist. By the monotonicity of the principal eigenvalue with respect to the potential, in this case,
\[
  \sigma[d_2\mathcal{L}_2-\mu-c\theta_{\{d_1,\lambda,a\}},\mathcal{B}_2,\Omega]<0,\qquad  \sigma[d_1\mathcal{L}_1-\lambda-b\theta_{\{d_2,\mu,d\}},\mathcal{B}_1,\Omega]<0.
\]
\item\label{item5.3} If $\sigma[d_1{\mathcal{L}}_1-{\lambda},\mathcal{B}_1,\Omega]>0$,  $\sigma[d_2{\mathcal{L}}_2-{\mu},\mathcal{B}_2,\Omega]<0$, and
     $\sigma[d_1\mathcal{L}_1-\lambda-b\theta_{\{d_2,\mu,d\}},\mathcal{B}_1,\Omega]<0$.
\end{enumerate}
On the other hand, there are no coexistence states in the following cases:
\begin{enumerate}[(A)]
\setcounter{enumi}{3}
\item\label{item5.4} If $\sigma[d_2{\mathcal{L}}_2-{\mu},\mathcal{B}_2,\Omega]<0$ and $\sigma[d_1\mathcal{L}_1-\lambda-b\theta_{\{d_2,\mu,d\}},\mathcal{B}_1,\Omega]>0$, which implies $\sigma[d_1{\mathcal{L}}_1-{\lambda},\mathcal{B}_1,\Omega]>0$.
\item\label{item5.5} If $\sigma[d_1{\mathcal{L}}_1-{\lambda},\mathcal{B}_1,\Omega]>0$ and $\sigma[d_2{\mathcal{L}}_2-{\mu},\mathcal{B}_2,\Omega]> 0$. In such a case, there are no semitrivial states.
\item\label{item5.6} If $\sigma[d_1{\mathcal{L}}_1-{\lambda},\mathcal{B}_1,\Omega]<0$ and $\sigma[d_2\mathcal{L}_2-\mu-c\theta_{\{d_1,\lambda,a\}},\mathcal{B}_2,\Omega]>0$, which implies
    $\sigma[d_2{\mathcal{L}}_2-{\mu},\mathcal{B}_2,\Omega]> 0$.
\end{enumerate}

All these regions in the parameter space $(\l,\mu)$ have been represented on Figure \ref{Fig6}, where we are assuming the functions $\l$ and $\mu$ to be constant on $\O$.

\begin{figure}[h!]
\centering
\includegraphics[scale=0.22]{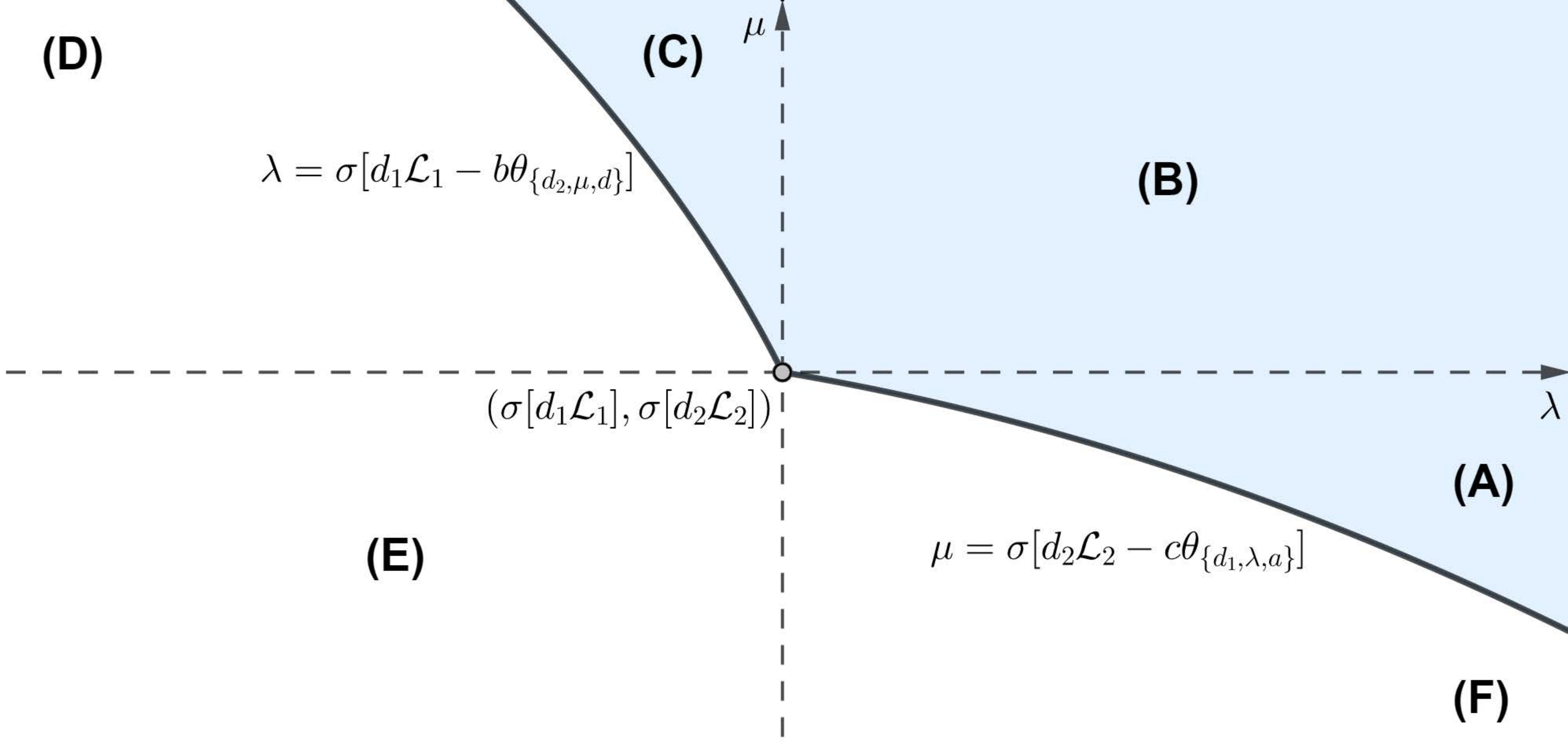}
\caption{Plot of an admissible $(\lambda,\mu)$-plane under the assumptions of Theorem \ref{th5.3}.
 For the sake of simplicity, the boundary conditions and the domain have been omitted from the principal eigenvalues. }
\label{Fig6}
\end{figure}

In Figure \ref{Fig6}, the curves separating the six regions (\ref{item5.1})--(\ref{item5.6}) consist of
the curves of change of stability of the semitrivial positive solutions,
\[
  \sigma[d_1\mathcal{L}_1-\lambda-b\theta_{\{d_2,\mu,d\}},\mathcal{B}_1,\Omega]=0, \qquad
  \sigma[d_2\mathcal{L}_2-\mu -c\theta_{\{d_1,\lambda,a\}},\mathcal{B}_2,\Omega]=0.
\]
together with the coordinate axis
\[
  \sigma[d_1\mathcal{L}_1-\lambda,\mathcal{B}_1,\Omega]=0, \qquad
   \sigma[d_2\mathcal{L}_2-\mu,\mathcal{B}_2,\Omega]=0.
\]
To complete the proof of Theorem \ref{th5.3} it remains to prove the following assertions
on each of these curves:
\begin{enumerate}[(I)]
\item\label{proof5.3.itemI} There are no coexistence states if either $\sigma[d_1{\mathcal{L}}_1-{\lambda},\mathcal{B}_1,\Omega]=0$ and $\sigma[d_2{\mathcal{L}}_2-{\mu},\mathcal{B}_2,\Omega]\ge 0$, or $\sigma[d_2{\mathcal{L}}_2-{\mu},\mathcal{B}_2,\Omega]=0$ and $\sigma[d_1{\mathcal{L}}_1-{\lambda},\mathcal{B}_1,\Omega]\ge 0$.
\item\label{proof5.3.itemII} There are no coexistence states if either
\[
  \sigma[d_1{\mathcal{L}}_1-{\lambda},\mathcal{B}_1,\Omega]<0\quad \hbox{and}\quad \sigma[d_2\mathcal{L}_2-\mu-c\theta_{\{d_1,\lambda,a\}},\mathcal{B}_2,\Omega]=0,
\]
which implies $\sigma[d_2{\mathcal{L}}_2-{\mu},\mathcal{B}_2,\Omega]> 0$, or else
\[
  \sigma[d_2{\mathcal{L}}_2-{\mu},\mathcal{B}_2,\Omega]<0\quad\hbox{and}\quad
  \sigma[d_1\mathcal{L}_1-\lambda-b\theta_{\{d_2,\mu,d\}},\mathcal{B}_1,\Omega]=0,
\]
which implies $\sigma[d_1{\mathcal{L}}_1-{\lambda},\mathcal{B}_1,\Omega]> 0$.
\item\label{proof5.3.itemIII} There exists a unique coexistence state if either
\[
  \sigma[d_1{\mathcal{L}}_1-{\lambda},\mathcal{B}_1,\Omega]=0\quad \hbox{and}\quad  \sigma[d_2{\mathcal{L}}_2-{\mu},\mathcal{B}_2,\Omega]< 0,
\]
which implies $\sigma[d_1\mathcal{L}_1-\lambda-b\theta_{\{d_2,\mu,d\}},\mathcal{B}_1,\Omega]<0$, or else \[
   \sigma[d_2{\mathcal{L}}_2-{\mu},\mathcal{B}_2,\Omega] =0\quad \hbox{and}\quad  \sigma[d_1{\mathcal{L}}_1-{\lambda},\mathcal{B}_1,\Omega]< 0,
\]
which implies $\sigma[d_2\mathcal{L}_2-\mu-c\theta_{\{d_1,\lambda,a\}},\mathcal{B}_2,\Omega]<0$.
\end{enumerate}
The proof of the second assertions on each of these cases follow by symmetry form the first ones. To prove the cases (\ref{proof5.3.itemI}) and (\ref{proof5.3.itemII}) we first show that  there is $\varepsilon_{0}>0$ such that
\begin{equation}
\label{5.36}
\left\{
\begin{array}{ll}
d_1(x)\mathcal{L}_1 u=(\lambda+\varepsilon_1) u-au^2+ b uv & \hbox{in}\;\;\Omega,\\
d_2(x)\mathcal{L}_2 v=(\mu+\varepsilon_2) v-dv^2+ c uv & \hbox{in}\;\;\Omega,\\
\mathcal{B}_1 u=\mathcal{B}_2 v=0 & \hbox{on}\;\;\partial\Omega
\end{array}
\right.
\end{equation}
admits a coexistence state for every $\varepsilon_1,\varepsilon_2\in (-\varepsilon_0,\varepsilon_0)$
if \eqref{5.16} admits a coexistence state. This is a direct consequence of the implicit function theorem,
for as  the linearization of \eqref{5.36} with respect to $(u,v)$,  particularized at $(\e_1,\e_2)=(0,0)$,   at the coexistence state of \eqref{5.16}  is invertible, because, due to Corollary \ref{co5.2}, its principal eigenvalue is positive. Actually,  the implicit function theorem provides us with a smooth surface of coexistence states for sufficiently small  $\varepsilon_1, \varepsilon_2$.
\par
Suppose we are in Case (\ref{proof5.3.itemI}), i.e., $\sigma[d_1{\mathcal{L}}_1-{\lambda},\mathcal{B}_1,\Omega]=0$ and $\sigma[d_2{\mathcal{L}}_2-{\mu},\mathcal{B}_2,\Omega]\ge 0$, and that,  arguing by contradiction,  the problem \eqref{5.16} admits some coexistence  state. Then, $\varepsilon_0>0$ exists such that \eqref{5.36} admits a coexistence  state for each $\varepsilon_1,\varepsilon_2\in [-\varepsilon_0,\varepsilon_0]$. Consider \eqref{5.36} with $\varepsilon_1=\varepsilon_2=-\varepsilon_0$. Then, by the monotonicity of the principal eigenvalue with respect to the potential,
\begin{equation*}
\sigma[d_1{\mathcal{L}}_1-{\lambda}+\varepsilon_0,\mathcal{B}_1,\Omega]>0\quad\hbox{and}\quad \sigma[d_2{\mathcal{L}}_2-{\mu}+\varepsilon_0,\mathcal{B}_2,\Omega]> 0.
\end{equation*}
Thus, the problem \eqref{5.36} with $\varepsilon_1=\varepsilon_2=-\varepsilon_0$ fits Case (\ref{item5.5}) above, for which we already know that a coexistence state should not exist. This contradiction ends the proof of Case (\ref{proof5.3.itemI}).
\par
Now, assume that we are in Case (\ref{proof5.3.itemII}), i.e., $\sigma[d_1{\mathcal{L}}_1-{\lambda},\mathcal{B}_1,\Omega]<0$ and $\sigma[d_2\mathcal{L}_2-\mu-c\theta_{\{d_1,\lambda,a\}},\mathcal{B}_2,\Omega]=0$ (so,  $\sigma[d_2{\mathcal{L}}_2-{\mu},\mathcal{B}_2,\Omega]> 0$), and that the problem \eqref{5.16} admits some coexistence state. Arguing as before, there is $\varepsilon_0>0$ such that \eqref{5.36} admits a coexistence state for each $\varepsilon_1,\varepsilon_2\in [-\varepsilon_0,\varepsilon_0]$. Consider \eqref{5.36} with $\varepsilon_1=0$ and $\varepsilon_2=-\varepsilon_0$. Thanks to the monotonicity of the principal eigenvalue with respect to the potential, we have that
\begin{equation*}
\sigma[d_2\mathcal{L}_2-(\mu-\varepsilon_0)-c\theta_{\{d_1,\lambda,a\}},\mathcal{B}_2,\Omega]>
\sigma[d_2\mathcal{L}_2-\mu-c\theta_{\{d_1,\lambda,a\}},\mathcal{B}_2,\Omega]=0.
\end{equation*}
Hence, the problem \eqref{5.36}, with $\varepsilon_1=0$ and $\varepsilon_2=-\varepsilon_0$, fits Case  (\ref{item5.6}) above, for which we already know that should not exist any coexistence state. This contradiction ends  the proof.
\par
It remains to prove Case (\ref{proof5.3.itemIII}). Assume that
\begin{equation}
\label{5.37}
\sigma[d_1{\mathcal{L}}_1-{\lambda},\mathcal{B}_1,\Omega]=0\quad\hbox{and}\quad\sigma[d_2{\mathcal{L}}_2-{\mu},\mathcal{B}_2,\Omega]< 0.
\end{equation}
Then,
\begin{equation}
\label{5.38}
\sigma[d_1\mathcal{L}_1-\lambda-b\theta_{\{d_2,\mu,d\}},\mathcal{B}_1,\Omega]<0.
\end{equation}
The coexistence state will be constructed through a compactness argument. For every integer $n\geq 1$, consider the problem
\begin{equation}
\label{5.39}
\left\{
\begin{array}{ll}
d_1(x)\mathcal{L}_1 u=(\lambda+\frac{1}{n}) u-au^2+ b uv & \hbox{in}\;\;\Omega,\\
d_2(x)\mathcal{L}_2 v=\mu v-dv^2+ c uv & \hbox{in}\;\;\Omega,\\
\mathcal{B}_1 u=\mathcal{B}_2 v=0 & \hbox{on}\;\;\partial\Omega.
\end{array}
\right.
\end{equation}
Since
\begin{equation*}
\sigma[d_2{\mathcal{L}}_2-{\mu},\mathcal{B}_2,\Omega]< 0\quad\hbox{and}\quad
\sigma\Big[ d_1{\mathcal{L}}_1-\Big({\lambda}+\frac{1}{n}\Big),\mathcal{B}_1,\Omega\Big]<0,
\end{equation*}
\eqref{5.39} satisfies  (\ref{item5.2}) above, and hence, it admits a unique coexistence state, $(u_n,v_n)$. By the proof of Lemma \ref{le5.4}, these coexistence states are uniformly bounded in $n\geq 1$. Moreover, for every $n\geq 1$,
\begin{align*}
u_n & =(d_1(x){\mathcal{L}}_1+m)^{-1} \left[ \Big({\lambda}+\frac{1}{n}+m-a u_n+ b v_n\Big) u_n\right],\\
v_n & =(d_2(x){\mathcal{L}}_2+m)^{-1}\left[ ({\mu}+m-d v_n+ c u_n) v_n\right].
\end{align*}
On the other hand, by the choice of $m$, $(d_1{\mathcal{L}}_1+m)^{-1}$ and $(d_2{\mathcal{L}}_2+m)^{-1}$ are strongly positive compact operators. Hence, $(u_n,v_n)$ admits a convergent subsequence, renamed by $(u_n,v_n)$, that converges to some $(u_0,v_0)$ as $n\uparrow +\infty$. Necessarily, $(u_0,v_0)\in W_{\mathcal{B}_1}^{2,\infty}\times W_{\mathcal{B}_2}^{2,\infty}$, $u_0,v_0\ge 0$ and $(u_0,v_0)$ satisfies
\begin{align*}
u_0 & =(d_1(x){\mathcal{L}}_1+m)^{-1} [({\lambda}+m-a u_0+ b v_0) u_0], \\
v_0 & =(d_2(x){\mathcal{L}}_2+m)^{-1} [({\mu}+m-d v_0+ c u_0) v_0],
\end{align*}
i.e., $(u_0,v_0)$ is a component-wise non-negative solution of \eqref{5.16}. Furthermore, by the Arzel\`a --Ascoli Theorem, the convergence is uniform $\bar\Omega$. Thus, either $(u_0,v_0)=(0,0)$, or $(u_0,v_0)=(0,\theta_{\{d_2,\mu,d\}})$, or $(u_0,v_0)$ is a coexistence state, for as $\t_{\{d_1,\l,a\}}=0$.
Since $(u_n,v_n)$ is a coexistence state of \eqref{5.39}, we have that, for every $n\geq 1$,
\begin{align*}
\sigma\Big[d_1(x)\mathcal{L}_1-\Big(\lambda+\frac{1}{n}\Big)+a u_n-b v_n,\mathcal{B}_1,\Omega\Big]& =0,\\
\sigma[d_2(x)\mathcal{L}_2 -\mu +d v_n- c u_n,\mathcal{B}_2,\Omega] & =0.
\end{align*}
Moreover, as the convergence is uniform on $\bar \O$, by the continuous dependence of the principal eigenvalues with respect to the potential, it becomes apparent that
\begin{equation*}
\sigma[d_2(x)\mathcal{L}_2 -\mu,\mathcal{B}_2,\Omega]=0
\end{equation*}
if $(u_0,v_0)=(0,0)$, which contradicts \eqref{5.37}. Similarly,
\begin{equation*}
\sigma[d_1(x)\mathcal{L}_1-\lambda-b \theta_{\{d_2,\mu,d\}},\mathcal{B}_1,\Omega]=0
\end{equation*}
if $(u_0,v_0)=(0,\theta_{\{d_2,\mu,d\}})$, which contradicts \eqref{5.38}. Therefore, $(u_0,v_0)$ is a coexistence state of \eqref{5.16} and the existence is established.
\par
The uniqueness can be easily derived with the following argument. Suppose \eqref{5.37} holds and  \eqref{5.16} admits two coexistence states, $(u_j,v_j)$, $j=1,2$. Then, since they are linearly stable, they must be  non-degenerate and hence applying the implicity function, it becomes apparent that the problem \eqref{5.36} possesses at least two coexistence states  for sufficiently small $\e_1$ and $\e_2$.  By choosing $\e_2=0$ and $\e_1>0$ we can enter on region (\ref{item5.2}), where we already know that the system admits a unique coexistence state. This contradiction shows the uniqueness.
\par
To conclude the proof of Theorem \ref{th5.3}, it suffices to observe  that whenever there exists a unique coexistence state which is linearly stable,  it is a global attractor for the component-wise positive solutions of the parabolic model, by compressiveness (see \cite[Rm. 33.2 \& Th. 33.3]{Hess}).
\par
According to Theorem \ref{th5.3}, when $(\lambda,\mu)\in\mathbb{R}^2$ is located in the shaded region
of Figure \ref{Fig6} (regions (\ref{item5.1}), (\ref{item5.2}), (\ref{item5.3}) and the semi-axis between them), then \eqref{5.16} admits a unique coexistence state which is a global attractor for the component-wise positive solutions. Moreover, if $(\lambda,\mu)$ belongs to either regions (\ref{item5.4}) or (\ref{item5.3}), then the stable semitrivial solution is actually a global attractor for the component-wise positive solutions since there are no coexistence states and the system is compressive (see \cite[Rm. 33.2 \& Th. 33.3]{Hess}). Finally, in region (\ref{item5.5}) the trivial solution must be a global attractor with respect to the positive solutions of the parabolic counterpart of \eqref{5.16}.

\section{Uniqueness of coexistence states in competitive systems}
\label{sec-6}

\noindent Throughout this section we maintain the notations introduced in Section \ref{sec-5} for the uniformly elliptic differential operators, $\mathcal{L}_i$, the mixed boundary operators, $\mathcal{B}_i$, as well as for the functions they are made of, i.e., $A_i\in\mathcal{M}_N^\mathrm{sym}(\mathcal{C}^1(\bar\Omega))$, $C_i\in\mathcal{C}(\bar\Omega)$ and $\beta_i\in\mathcal{C}(\Gamma_{\mathcal{R}}^{i})$, $i\in\{1,2\}$. The competition model of Lotka--Volterra  is given by
\begin{equation}
\label{6.1}
\left\{
\begin{array}{ll}
d_1(x)\mathcal{L}_1 u=\lambda u-au^2 - buv & \mbox{in}\;\;\Omega,\\
d_2(x)\mathcal{L}_2 v=\mu v-d v^2 - cuv & \mbox{in}\;\;\Omega,\\
\mathcal{B}_1 u=\mathcal{B}_2 v = 0 & \mbox{on}\;\;\partial\Omega,
\end{array}
\right.
\end{equation}
where $d_1,d_2\in\mathcal{C}(\bar\Omega;(0,+\infty))$ are the diffusion rates of the species, $\lambda,\mu\in\mathcal{C}(\bar\Omega)$ stand for their growth rates, $a,d\in\mathcal{C}(\bar\Omega;(0,+\infty))$ are their intra-specific competition coefficients, and $b,c\in\mathcal{C}(\bar\Omega;(0,+\infty))$ their inter-specific competition rates. It is said that the species exhibit \emph{low competition} if
\begin{equation*}
bc\lneq ad\quad\mbox{in}\;\;\Omega.
\end{equation*}
The next result provides us with a  counterpart of Theorem \ref{th5.1} and Corollary \ref{co5.2} for competition models. Some of the comments on Section \ref{sec-5} after the statement of Theorem \ref{th5.1} also apply here.

\begin{theorem}
\label{th6.1}
Assume that $\kappa=\frac{bc}{ad}\lneq 1$ in $\Omega$ and
\begin{equation*}
\max_{\bar{\Omega}}\left( \frac{ad^2}{c^3}\frac{d_2}{d_1}F_{-}(\kappa)\right) \le \min_{\bar{\Omega}}\left( \frac{ad^2}{c^3}\frac{d_2}{d_1}F_{+}(\kappa)\right),
\end{equation*}
where $F_\pm$ are the functions defined in  \eqref{5.2}. Then, every coexistence state of \eqref{6.1}
is linearly stable. In particular, there exist functions $d_1(x),d_2(x)\in\mathcal{C}(\bar\Omega;(0,+\infty))$, as close to zero as desired, such that every coexistence  state of \eqref{6.1} is linearly stable.
\end{theorem}
\begin{proof}
Consider the elliptic problem
\begin{equation}
\label{6.2}
\left\{
\begin{array}{ll}
\mathcal{L}_1 u=\tilde{\lambda} u-\tilde{a}u^2 - \tilde{b}uv & \mbox{in}\;\;\Omega,\\
\mathcal{L}_2 v=\tilde{\mu} v-\tilde{d} v^2 - \tilde{c}uv & \mbox{in}\;\;\Omega,\\
\mathcal{B}_1 u=\mathcal{B}_2 v = 0 & \mbox{on}\;\;\partial\Omega,
\end{array}
\right.
\end{equation}
obtained from \eqref{6.1} by performing the changes
\begin{equation*}
\tilde{\lambda}=\frac{\lambda}{d_1},\quad \tilde{a}=\frac{a}{d_1},\quad\tilde{b}=\frac{b}{d_1}, \quad
\tilde{\mu}=\frac{\mu}{d_2},\quad \tilde{c}=\frac{c}{d_2},\quad\hbox{and}\quad\tilde{d}=\frac{d}{d_2}.
\end{equation*}
Let $(u,v)$ be a coexistence state of \eqref{6.2}, and so, of \eqref{6.1}. Linearizing \eqref{6.2} at $(u,v)$ yields
\begin{equation*}
\left\{
\begin{array}{ll}
[\mathcal{L}_1 -\tilde\lambda +2\tilde a u+ \tilde b v]\varphi+ \tilde b u\psi=\sigma \varphi & \hbox{in}\;\;\Omega,\\[0ex]
[\mathcal{L}_2 -\tilde\mu +2\tilde d v + \tilde c u]\psi+ \tilde c v\varphi= \sigma \psi & \hbox{in}\;\;\Omega,\\[0ex]
\mathcal{B}_1 \varphi=\mathcal{B}_2 \psi=0 & \hbox{on}\;\;\partial\Omega,
\end{array}
\right.
\end{equation*}
which is an eigenvalue problem for a quasi-cooperative operator. Thus, by making the change $\tilde\varphi:=\varphi$ and $\tilde\psi:=-\psi$, the previous eigenvalue problem is equivalent to the next one
\begin{equation}
\label{6.3}
\left\{
\begin{array}{ll}
[\mathcal{L}_1 -\tilde\lambda +2\tilde a u+\tilde b v]\tilde\varphi- \tilde b u\tilde\psi=\sigma\tilde\varphi & \hbox{in}\;\;\Omega,\\[0ex]
[\mathcal{L}_2 -\tilde \mu +2\tilde d v + \tilde c u]\tilde\psi- \tilde c v\tilde\varphi=\sigma \tilde\psi & \hbox{in}\;\;\Omega,\\[0ex]
\mathcal{B}_1 \tilde\varphi=\mathcal{B}_2 \tilde\psi=0 & \hbox{on}\;\;\partial\Omega,
\end{array}
\right.
\end{equation}
of cooperative type, which shares the properties of \eqref{5.4}. In particular, there exists a unique principal eigenvalue, $\sigma_0$, associated with it there is an eigenfunction, $(\tilde\varphi,\tilde\psi)$, such that $\tilde\varphi\gg 0$ and $\tilde\psi\gg 0$. It should be noted that multiplying the equations in \eqref{6.3} by $u$ and $v$, respectively, and particularizing at $\sigma=\sigma_0$, yields
\begin{align*}
u\mathcal{L}_1\tilde\varphi-\tilde\varphi\mathcal{L}_1 u+u^2(\tilde{a} \tilde\varphi-\tilde{b}\tilde\psi)& =\sigma_0 u \tilde\varphi, \\ v\mathcal{L}_2\tilde\psi-\tilde\psi\mathcal{L}_2 v+v^2(\tilde{d} \tilde\psi-\tilde{c}\tilde\varphi) & =\sigma_0 v\tilde\psi,
\end{align*}
which mimic the equations \eqref{5.5} and \eqref{5.6}, respectively. The rest of the proof follows identical patterns as the proof of Theorem \ref{th5.1} and Corollary \ref{co5.2}.
\end{proof}

As in Theorem \ref{th5.3}, the next result shows how the linear stability of the semitrivial
solutions determine the existence of coexistence states.

\begin{theorem}
\label{th6.2}
Assume that $d_1,d_2\in\mathcal{C}(\bar\Omega;(0,+\infty))$ and $\kappa=\frac{bc}{ad}\lneq 1$ in $\Omega$. If
\begin{equation*}
\max_{\bar{\Omega}}\left( \frac{ad^2}{c^3}\frac{d_2}{d_1}F_{-}(\kappa)\right)\le \min_{\bar{\Omega}}\left(\frac{ad^2}{c^3}\frac{d_2}{d_1}F_{+}(\kappa)\right),
\end{equation*}
then:
\begin{enumerate}[{\rm a)}]
\item\label{th6.2.item1} If both semitrivial solutions exist and are linearly unstable, then \eqref{6.1} admits a unique coexistence state.  Moreover, it is linearly stable.
\item\label{th6.2.item2} In any other case the system \eqref{6.1} does not admit any coexistence state.
\item\label{th6.2.item3} Both semitrivial solutions  cannot be linearly stable simultaneously.
\item\label{th6.2.item4} If a solution of \eqref{6.1} is linearly stable, then it is a global attractor for the component-wise positive solutions of the parabolic counterpart of \eqref{6.1}.
\end{enumerate}
\end{theorem}

Even though the proof of this result follows the same general patterns as the proof of Theorem \ref{th5.3}, it is imperative to construct the bounded open set containing all the non-negative solutions in order to apply  the fixed point index in cones. The next result accomplished this task.

\begin{lemma}
\label{le6.3}
There exits a bounded set $\mathcal{U}\times\mathcal{V}\subset W^{2,\infty}_{\mathcal{B}_1}\times W^{2,\infty}_{\mathcal{B}_2}$, independent of $b$ and $c$, such that $(u,v)\in\mathrm{int}\;(\mathcal{U}\times \mathcal{V})$ if $(u,v)$ is a solution of \eqref{6.1} with $(u,v)\in P_{W^{2,\infty}_{\mathcal{B}_1}}\times \,P_{W^{2,\infty}_{\mathcal{B}_2}}$.
\end{lemma}
\begin{proof}
It suffices to note that, if $(u,v)$ is a solution of \eqref{6.1} with $(u,v)\in P_{W^{2,\infty}_{\mathcal{B}_1}}\times \,P_{W^{2,\infty}_{\mathcal{B}_2}}$, then
\begin{equation*}
\left\{
\begin{array}{ll}
d_1\mathcal{L}_1 u=\lambda u-a u^2 - b u v\le \lambda u - a u^2 & \hbox{in}\;\;\Omega,\\
\mathcal{B}_1 u=0 & \hbox{on}\;\;\partial\Omega.
\end{array}
\right.
\end{equation*}
Similarly,
\begin{equation*}
\left\{
\begin{array}{ll}
d_2\mathcal{L}_2 v=\mu v-d v^2 - c u v\le \mu v - d v^2 & \hbox{in}\;\;\Omega,\\
\mathcal{B}_2 v=0 & \hbox{on}\;\;\partial\Omega.
\end{array}
\right.
\end{equation*}
Thus, $u$ and $v$ are subsolutions of the associated logistic bvp's. Moreover,
$\kappa\Psi$ provides us with supersolutions of these problems for sufficiently large $\kappa>1$, where
$\Psi$ should be chosen as in the proof of Lemma \ref{le5.4}.
Therefore, by the uniqueness of solution to these problems, it follows from \cite[Th. 7.10]{LG13}, or  Lemma 3.4 of \cite{FKLM}, that
\begin{equation*}
0\le u\le \theta_{\{d_1,\lambda,a\}}\quad \hbox{and}\quad 0\le v\le \theta_{\{d_2,\mu,d\}}\quad\hbox{in}\;\;\Omega.
\end{equation*}
This ends the proof.
\end{proof}

Subsequently, we consider $\mathcal{U}\times\mathcal{V}$, the bounded set provided by Lemma \ref{le6.3}. Let $m>0$ be large enough so that
\begin{equation*}
\sigma[d_1\mathcal{L}_1+m,\mathcal{B}_1,\Omega]>1, \qquad
\sigma[d_2\mathcal{L}_2+m,\mathcal{B}_2,\Omega]>1,
\end{equation*}
\begin{equation*}
\lambda-a u- b v+m>0\quad\hbox{and}\quad
\mu-dv-c u+m>0\quad\hbox{in}\;\;\bar\Omega\quad \hbox{for all} \;\;u\in\mathcal{U},\;\; v\in\mathcal{V},
\end{equation*}
and consider the operator $\mc{I}:\mathcal{U}\times\mathcal{V}\to W^{2,\infty}_{\mathcal{B}_1}\times W^{2,\infty}_{\mathcal{B}_2}$ defined through
\begin{equation*}
\mc{I}(u,v):=\left(
\begin{array}{ll}
(d_1{\mathcal{L}}_1+m)^{-1}[({\lambda}-au- bv+m)u]\\
(d_2{\mathcal{L}}_2+m)^{-1}[({\mu}-dv- cu+m)v]
\end{array}
\right),
\end{equation*}
which is a compact order preserving operator, by our assumptions on $m$. Moreover, its fixed points are  the solutions of \eqref{6.1}. The next result provides us with the fixed point indices of the trivial, semitrivial and coexistence states of \eqref{6.1}. As the proof follows, almost \emph{mutatis mutandis}, the patterns of Lemmas \ref{le5.6}, \ref{le5.7}, \ref{le5.9} and \ref{le5.10}, it will be omitted here by repetitive.

\begin{lemma}
\label{le6.4}
The following statements hold:
\begin{enumerate}[{\rm a)}]
\item\label{le6.4.item1} The total fixed point index is given by
\begin{equation*}
i_{P_{W^{2,\infty}_{\mathcal{B}_1}}\times P_{W^{2,\infty}_{\mathcal{B}_2}}}(\mc{I}(\cdot,\cdot),\mathrm{int}\,(\mathcal{U}\times\mathcal{V}))=1.
\end{equation*}
\item\label{le6.4.item2} If $\sigma[d_1\mathcal{L}_1-\lambda,\mathcal{B}_1,\Omega]\neq 0$ and $\sigma[d_2\mathcal{L}_2-\mu,\mathcal{B}_2,\Omega]\neq 0$, then
\begin{equation*}
i_{P_{W^{2,\infty}_{\mathcal{B}_1}}\times P_{W^{2,\infty}_{\mathcal{B}_2}}}(\mc{I}(\cdot,\cdot),(0,0))=
\left\{
\begin{array}{ll}
1 & \hbox{if}\;\; \sigma[d_1{\mathcal{L}}_1-{\lambda},\mathcal{B}_1,\Omega]>0\;\; \hbox{and}\;\; \sigma[d_2{\mathcal{L}}_2-{\mu},\mathcal{B}_2,\Omega]> 0,\\
0 & \hbox{if}\;\; \sigma[d_1{\mathcal{L}}_1-{\lambda},\mathcal{B}_1,\Omega]<0,\;\; \hbox{or}\;\; \sigma[d_2{\mathcal{L}}_2-{\mu},\mathcal{B}_2,\Omega]< 0.
\end{array}
\right.
\end{equation*}
\item\label{le6.4.item3} The indices of the semitrivial solutions are
\begin{equation*}
i_{P_{W^{2,\infty}_{\mathcal{B}_1}}\times P_{W^{2,\infty}_{\mathcal{B}_2}}}(\mc{I}(\cdot,\cdot),(\theta_{\{d_1,\lambda,a\}},0))=
\left\{
\begin{array}{ll}
0 & \hbox{if}\;\;\sigma[d_2\mathcal{L}_2-\mu+c\theta_{\{d_1,\lambda,a\}},\mathcal{B}_2,\Omega]<0,\\
1 & \hbox{if}\;\;\sigma[d_2\mathcal{L}_2-\mu+c\theta_{\{d_1,\lambda,a\}},\mathcal{B}_2,\Omega]>0,
\end{array}
\right.
\end{equation*}
\begin{equation*}
i_{P_{W^{2,\infty}_{\mathcal{B}_1}}\times P_{W^{2,\infty}_{\mathcal{B}_2}}}(\mc{I}(\cdot,\cdot),(0,\theta_{\{d_2,\mu,d\}}))=
\left\{
\begin{array}{ll}
0 & \hbox{if}\;\;\sigma[d_1\mathcal{L}_1-\lambda+b\theta_{\{d_2,\mu,d\}},\mathcal{B}_1,\Omega]<0,\\
1 & \hbox{if}\;\;\sigma[d_1\mathcal{L}_1-\lambda+b\theta_{\{d_2,\mu,d\}},\mathcal{B}_1,\Omega]>0.
\end{array}
\right.
\end{equation*}
\item\label{le6.4.item4} Under the hypothesis of Theorem \ref{th6.2}, for every coexistence state, $(u,v)$, of \eqref{6.1},
\begin{equation*}
i_{P_{W^{2,\infty}_{\mathcal{B}_1}}\times P_{W^{2,\infty}_{\mathcal{B}_2}}}(\mc{I}(\cdot,\cdot),(u,v))=1.
\end{equation*}
\end{enumerate}
\end{lemma}

\subsection*{Proof of Theorem \ref{th6.2}:} As a direct consequence of the additivity property of the fixed point index (see \cite[Th.11.1]{Am76}) and the conclusions of Lemma \ref{le6.4}, the problem \eqref{6.1} admits a unique coexistence state in the next case
\begin{enumerate}[(A)]
\item\label{item6.1} If $\sigma[d_1{\mathcal{L}}_1-{\lambda},\mathcal{B}_1,\Omega]<0$, $\sigma[d_2{\mathcal{L}}_2-{\mu},\mathcal{B}_2,\Omega]< 0$, and
 \[
  \sigma[d_2\mathcal{L}_2-\mu+c\theta_{\{d_1,\lambda,a\}},\mathcal{B}_2,\Omega]<0,\qquad \sigma[d_1\mathcal{L}_1-\lambda+b\theta_{\{d_2,\mu,d\}},\mathcal{B}_1,\Omega]<0,
 \]
i.e., if both semitrivial states do exist and are linearly unstable.
\end{enumerate}
On the other hand, \eqref{6.1} cannot admit a coexistence state in each of the following cases:
\begin{enumerate}[(A)]
\setcounter{enumi}{1}
\item\label{item6.2} If $\sigma[d_1{\mathcal{L}}_1-{\lambda},\mathcal{B}_1,\Omega]<0$, $\sigma[d_2{\mathcal{L}}_2-{\mu},\mathcal{B}_2,\Omega]< 0$, and
\[
  \sigma[d_2\mathcal{L}_2-\mu+c\theta_{\{d_1,\lambda,a\}},\mathcal{B}_2,\Omega]<0,\qquad  \sigma[d_1\mathcal{L}_1-\lambda+b\theta_{\{d_2,\mu,d\}},\mathcal{B}_1,\Omega]>0.
\]
\item\label{item6.3} If $\sigma[d_1{\mathcal{L}}_1-{\lambda},\mathcal{B}_1,\Omega]>0$ and $\sigma[d_2{\mathcal{L}}_2-{\mu},\mathcal{B}_2,\Omega]< 0$, which implies $\sigma[d_1\mathcal{L}_1-\lambda+b\theta_{\{d_2,\mu,d\}},\mathcal{B}_1,\Omega]>0$.
\item\label{item6.4} If $\sigma[d_1{\mathcal{L}}_1-{\lambda},\mathcal{B}_1,\Omega]>0$, $\sigma[d_2{\mathcal{L}}_2-{\mu},\mathcal{B}_2,\Omega]> 0$, i.e., there are no semitrivial states.
\item\label{item6.5} If $\sigma[d_1{\mathcal{L}}_1-{\lambda},\mathcal{B}_1,\Omega]<0$ and $\sigma[d_2{\mathcal{L}}_2-{\mu},\mathcal{B}_2,\Omega]> 0$, which implies $\sigma[d_2\mathcal{L}_2-\mu+c\theta_{\{d_1,\lambda,a\}},\mathcal{B}_2,\Omega]>0$.
\item\label{item6.6} If $\sigma[d_1{\mathcal{L}}_1-{\lambda},\mathcal{B}_1,\Omega]<0$, $\sigma[d_2{\mathcal{L}}_2-{\mu},\mathcal{B}_2,\Omega]< 0$, and
\[
  \sigma[d_2\mathcal{L}_2-\mu+c\theta_{\{d_1,\lambda,a\}},\mathcal{B}_2,\Omega]>0,\qquad
\sigma[d_1\mathcal{L}_1-\lambda+b\theta_{\{d_2,\mu,d\}},\mathcal{B}_1,\Omega]<0.
\]
\end{enumerate}
By the additivity property of the fixed point index, should they exist, the semitrivial solutions cannot be simultaneously linearly stable. All these regions have been represented in Figure \ref{Fig7} in the special case when $\l$ and $\mu$ are positive constants.

\begin{figure}[h!]
\centering
\includegraphics[scale=0.22]{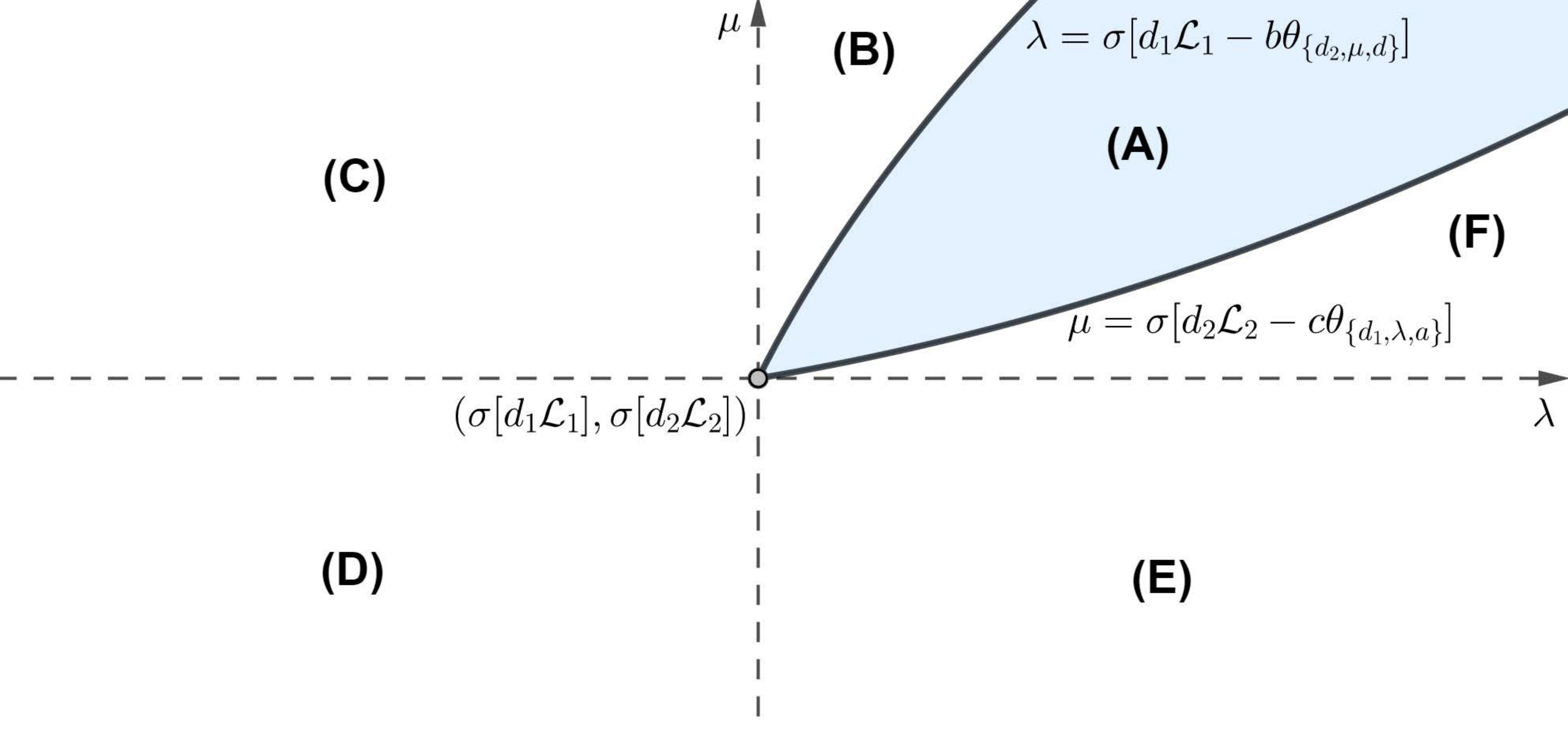}
\caption{Plot of an admissible  $(\lambda,\mu)$-plane for a problem whose coefficients satisfy the hypothesis of Theorem \ref{th6.2}. For the sake of simplicity, the boundary condition and the domain have been omitted in the principal eigenvalues.}
\label{Fig7}
\end{figure}

It remains to make sure that there are no coexistence states on each of the following limiting situation cases:
\begin{enumerate}[(I)]
\item\label{proofth6.2.itemI} $\sigma[d_1{\mathcal{L}}_1-{\lambda},\mathcal{B}_1,\Omega]<0$ and $\sigma[d_2{\mathcal{L}}_2-{\mu},\mathcal{B}_2,\Omega]= 0$, or  $\sigma[d_1{\mathcal{L}}_1-{\lambda},\mathcal{B}_1,\Omega]=0$ and $\sigma[d_2{\mathcal{L}}_2-{\mu},\mathcal{B}_2,\Omega]< 0$.
\item\label{proofth6.2.itemII} $\sigma[d_1{\mathcal{L}}_1-{\lambda},\mathcal{B}_1,\Omega]\ge 0$ and $\sigma[d_2{\mathcal{L}}_2-{\mu},\mathcal{B}_2,\Omega]= 0$, or  $\sigma[d_1{\mathcal{L}}_1-{\lambda},\mathcal{B}_1,\Omega]=0$ and $\sigma[d_2{\mathcal{L}}_2-{\mu},\mathcal{B}_2,\Omega]\ge 0$.
\item\label{proofth6.2.itemIII} $\sigma[d_1{\mathcal{L}}_1-{\lambda},\mathcal{B}_1,\Omega]<0$ and $\sigma[d_2{\mathcal{L}}_2-{\mu}+c \theta_{\{d_1,\lambda,a\}},\mathcal{B}_2,\Omega]= 0$, or  $\sigma[d_2{\mathcal{L}}_2-{\mu},\mathcal{B}_2,\Omega]< 0$ and $\sigma[d_1{\mathcal{L}}_1-{\lambda}+b\theta_{\{d_2,\mu,d\}},\mathcal{B}_1,\Omega]=0$.
\end{enumerate}
In each of these cases, the non-existence result is an immediate consequence of the implicit function theorem. Indeed, suppose that \eqref{6.1} admits a coexistence state, $(u_0,v_0)$, and, for every  $\varepsilon_1,\varepsilon_2\in \mathbb{R}$, consider the problem
\begin{equation}
\label{6.4}
\left\{
\begin{array}{ll}
d_1(x)\mathcal{L}_1 u=(\lambda+\varepsilon_1) u-au^2- b uv & \hbox{in}\;\;\Omega,\\
d_2(x)\mathcal{L}_2 v=(\mu+\varepsilon_2) v-dv^2- c uv & \hbox{in}\;\;\Omega,\\
\mathcal{B}_1 u=\mathcal{B}_2 v=0 & \hbox{on}\;\;\partial\Omega.
\end{array}
\right.
\end{equation}
The coexistence state $(u_0,v_0)$ solves \eqref{6.4} for the choice $(\varepsilon_1,\varepsilon_2)=(0,0)$. Moreover, by Theorem \ref{th6.1}, $(u_0,v_0)$ is linearly stable and, hence, non-degenerate. Thus, by the implicit function theorem, much like in the proof of  Theorem \ref{th5.3}, there exists $\varepsilon_0$ such that \eqref{6.4} admits a coexistence state for all $\varepsilon_1,\varepsilon_2\in[-\varepsilon_0,\varepsilon_0]$. And, due to Theorem \ref{th6.1}, these coexistence states are linearly stable.
\par
Suppose that Case (\ref{proofth6.2.itemI}) holds, i.e.,
\begin{equation*}
\sigma[d_1{\mathcal{L}}_1-{\lambda},\mathcal{B}_1,\Omega]<0\quad\hbox{and}\quad
\sigma[d_2{\mathcal{L}}_2-{\mu},\mathcal{B}_2,\Omega]= 0,
\end{equation*}
and the problem \eqref{6.1} has a coexistence state. Then, by the implicit function theorem,
\eqref{6.4} admits a coexistence state for $\varepsilon_1=0$ and $\varepsilon_2=-\varepsilon_0$ for sufficiently small $\varepsilon_0>0$. However, since
\begin{equation*}
\sigma[d_2{\mathcal{L}}_2-({\mu}-\varepsilon_0),\mathcal{B}_2,\Omega]>
\sigma[d_2{\mathcal{L}}_2-{\mu},\mathcal{B}_2,\Omega]= 0,
\end{equation*}
that problem fits the case (\ref{item6.5}) above, for which we already know that there are no coexistence states, leading to a contradiction.
\par
Now, suppose that the conditions of the first part of Case (\ref{proofth6.2.itemII}) hold, i.e.,
\begin{equation*}
\sigma[d_1{\mathcal{L}}_1-{\lambda},\mathcal{B}_1,\Omega]\ge 0\quad\hbox{and}\quad \sigma[d_2{\mathcal{L}}_2-{\mu},\mathcal{B}_2,\Omega]= 0,
\end{equation*}
and \eqref{6.1} admits a coexistence state. Then, \eqref{6.4} also has a coexistence state for $\varepsilon_1=\varepsilon_2=-\varepsilon_0$ with sufficiently small $\varepsilon_0>0$.  However, for this choice,
\begin{equation*}
\sigma[d_1{\mathcal{L}}_1-(\lambda-\varepsilon_0),\mathcal{B}_1,\Omega]> 0\quad\hbox{and}\quad \sigma[d_2{\mathcal{L}}_2-({\mu}-\varepsilon_0),\mathcal{B}_2,\Omega]> 0,
\end{equation*}
and hence, \eqref{6.4} matches the situation (\ref{item6.4}) above for which there are no coexistence states. This contradiction ends the proof in this case also.
\par
Finally, suppose that we are under the assumptions of the first part of Case (\ref{proofth6.2.itemIII}), i.e.,
\begin{equation*}
\sigma[d_1{\mathcal{L}}_1-{\lambda},\mathcal{B}_1,\Omega]<0\quad\hbox{and}\quad
\sigma[d_2{\mathcal{L}}_2-{\mu}+c \theta_{\{d_1,\lambda,a\}},\mathcal{B}_2,\Omega]= 0,
\end{equation*}
and that \eqref{6.1} admits a coexistence state. Then, by the implicit function theorem, \eqref{6.4} also admits a coexistence state with $\varepsilon_1=0$ and $\varepsilon_2=-\varepsilon_0$ for
sufficiently small $\varepsilon_0>0$. Moreover,
\begin{equation*}
\sigma[d_2{\mathcal{L}}_2-({\mu}-\varepsilon_0)+c \theta_{\{d_1,\lambda,a\}},\mathcal{B}_2,\Omega]> 0,
\end{equation*}
which fits the situation  (\ref{item6.6}) provided $\e_0$ is chosen so that
\[
  \sigma[d_2{\mathcal{L}}_2-({\mu}-\varepsilon_0),\mathcal{B}_2,\Omega]<0.
\]
This is again impossible, as we already know that the problem cannot admit a coexistence state in any of those situations.
\par
The fact that the unique stable steady-state solution is a global attractor with respect to the positive solutions of the parabolic counterpart of \eqref{6.1} can be easily adapted from the proof of Theorem \ref{th5.3}. This ends the proof of Theorem \ref{th6.2}.



\end{document}